\documentclass[12pt,reqno]{amsart}
\usepackage{amsmath,a4wide}
\usepackage{xfrac}
\usepackage{stmaryrd,mathrsfs,bm,amsthm,mathtools,yfonts,amssymb,color}
\usepackage{xcolor}
\usepackage{mathtools}
\usepackage{enumerate}
\numberwithin{equation}{section}

\usepackage[margin=3cm]{geometry} 
\usepackage{graphicx}           
\usepackage{amsmath}             
\usepackage{amsfonts}            
\usepackage{amsthm}              
\usepackage{amssymb}
\usepackage{bbm}
\usepackage[activeacute, english]{babel}
\usepackage[utf8]{inputenc}
\usepackage{multicol}
\usepackage{mathrsfs}
\usepackage{tikz-cd}
\usepackage{wrapfig}
\usetikzlibrary{calc}
\usetikzlibrary{matrix}

\newtheorem{theorem}{Theorem}[section]
\newtheorem{lemma}[theorem]{Lemma}
\newtheorem{proposition}[theorem]{Proposition}
\newtheorem{corollary}[theorem]{Corollary}

\newtheorem{resultado}{Theorem}

\theoremstyle{definition}
\newtheorem{definition}[theorem]{Definition}

\usepackage{xcolor}

\newtheorem{remark}[theorem]{Remark}

\newtheorem{assumption}[theorem]{Assumption}

\newcommand{\RR}{\mathbb{R}}

\newcommand{\HH}{\mathcal{H}}

\newcommand{\dist}{\text{dist}}

\newcommand{\p}{\mathbf{p}}
\newcommand{\bB}{\mathbf{B}}
\def\a#1{\left\llbracket{#1}\right\rrbracket}

\newcommand\res{\mathop{\hbox{\vrule height 7pt width .3pt depth 0pt
\vrule height .3pt width 5pt depth 0pt}}\nolimits}

\title[Uniqueness of tangent cones at minimum density boundary points]{On the uniqueness of tangent cones to area minimizing currents at boundaries with arbitrary multiplicity}
\author{Ian Fleschler}

\begin{document}
\begin{abstract}
We consider an area minimizing current $T$ in a $C^2$ submanifold $\Sigma$  of $\RR^{m+n}$, with arbitrary integer boundary multiplicity $\partial T=Q\a{\Gamma}$ where $\Gamma$ is a $C^2$ submanifold of $\Sigma$. We show that at every density $Q/2$ boundary point the tangent cone to $T$ is unique and there is a power rate of convergence to the unique tangent cone.

In particular, if $\Gamma$ is a closed manifold which lies at the boundary of a uniformly convex set $\Omega$ and $\Sigma=\mathbb{R}^{m+n}$ then $T$ has a unique tangent cone at every boundary point.

As a structural consequence of the uniqueness of the tangent cone, we obtain a decomposition theorem which is the starting point of the boundary regularity theory we develop in another paper in collaboration with Reinaldo Resende. The regularity theory we obtain generalizes Allard's boundary regularity theorem to a higher multiplicity setting. \end{abstract}

\maketitle
\tableofcontents
\begin{section}{Introduction}
The present work, together with \cite{ianreinaldo2024regularity} and \cite{ian2024example}, provides a sharp answer to a long-standing open question posed in Allard’s 1969 PhD thesis \cite{AllPhD}, on the boundary regularity of area-minimizing \( m \)-currents in arbitrary dimension, codimension, and boundary multiplicity, under a convexity assumption. Our results sharply extend Allard’s celebrated 1975 boundary regularity theorem \cite{AllB} to the much more delicate and previously unexplored case of higher boundary multiplicity, in the minimizing setting. Before our work, even the existence of a single regular boundary point was not known in this context.

In this first paper of the series, we prove the uniqueness of the tangent cone to an area-minimizing current with smooth multiplicity \( Q \) boundary, at boundary points where the density is the lowest possible, that is, \( Q/2 \).

The problem of understanding higher multiplicity boundaries, in the precise form we study here, was first raised in Allard’s 1969 PhD thesis \cite{AllPhD}, and has remained open ever since. It was later highlighted in White’s famous list of open problems from the 1984 AMS Summer Institute in Geometric Measure Theory \cite[Problem~4.19]{GMT_prob}, and more recently in the ICM 2022 survey by De Lellis \cite{de2022regularity}.

The uniqueness of the tangent cone at density \( Q/2 \) boundary points, together with a decomposition theorem—both established in this paper—provide the structural foundation for the boundary regularity theory we develop in collaboration with Reinaldo Resende. The culmination of the theory appears in \cite{ianreinaldo2024regularity}, where we prove the \emph{sharp} \( \mathcal{H}^{m-3} \)-rectifiability of the boundary singular set at points of density \( Q/2 \), and show that the boundary regular set—without any density assumption—is both open and dense. Our \textit{sharp} theorems apply under a convex barrier hypothesis—a general setting previously considered by Allard in the multiplicity 1 case— since in this general setting all boundary points must have density \( Q/2 \).

The \emph{dimensional sharpness} of this regularity theory is confirmed in the companion article \cite{ian2024example}, where we construct a broad class of essential boundary singularities of density 1.
More specifically, we use the results and methods of this paper to construct a general family of essential boundary singularities for 3-dimensional area-minimizing currents in \( \mathbb{R}^5 \), with a multiplicity 2 boundary that can be taken to be real analytic. The existence of these singularities is established through an excess decay theorem at points of ``corner type,'' where smooth boundary surfaces meet along a smooth \( (m\!-\!2) \)-dimensional common interface. For this class of boundaries, we prove quantitative decay to a unique tangent cone at the interface.

We note that, beyond its significance for the development of regularity theory, the uniqueness of tangent cones for area-minimizing currents is an \textit{independent} and long-standing problem in its own right. In this work, we settle the challenging and \textit{intrinsically interesting} question of uniqueness of tangent cones in the higher multiplicity boundary setting under consideration. In contrast to the work of De Lellis–Nardulli–Steinbrüchel \cite{delellis2021uniqueness}, which establishes uniqueness in the two-dimensional case using the epiperimetric inequality, our result holds in all dimensions and relies instead on Simon’s \cite{simon1993cylindrical} method for proving uniqueness of cylindrical tangent cones, as developed in De Lellis–Minter–Skorobogatova \cite{de2023fineIII}. In fact, even in the two-dimensional setting, this paper gives a new proof of the uniqueness of tangent cones at minimum density points.

For a short survey, which announces the results in this article together with \cite{ianreinaldo2024regularity} and \cite{ian2024example}, we refer the reader to the proceedings article \cite{ian2024allard}.

\subsection{History}
In his 1993 paper \cite{simon1993cylindrical}, Simon investigated the uniqueness of tangent cones for a class of stationary varifolds with multiplicity 1. In this work, he developed what are now referred to as Simon's estimates, using the remainder of the monotonicity formula under the presence of points of suitable density, often called the "no holes" condition. These estimates show that the $L^2$ height of the varifold does not concentrate at the spine of the cone, thus reducing the task of proving uniqueness of the tangent cone to studying the energy decay for the associated linear problem.

Also in an interior regularity setting, Wickramasekera in \cite{WickramasekeraSt} extended Simon’s work to a higher multiplicity setting in codimension 1. He studied stationary varifolds that are stable in the regular part and have no classical singularities. More recently, De Lellis, Minter, and Skorobogatova \cite{de2023fineIII}, along with Krummel and Wickramasekera \cite{krummel2023analysisI,krummel2023analysisII}, independently used Simon's techniques at the interior for area minimizing currents with higher multiplicity and higher codimension.

The recent work \cite{DMSModp} studies area minimizing currents mod $p$, where the tangent cones under consideration include open books (composed of planes meeting at a common spine) and classical planes with multiplicity. In the present paper, we study, for the first time, the case of area minimizing currents with boundary taken with integer multiplicity $Q > 1$ for an integral current of arbitrary dimension and codimension. We mention that in this context, the "no holes" condition is naturally satisfied at every point, allowing us to establish the uniqueness of tangent cones at \textbf{every} boundary point of minimum density.

An important difference from \cite{DMSModp} is that, while their work also considers tangent cones that includes open books in the $m-1$ strata, the currents do not have boundary mod $p$. Their results instead address, in a sense, an interior regularity problem for a class of stationary varifolds. Furthermore, their cones satisfy a balancing condition: $\sum Q_i\nu_i=0$, where $\nu_i$ are the normals to the spine, and the multiplicities $Q_i$ fulfill $Q_i \leq p/2$ and $\sum Q_i = p$. For a related uniqueness of tangent cones result concerning a subclass of open books in codimension $1$, see \cite{Minter2024structure}.

The approaches to studying the uniqueness of the tangent cone differ substantially in the context of $2$d area minimizing currents since they rely on the epiperimetric inequality which is only available in $2$d. Indeed in \cite{White1983tangent}, White showed the uniqueness of tangent cones at interior points in this setting using an epiperimetric inequality. Hirsch and Marini later adapted White’s result to the case of \textit{boundary} uniqueness for $2$d area minimizing currents with boundary multiplicity 1 \cite{hirsch2019uniqueness}. More recently, De Lellis, Nardulli, and Steinbrüchel \cite{delellis2021uniqueness} extended these results to arbitrary boundary multiplicity.
\subsection{Results}This paper studies regularity properties at the boundary of an area-minimizing current at points of minimum density. We assume the current has a $C^2$ boundary and is area minminimizing in a $C^2$ manifold.
\begin{assumption}\label{A:general}
Let $m,n\geq 2$, $\bar{n}\geq 1$, $Q\geq 1$ be integers. Consider $\Sigma$ an $(m+\bar{n})$-manifold of class $C^2$ in $\RR^{m+n}$, $\Gamma$ a $C^2$ oriented $(m-1)$-submanifold of $\Sigma$ without boundary, and assume $0\in\Gamma$. Let $T$ be an integral $m$-dimensional area minimizing current in $\Sigma\cap\bB_1$ with boundary $\partial T \res \bB_1=Q\a{\Gamma \cap \bB_1}$.
\end{assumption}

The convex barrier condition in Euclidean space is the following:  
\begin{assumption}[Convex Barrier] \label{convexbarrier} Let $\Omega \subset \RR^{m+n}$ be a domain such that $\partial \Omega$ is a $C^2$ uniformly convex submanifold of $\RR^{m+n}$. We say that $\sum_{i=1}^N Q_i \a{\Gamma_i}$ has a convex barrier if $Q_i$ are positive integers and $\Gamma_i \subset \partial \Omega$ are disjoint $C^2$ closed oriented submanifolds of $\partial \Omega$. In this setting we consider $T$ an area minimizing current with $\partial T=\sum_i Q_i \a{\Gamma_i}.$ 
\end{assumption}

We will also consider Assumption \ref{a:convexbarrierlocal}, a more general version of the convex barrier condition, later in the paper. This condition is local in nature and includes the possibility of an ambient Riemannian manifold.

We start the program by obtaining a classification of the area minimizing cones. We define open books as follows:
\begin{definition}
A cone $C$ is an \textit{open book} if it can be written as
 $C=\sum_{i=1}^N Q_i\a{H_i}$ where $H_i$ are distinct half-planes with  $\partial\a{H_i}=\a{V}$
where $V$ is an $m-1$ dimensional plane, referred to as the spine of the cone, and $Q_i$ are positive integers.  
\end{definition}

\begin{resultado}[Classification of the tangent cones]\label{thm:classification}Let $C$ be an $m$-dimensional oriented area minimizing cone in $\RR^{m+n}$ with $\partial \a{C}=Q \a{\RR^{m-1}\times \left\{0\right\}}$.
Then $\Theta(C,0) \geq Q/2$. The equality case holds if and only if $C$ is an open book. Moreover, there is a density jump: if $C$ is not an open book then $\Theta(C,0) \geq Q/2 +\varepsilon(m,n)$ where $\varepsilon(m,n)>0$.
\end{resultado}

Under Assumption \ref{A:general},  a monotonicity formula for the density holds at boundary points (Theorem \ref{t:monotonicityformula}). As a consequence of the classification of Theorem \ref{thm:classification}, the boundary density must be at least $Q/2$ and equality only holds when every tangent cone is an open book. The density jump mentioned is only obtained at the end of the program as a consequence of the techniques for the main theorem rather than at the start.

Moreover, if $T$ is an area minimizing current with $\partial T=\sum_{i=1}^N Q_i \a{\Gamma_i}$ which satisfies the convex barrier Assumption \ref{convexbarrier}, then for any $p \in \Gamma_i$, $\Theta(T,p)=Q_i/2$, and thus it belongs to the class of minimum density points, which we are studying.

The main theorem of the work is the following:
\begin{resultado}[Uniqueness of the tangent cone] \label{t:uniquenessintroduction}
Let $T, \Gamma, \Sigma$ be as in Assumption \ref{A:general}. Further assume that 
$\Theta(T,0)=Q/2$. Then, the tangent cone at $0$ is an open book and it is unique, meaning it does not depend on the blow up sequence. Moreover, the current $T$ converges to the unique tangent cone with a power rate.
\end{resultado}
The previous statement in dimension $2$ was proved in \cite{delellis2021uniqueness}. We consider the convex barrier assumption considered in \cite{AllPhD,AllB} and obtain the following structural consequence of the uniqueness of the tangent cone:
\begin{resultado}
Let $T$ be as in either Assumption \ref{convexbarrier} (convex barrier) or Assumption \ref{a:convexbarrierlocal} (Local convex barrier), then every tangent cone to $T$ at a boundary point $q$ is an open book. Moreover, we have:
\begin{enumerate}[a)]
    \item if $q \in \Gamma_i$ in the first case, then $\Theta(T,q)=Q_i/2$,

    \item if $q \in \Gamma$ in the second case, then $\Theta(T,q)=Q/2$.
\end{enumerate}
Thus there is a unique tangent cone for $T$ in both settings, convex barrier \ref{convexbarrier} and local convex barrier \ref{convexbarrier2}, at every boundary point. Moreover, the current $T$ has a power rate of convergence to the unique tangent cone.
\end{resultado}
In order to reduce the above theorem to Theorem \ref{t:uniquenessintroduction}, we need to classify tangent cones with Allard's convex barrier assumption. In \ref{t:ClassificationTgtConesConvexBarrier}, we classify those tangent cones relying on Lemma 5.1 \cite{AllB}.

An essential consequence of the uniqueness of tangent cones, which will enable us to run the regularity theory in \cite{ianreinaldo2024regularity}, is a decomposition theorem. This decomposition theorem allows us to reduce the study of boundary singularities to the study of \textbf{flat} boundary singularities. In dimension $2$, the decomposition was proved by \cite{delellis2021uniqueness,delellis2021allardtype}.
\begin{resultado}[Decomposition theorem]\label{t:decomposition} Let $T, \Gamma, \Sigma$  be as in Assumption \ref{A:general}. Further assume that $\Theta(T,0)=Q/2$. Let $C=\sum_{i=1}^N Q_i \a{H_i}$ be the unique tangent cone to $T$ at $0$, where the half-planes $H_i$ in the representation are assumed to be distinct. Then there exists $\rho>0$ and area minimizing currents $T_1,T_2,...,T_N$ in $\bB_{\rho}$ such that
\begin{equation*}
T \res \bB_{\rho}=\sum_{i=1}^N T_i 
\end{equation*}
where the supports of $T_i$ only intersect at $\Gamma$, $\partial T_i \res \bB_{\rho}=Q_i \a{\Gamma}$, and the (unique) tangent cone at $0$ of $T_i$ is $Q_i \a{H_i}$.
\end{resultado}
We will work with multivalued functions, understood as in \cite{Alm} and \cite{DS1}.
Another consequence of the uniqueness of tangent cones that plays a key role in constructing the example is the Hölder continuity of the multivalued normal map.
\begin{resultado} Let $T, \Gamma,\Sigma$ be as in Assumption \ref{A:general}. There is a well-defined multivalued normal map induced by the tangent cone, which is defined in \ref{d:multivaluednormalmap}. The normal map is a locally Hölder map in $U:=\left\{x \in \Gamma: \Theta(T,x)=Q/2\right\}$, which forms a relative open set of $\Gamma$. The set $U$ is relatively open by upper semicontinuity of the density, since there exists $\theta(Q,m,n)>0$ such that $\Theta(T,x)<Q/2+\theta$ implies $\Theta(T,x)=Q/2$ .
\end{resultado}

We also comment that there are analogous theorems for the linear problem. The general linear problem is defined as: \begin{assumption}\label{a:linearproblemintro}
Given a Dir-minimizing function $u \in W^{1,2}(\Omega,\mathcal{A}_Q(\RR^n))$. The boundary linear problem consists of understanding functions $u$ such that $u \res \partial \Omega \cap \bB_1=Q\a{0}$.
\end{assumption}
We denote the frequency function by $I$ in the theorem below using a suitably smooth\-ened distance function. For the precise definition of the frequency function we refer to Section 10.
\begin{resultado}
Assume that $\Omega$ is a $C^2$ domain and $u$ as in Assumption \ref{a:linearproblemintro}. 
\begin{itemize}
    \item For all $q \in \partial \Omega$ the frequency can be defined in a neighbourhood of $q$ and is almost monotone. This means, \begin{equation*}
        e^{C\mathbf{A}_{\partial \Omega}r}I(r)
    \end{equation*} is monotone in $r$ for appropriate constant $C$, where $\mathbf{A}_{\partial \Omega}$ denotes the $C^0$ norm of the second fundamental form of $\Omega$.
    \item At the limit, the frequency is independent of the choice of a distance function and $\forall q \in \partial \Omega$, $I(q) \geq 1$. 
    \item Either $I(q)=1$ or $I(q) \geq 1+\alpha(Q,m,n)$ for a positive constant $\alpha(Q,m,n)$. If $I(q)=1$ then the blowups at $q$ are unique and linear, in the sense they are of the form
    \begin{equation*}
    f(x)=\sum_{i=1}^N Q_i\a{v_i(x\cdot \boldsymbol{\eta}_{\partial \Omega,q})}
    \end{equation*}
    with $\boldsymbol{\eta}_{\partial \Omega,q}$ the normal to $\partial \Omega$ at $q$,
    and unique. The uniqueness is quantitative and thus there is a well-defined multivalued normal derivative at the boundary which is Hölder continuous.
\end{itemize}
\end{resultado}

\subsection{Basic definitions}\label{S:preliminaries}
Throughout the paper, we use the notation '$\lesssim$' to denote inequality up to a multiplicative constant, which may depend only on $m$, $n$, $\overline{n}$, and $Q$. We also use $C(\alpha_1,...,\alpha_k)$ to denote a constant which depends on $\alpha_1,...,\alpha_k$.

We define $\bB_{r}(p) := \{x \in \RR^{m+n} : |x - p| < r\}$ and typically use $V := \RR^{m-1} \times \{0\} \subset \RR^{m+n}$ as the spine of our cones. For a plane $\pi$, let $\mathbf{p}_{\pi}$ denote the projection onto $\pi$. For any $\varepsilon > 0$, we use $B_{\varepsilon}(V)$ to represent the $\varepsilon$-neighborhood of $V$, that is, $B_{\varepsilon}(V) = \{x \in \RR^{m+n} : \dist(x, V) < \varepsilon\}$. 

We assume each $k$-dimensional linear subspace \(\pi\) of $\RR^{m+n}$ is oriented by a \(k\)-vector \(\vec{\pi} = v_1 \wedge \cdots \wedge v_k\), where \((v_i)_{i=1}^k\) forms an orthonormal basis for $\pi$. Additionally, we adopt the notation \(\left|\pi_2 - \pi_1\right|\) for \(\left|\vec{\pi}_2 - \vec{\pi}_1\right|\), where \(|\cdot|\) represents the norm tied to the standard inner product of $k$-vectors. The $k$-dimensional Hausdorff measure on $\RR^{m+n}$ is written as $\HH^k$. We denote $\mathcal{A}_Q(\RR^{m+n})$  the $Q$-tuples in $\RR^{m+n}$ for some integer $Q \geq 1$. We denote by $\mathcal{G}(p_1,p_2)$, the Wasserstein distance between a pair of points $p_1, p_2$ on the space $\mathcal{A}_Q(\RR^{m+n})$. For more on multi-valued functions, see \cite{DS1}. We denote by $\mathbf{A}_M$ the $C^0$-norm of the second fundamental form of a smooth submanifold $M$ of $\RR^{m+n}$.

For further basic definitions and standard notations, we refer to \cite{Fed} and the survey \cite{delellis2015size}.

We denote $\textbf{G}_f$ the current induced by the graph of a function $f$ over it's domain.

We will use $\mathbf{E}$ to denote the $L^2$ height excess, and local versions which we introduce later. The $L^2$ excess relative to an open book $C$ is defined as
\begin{equation*}
    \mathbf{E}(T, C, \bB_r(p)) := \frac{1}{r^{m+2}} \int_{\bB_r(p)} \dist(x, C)^2 d|T|(x).
\end{equation*}
We will use $\mathbb{E}$ for a suitable measure-theoretic excess, which will be defined later in equation \eqref{e:strongexcess}. Finally, we introduce the circular projection $\pi_{\circ} : \RR^{m+n} \to \RR^{m}$, defined as    $\pi_{\circ}(x, y) = (x, |y|)$ for $(x, y) \in \RR^{m-1} \times \RR^{n+1}$.
\begin{definition}\label{D:AreaMin}
An $m$-dimensional integral current $T$ in $U$ with $\textup{spt}(T)\subset\Sigma$ is \emph{area minimizing in $\Sigma\cap U$}, if it holds that $\mathbf{M}(T) \leq \mathbf{M}(T+\partial \tilde{T})$, for all $\tilde{T}$ $m+1$-dimensional integral current in $U$ with $\textup{spt}(\tilde{T})\subset\Sigma\cap U$. 
\end{definition}
\subsubsection{Cones}
We remind the definition of an open book and introduce their angle:
\begin{definition}
A cone $C$ is an \textit{open book} if it can be written as
 $C=\sum_{i=1}^N Q_i\a{H_i}$ where $H_i$ are distinct half-planes with  $\partial\a{H_i}=\a{V}$
where $V$ is an $m-1$ dimensional plane, referred to as the spine of the cone. We define the minimum angle of an open book as
\begin{equation*}
\alpha(C):=\min_{i \neq j}
\langle H_i,H_j \rangle
\end{equation*}
where $\alpha(C):=1$ if $C=Q\a{H}$ for $H$ a half-plane.
\end{definition}

\begin{definition}
A cone $C$ is a \textit{two-sided flat cone} if 
\begin{equation*}
 C=(Q/2+Q^*)\a{H}-(Q^*)\a{-H} 
\end{equation*} for $Q$ and $Q^*$ integers and $H$ a half plane.
\end{definition}
We introduce a distance between open books with the same spine. 
\begin{definition}[Distance Open Books]
Given a pair of open books $C^j$ for $j=0,1$ which are of the form
\begin{equation*}
C^j:=\sum_i Q_i^j\a{H_i^j}
\end{equation*}
 where $\partial \a{H_i^j}=\a{V}$. We take $\eta_i^j$ to be the unit normal to $V$ that determines $H_i$.
We define $\eta^j \in \mathcal{A}_Q(V^{\perp})$ as $\eta^j= \sum Q_i\a{\eta_i^j}$.
We define 
\begin{equation}\label{eq:distopenbooks}
\mathcal{G}(C^1,C^2):=\mathcal{G}(\eta^1,\eta^2).
\end{equation}
\end{definition}
\subsubsection{Tangent cones.} For $x\in\RR^{m+n}$ and $r>0$, define $\iota_{x, r}(y):= \frac{y-x}{r}$. For any integral current $T$, we define the rescaled current $T$ at $x$ at scale $r$ as ${\iota_{x, r}}_{\sharp} T=:T_{x,r}$ and $T_r =: {\iota_{0,r}}_{\sharp} T$. The current $T_{x}$ is called a blow-up of $T$ at $x$, if there exists a sequence of radii $r_j\to0$ such that $T_{x,r_j}\to T_x$ in the weak topology. By  \cite{Theodora}, it follows from the monotonicity formula that if $T$ is area minimizing in $\Sigma$ and $\partial T = Q\a{\Gamma}$ for $\Gamma\in C^{1,\kappa}$, every convergent subsequence of $T_{p,r_j}$ converges to an oriented cone about $p$ for any $p \in \Gamma$ and any $r_j\downarrow 0$. An oriented cone about $p$ is understood as a current $C$ such that $C_{p,r} = C$ for any $r>0$. See Theorem  \ref{t:monotonicityformula} for the precise statements of the monotonicity formula and first variation formula for an area minimizing current $T$ in $\Sigma$ with $\partial T=Q\a{\Gamma}$ where $\Sigma$ and $\Gamma$ are $C^2$ submanifolds. 

\subsection{Structure of the paper}
\begin{itemize} 
\item In Section 2, we define a measure-theoretic/stronger excess, which allows us to track not only the closeness to an open book but also the multiplicity to which it is close. We also prove the monotonicity formula with the corresponding error terms suitable for our use in the later estimates. 
\item In Section 3, we show that cones of density $Q/2$ with a flat boundary with multiplicity $Q$ are open books. We provide two different arguments.

\item In Section 4, we examine the convex barrier property and the local convex barrier property. We show that they are boundary points of the type we consider, i.e., of density $Q/2$, whose tangent cones are open books.  \item In Section 5, we perform the main technical reductions of the theorems in this paper, reducing them to an excess decay lemma, which the remainder of the paper focuses on proving. 
\item In Section 6, we construct a Lipschitz approximation and show a reverse excess estimate, which allows us to control the measure-theoretic excess introduced in Section 2 by the $L^2$ excess. 
\item In Section 7, we prove Simon’s non-concentration estimates and their corollaries in our setting.
\item In Section 8, we study the linear problem. We show that homogeneous frequency 1 Dir-minimizers are linear, and that local Dir-minimizers away from the boundary are global Dir-minimizers. We provide a decay lemma to a linear average that is suitable to run the excess decay lemma.
\item In Section 9, we prove the excess decay lemma stated in Section 5, thus concluding the uniqueness of the tangent cone. 
\item In Section 10, we discuss important consequences of the theory developed in this paper for the linear problem, which mirror the properties we have discussed throughout the paper for the nonlinear problem. 
\end{itemize}
\subsection{Acknowledgments} 
The author acknowledges the support of the National Science Foundation through the grant FRG-1854147.

The author is extremely grateful to his advisor, Professor Camillo de Lellis, for introducing him to the subject by suggesting this problem and for his outstanding mentorship throughout. The author is also grateful to Paul Minter and Anna Skorobogatova for answering many relevant questions related to this project, and wishes to express gratitude to Reinaldo Resende for providing detailed feedback on the manuscript. The author wishes to thank Antoine Song for a suggestion that led to a simplification of the alternative proof of the classification of tangent cones.
\end{section}
\begin{section}{Preliminaries}
\begin{subsection}{A measure theoretic excess}
We introduce a stronger version of the $L^2$-excess using a variant of the Wasserstein $L^2$ distance, which we recall below.

Given two measures $\mu^1$ and $\mu^2$, an admissible transport plan $\sigma$ is a measure on $\mathbb{R}^{m+n} \times \mathbb{R}^{m+n}$ satisfying $(\pi_1)_{\#} \sigma = \mu^1$ and $(\pi_2)_{\#} \sigma = \mu^2$ where $\pi_1$ and $\pi_2$ are the projections in the first and second component, respectively.

The Wasserstein $L^2$ distance between two measures $\mu^1$ and $\mu^2$ is defined as
\begin{equation}
    W_2(\mu^1,\mu^2)^2 := \inf \left\{ \int d(x,y)^2 \, d\sigma(x,y) : (\pi_1)_{\#}\sigma = \mu^1 \: \text{and} \: (\pi_2)_{\#}\sigma = \mu^2\right\}
\end{equation}
where we set $W_2(\mu^1,\mu^2)=\infty$ if $\mu^1(\RR^{m+n})\neq \mu^2(\RR^{m+n}).$

Given a fixed $m-1$ dimensional plane $V$, we define the distance $d_V$ between two measures $\mu^1$ and $\mu^2$ as follows:

\begin{equation}
d_V(\mu^1, \mu^2) := \inf \left\{ 
W_2(\mu_1^1, \mu_1^2)^2 
+ \sum_{i=1}^2 \int \textup{dist}(x,V)^2 \, d|\mu_2^i|  
\ \middle|\ \mu_1^i + \mu_2^i = \mu^i, i=1,2.
\right\}.
\end{equation}

In this definition, $\mu^1$, $\mu^2$, $\mu_1^1$, and $\mu_1^2$ are Radon measures, while $\mu_2^1$ and $\mu_2^2$ are signed measures. This allows us to add or remove a part of the measure $\mu^1$ to transport it to $\mu^2$, with a cost proportional to the weighted total variation of the leftover measures.

In the rest of the paper we will drop $V$ from the notation and just use $d$, since it is more convenient and the plane $V$ will be clear from the context (and fixed in most arguments).

If $d(\mu^1,\mu^2)=0$ then $\textup{spt}(|\mu^1-\mu^2|) \subset V$. We remark that in the cases of interest the measures $\pi_1$ and $\pi_2$ considered are always absolutely continuous with respect to $\HH^{m}$ and thus $d(\mu^1,\mu^2)=0$ implies $\mu^1=\mu^2$ in that case.

Let $\phi$ be a smooth, canonical bump function supported in $\bB_1(0)$, constantly equal to 1 in $\bB_{1/2}(0)$, and satisfying $\phi(x) \lesssim \operatorname{dist}(x, \partial \bB_1)^2$. We define $\phi_r := \phi(\cdot/r)$.

The stronger excess relative to an open book $C=\sum_{i=1}^N Q_i\a{H_i}$ with $\partial H_i=V$ as
\begin{equation}\label{e:strongexcess}
  \mathbb{E}(T,C,\bB_r) := r^{-(m+2)} d(\phi_r d|T|, \phi_r d|C|).
\end{equation}

The bump function is introduced for technical reasons: even if the current were a classical smooth graph over the cone, the map given by the graphical structure might extend beyond the ball $\bB_r$ near the boundary.

\begin{remark}[Strong excess versus classical $L^2$ and reverse $L^2$]
The classical $L^2$ excess and the reverse $L^2$ excess are both controlled by this stronger excess. This follows immediately from the admissibility of the transport plan and Fubini's theorem.
Precisely we have:
\begin{equation}
\frac{1}{r^{m+2}}\int\phi_r(x)\dist(x,C)^2 d|\!|T|\!|(x) + 
\frac{1}{r^{m+2}}\int\phi_r(x)\dist(x,\textup{spt}(T))^2 d|\!|C|\!|(x)\leq \mathbb{E}(T,C,\bB_r).
\end{equation}
\end{remark}

To motivate the definition of the excess, we show that it is bounded in terms of standard quantities that we can control in practice.
More precisely we have the following remark which will give the desired motivation. 
\begin{remark}[Lipschitz Approximation and Strong Excess]\label{rmk:strongexcessgraphicalbound}
Let $C=\sum_{i=1}^N Q_i\a{H_i}$ and let $u_i \in \mathcal{A}_{Q_i}(H_i \cap \bB_2)$ be multi-valued Lipschitz graphs over $C$ of small enough Lipschitz constant. Let $K_i$ be a subset of $H_i$ and $T_i=\mathbf{G}_{u_i \res K_i}$  . Assume that $\textup{spt}(T_i) \cap \bB_1 \subseteq \textup{spt}(T) \cap \bB_2$ and that $\forall x \in \textup{spt}(T_i) \cap \bB_1$ $\Theta(T_i,x) \leq \Theta(T,x)$.

Then
\begin{align}
\mathbb{E}(T,C,\bB_1) \lesssim & \max_{i} |u_i|_{L^\infty(\bB_1)}^3 + \sum_{i=1}^N\int_{\bB_1 \cap \left(H_i \setminus K_i \right)}\dist(x,V)^2d|\!|H_i|\!|(x) \\+&\int_{\bB_1}\dist(x,V)^2 d|\!|T - \sum_{i=1}^N T_i|\!|(x)  +  \sum_{i=1}^N\int_{\bB_1} \dist(x,C)^2 d|\!|T_i|\!|(x)\\ +& \sum_{i=1}^N\int_{K_i}\dist(x,V)^2 |Du_i|^2.
\end{align}
\end{remark}
\begin{proof}
Let $A_i$ be the subset of $K_i$ such that $|x+f_i^j(x)|>1$ for some $1 \leq j \leq Q_i$ (i.e. the set where the Lipschitz graph escapes the ball).

We must have that 
\begin{align}
\mathbb{E}(T,C,\bB_1) \lesssim &\sum_{i=1}^N \int_{\bB_1 \cap \left((H_i \setminus K_i) \cup A_i\right)}\dist(x,V)^2d|\!|H_i|\!|(x) +  \int_{\bB_1}\dist(x,V)^2 d|\!|T-\sum_{i=1}^k T_i|\!|(x) \\+&
W_2(|\!|T_i|\!|, ({\pi_{H_i}})_{\#}|\!|T_i|\!|)^2 + \int_{K_i} \dist(x,V)^2 \left(d ({\pi_{H_i}})_{\#}(|T_i|)-d|\!|H_i|\!|(x)\right)
 \end{align}
This inequality comes from the pairing of $|\!|T_i|\!|$ with $(\pi_{H_i})_{\#}|\!|T_i|\!|$ and getting as error terms the measures $|\!|T-\sum T_i|\!|$, $\sum_{i=1}^N d|\!|\bB_1 \setminus \left(A_i \cup (H_i \setminus K_i) \right)|\!|$ and $(\pi_{H_i})_{\#}|\!|T_i|\!|-d|\!|K_i \setminus A_i|\!|$.

We must have that $A_i$ is contained in an annulus whose outer boundary is $\partial \bB_1$ and its width is bounded by $\left\|u_i\right\|_{L^\infty}$, which gives us the bound $\left\|u_i\right\|_{L^\infty}^3$. The linear contribution comes from the area of the region and the quadratic contribution comes from the bump function.

We observe that

\begin{equation}
W_2(|T_i|, ({\pi_{H_i}})_{\#}|T_i|)^2 \leq \int_{K_i} \dist(x,C)^2 d|T_i|.
\end{equation}
Finally we notice that 
\begin{equation}
\int_{K_i} \dist(x,V)^2 d(\pi_{H_i})_{\#}(T_i)-d|H_i| \lesssim \int_{K_i} \dist(x,V)^2 |Du_i|^2 .
\end{equation}
\end{proof}
\begin{remark}
    The question whether we can control excess from one scale to a smaller one is technically non trivial and requires the construction of the Lipschitz approximation and the smallness of the excess. We will also need to be able to switch center. 
    
    These issues are subtle because the transport plan for a given scale might not be suitable for a smaller scale. See Lemma \ref{l:excessboundscales}.
\end{remark}
\end{subsection}
\begin{subsection}{Monotonicity formula}

Given $T, \Gamma, \Sigma$ as in Assumption \ref{A:general}.
We define, for some $C_0(m,n,\overline{n})$ a suitable large constant
\begin{equation}
\Theta_{i}(T,p,r):=\textup{exp}(C_0\mathbf{A}_{\Sigma}^2r^2)\frac{|\!|T|\!|(\bB_r(p))}{\omega_mr^m}
\end{equation}
and 
\begin{equation}
\Theta_{b}(T,p,r):=\textup{exp}(C_0\left(\mathbf{A}_{\Sigma}+\mathbf{A}_{\Gamma}\right)r)\frac{|\!|T|\!|(\bB_r(p))}{\omega_mr^m}.
\end{equation}

\begin{theorem}[Monotonicity formula]\label{t:monotonicityformula}
Let $T, \Gamma, \Sigma$ be as in Assumption \ref{A:general}. 
\begin{itemize}
\item If $p \in \bB_1 \setminus \Gamma$ then $r \rightarrow \Theta_{i}(T,p,r)$ is monotone on the interval 
$(0,\min\left\{\dist(p,\Gamma),1-|p|\right\})$.
\item If $p \in \bB_1 \cap \Gamma$, then $r \rightarrow \Theta_{b}(T,p,r)$ is monotone on $(0,1-|p|)$.
\end{itemize}
Thus the density exists at every point. Moreover, the restrictions of the map $p \rightarrow \Theta(T,p)$ to $\bB_1 \cap \Gamma$ and $\bB_1 \setminus \Gamma$ are upper semicontinous. 

Additionally, this implies that blowup limits are cones, both at the interior and at the boundary.

If $X \in C^1_{c}(\bB_1,\RR^n)$, then
\begin{equation}
\delta T(X)=- \int_{\bB_1} X \cdot \overrightarrow{H}_{T}(x)d|\!|T|\!|(x)+Q \int_{\Gamma \cap \bB_1} X \cdot \overrightarrow {\eta}(x) d\HH^{m-1}(x)
\end{equation}
where $\overrightarrow{H}_{T}=\sum_{i=1}^m A_{\Sigma}(e_i,e_i)$ where $\overrightarrow{T}=\wedge_{i=1}^m e_i$ is an orthonormal basis of $\overrightarrow{T}$. The vector $\overrightarrow{\eta}$ is defined $\HH^{m-1}$ almost everywhere as:
\begin{itemize}
\item For every tangent cone $C$ at $p \in \Gamma$ which is an open book $C=\sum_{i=1}^N Q_i\a{H_i}$, 
\begin{equation}
\overrightarrow{\eta}=\sum_{i=1}^N \frac{Q_i \overrightarrow{\eta_i}}{Q}
\end{equation}
where $\overrightarrow{\eta_i}$ is the normal to $T_p(\Gamma)$ which determines the half plane $H_i$. The vector $\overrightarrow{\eta}$ is independent of  the open book $C$ at $p$, but the open book might depend on the blowup sequence.
\item For every tangent cone $C=(Q/2+Q^*)\a{H}-(Q^*)\a{-H}$ at $p \in \Gamma$ which is a two-sided flat cone \begin{equation*}
    \overrightarrow{\eta}=\eta_{H}
\end{equation*}
where $\eta_H$ is the normal to $T_p(\Gamma)$ which determines the half plane $H$.
\end{itemize}
\end{theorem}
\begin{corollary}[Monotonicity formula with error terms]\label{c:monotonerror}
Under the hypothesis of Theorem \ref{t:monotonicityformula} for every $p \in \bB_1$ and $0<s<r<1-|p|$,
\begin{align}
    r^{-m}|\!|T|\!|(\bB_r(p))-s^{-m}|\!|T|\!|(\bB_s(p))= \int_{\bB_r(p) \setminus \bB_s(p)}\frac{|\left(x-p\right)^{\perp}|^2}{|x-p|^{m+2}}d|\!|T|\!|(x) \notag \\
    +\int_{s}^{r} \rho ^{-m-1} \left[\int_{\bB_{\rho}(p)}(x-p)^{\perp} \cdot \overrightarrow{H}_{T}(x)d|\!|T|\!|(x)+ Q \int_{\Gamma \cap \bB_{\rho}(p)}(x-p) \cdot \overrightarrow{\eta}d\HH^{m-1}\right]d\rho
\end{align}
\end{corollary}
\begin{proof}[Proof of Theorem \ref{t:monotonicityformula} and Corollary \ref{c:monotonerror}]
The proof follows the same steps as Theorem 3.2 in \cite{DDHM}. However, the key difference is that we allow for arbitrary boundary multiplicity, whereas \cite{DDHM} considers only boundaries with multiplicity $1$. The first two statements are a consequence of Allard interior and boundary monotonicity formulas for stationary varifolds. 

The first variation formula is established in the same way as in \cite{DDHM}, which we briefly sketch. We first obtain the expression  $\delta T= \overrightarrow{H}_{T} |\!|T|\!| + \delta T_s$ where $\delta T_s$ is a singular vector valued measure supported in $\Gamma$. This leads to the identity
\begin{equation}\label{eq:firstvariationformulaproof}
\delta T(x)=-\int_{\bB_1} X \cdot \overrightarrow{H}_{T}(x)d|\!|T|\!|(x)+Q \int_{\Gamma \cap \bB_1} X \cdot \overrightarrow{\eta}(x) d||
\delta T_s|\!|.
\end{equation}

As a consequence of the monotonicity formula we get that $|\!|\delta T_s|\!|$ is an absolutely continuous measure with respect to $\HH^{m-1}$. The density of the singular measure and vector $\overrightarrow{\eta}$ are determined by the $\HH^{m-1}$ normal and density at the blowups.

In this setting, there exists a stratification of the tangent cones at the boundary analogous to all other stratification theorems from \cite{White97}. See \cite[Section 3.2]{DDHM} for more details. In this setting, we obtain $\HH^{m-1}$-almost everywhere there exists a tangent cone with an $m-1$ dimensional spine. By the classification of $1$d area minimizing cones, we conclude that $\HH^{m-1}$ almost everywhere, there exists a tangent cone which is either an open book or a flat two sided cone.

We consider the first variation formula in the case that $C$ is an open book or $C$ is two-sided flat.

If $C=\sum_{i=1}^N Q_i\a{H_i}$ where $V$ is the common spine and $\overrightarrow{\eta_i}$ is the normal vector to the spine that determines $H_i$, then
\begin{equation}
\delta C(X)= \sum_{i=1}^N Q_i\int_{V} \left(X \cdot \overrightarrow{\eta_i}\right) d\HH^{m-1}.
\end{equation}
If $C=(Q+Q^*)\a{H}-Q^*\a{H^-}$ then 
\begin{equation}
\delta C(X)= Q\int_{V} \left(X \cdot \eta_{H}\right) d\HH^{m-1}.
\end{equation}
Where $\eta_{H}$ is the normal to $V$ which determines $H$.

In the first case, we define $\overrightarrow{\eta}=\sum_{i=1}^N\frac{Q_i\eta}{Q}$. In the second case, $\overrightarrow{\eta}=\eta_{H}$.

By combining \eqref{eq:firstvariationformulaproof} with the absolute continuity of $|\!|\delta T_s |\!|$ with respect to $\HH^{m-1}$ we obtain $\HH^{m-1}$ almost everywhere by Lebesgue differentiation theorem that $\overrightarrow{\eta}(p)$ agrees with the vector $\overrightarrow{\eta}$ for the tangent cones $C$ in the $m-1$- strata.

From the above discussion, we conclude Theorem \ref{t:monotonicityformula}. The proof of Corollary \ref{c:monotonerror} as a consequence of Theorem \ref{t:monotonicityformula} is analogous to \cite{DDHM}.
\end{proof}
\end{subsection}
\end{section}

\begin{section}{Classification of area minimizing cones of density $Q/2$ with boundary}
In this section, we prove the following classification of area minimizing cones of density $Q/2$ with boundary:

\begin{theorem}[Classification of area minimizing cones with density $Q/2$]\label{t:ClassificationTgtConesNConvex}

Let $C$ be an $m$-dimensional area minimizing cone in $\RR^{m+n}$. Assume that $\partial C=Q\a{\RR^{m-1} \times \{0\}}$ for some positive integer $Q$.
Then $\Theta(C,0) \geq Q/2$. Moreover, equality holds if and only if $C$ is an open book.
\end{theorem}

The first proof follows an argument inspired by Federer's dimension reduction, while the second proof uses structural information gained from the projection  $\pi_{\circ}: \RR^{m+n-1} \rightarrow \RR^{m-1} \times [0,\infty)$ defined as \begin{equation*}
\pi_{\circ}(x,y)=(x,|y|) \; \; \textup{where} \; (x,y) \in \left( \RR^{m-1} \times  \RR^{n+1} \right).
\end{equation*}
This projection is well adapted to the this boundary setting, ${\pi_{\circ}}_{\#}C=Q\a{\RR^{m-1}\times \RR^{+}}$ for any area minimizing cone $C$ with $\partial C=Q\a{\RR^{m-1}\times \left\{0\right\}}$.

\begin{proof}[Proof of Theorem \ref{t:ClassificationTgtConesNConvex}]
Let $C$ be an area minimizing cone with $\partial C=Q \a{\RR^{m-1} \times \left\{ 0\right\}}$. We denote $V=\RR^{m-1} \times \left\{ 0\right\}$. 

By Theorem \ref{t:monotonicityformula}, the classical monotonicity formula holds without error terms for any area minimizing current $T$ on $\RR^{m+n}$ with $\partial T=Q\a{V}$. Indeed, in the formulas of Theorem \ref{t:monotonicityformula} we have $\mathbf{A}_{V}=\mathbf{A}_{\Sigma}=0.$ Alternatively, if we consider $T$ as a varifold and reflect through $V$ we obtain a stationary varifold, which has twice the mass at every ball centered at $V$, for which the classical monotonicity formula for stationary varifolds holds. Moreover we also have, as usual, that if $p \in V$ and 
\begin{equation}
|\!|T|\!|(\bB_r(p))=\Theta_{b}(T,p)
\end{equation}
 for all $r \leq r_0$, then $T$ is a cone at $p$.
 
To establish $\Theta(C,0) \geq Q/2$, assume by contradiction that $\Theta(C,0)=\Theta<Q/2$. Since $C$ is a cone then $\Theta(C,p)=\Theta(C,rp)$ for every $p \in V$ such that $p \neq 0$ and $r>0$. By upper semi-continuity of the density for every $p \in V$ we must have that $\Theta(C,p) \leq \Theta(C,0).$ Consequently, any blowup at $p$ results in another area minimizing cone $C_1$ with density strictly lower than $\Theta$. However $C_1$ must be a product of a cone $C_1' \times \a{L}$ 
where up to rotation $\partial C_1'= \a{\RR^{m-2} \times \left \{ 0\right\}}$.

By repeating the argument, we are lowering the density but increasing the number of symmetries. We obtain an area minimizing cone $C_{m-1}$ with 
\begin{equation}
    C_{m-1}=\a{\RR^{m-1}}  \times C_{m-1}'
\end{equation} where $C_{m-1}'$ is a $1$d area minimizing cone with boundary $Q\a{0}$. 

The only options for such cones are a union of half-lines whose multiplicities sum up to $Q$. The cone $C_{m-1}$ must have density equal to $Q/2$, which is a contradiction since the density is at most $\Theta$ which is strictly lower than $Q/2$.

We note that if $C$ is an open book then trivially $\Theta(C,0)=Q/2$. Thus it remains to show that if $\Theta(C,0)=Q/2$ then $C$ is an open book.

We know that for every $p \in V$ 
\begin{equation}
\lim_{r \rightarrow \infty} \Theta(C,p,r)= \lim_{r \rightarrow \infty} \frac{ |C|(\bB_r(p))}{\omega_m r^m} \leq \lim_{r \rightarrow \infty}\frac{(r+|p|)^m}{r^m}\frac{|C|(\bB_{r+|p|})}{\omega_m (r+|p|)^m}=Q/2.
\end{equation}

This implies that we must have the equality in the monotonicity formula centered at every $p \in V$, which implies that $C$ is independent of the directions of $V$ and thus $C$ comes from a $1$d area minimizing cone. As we have seen, this implies $C$ is an open book.
\end{proof}
\begin{proof}[Alternative Proof of Theorem \ref{t:ClassificationTgtConesNConvex}]

Define $\pi_{\circ}: \RR^{m+n-1} \rightarrow \RR^{m-1} \times [0,\infty)$ as 
\begin{equation*}
\pi_{\circ}(x,y)=(x,|y|) \; \; \textup{where} \; (x,y) \in \left( \RR^{m-1} \times  \RR^{n+1} \right).
\end{equation*}

We observe that $ \pi_{\circ}$ is $1$-Lipschitz:
\begin{align*}
| \pi_{\circ}(x_1,y_1)- \pi_{\circ}(x_2,y_2)|^2=&\left| (x_1-x_2,|y_1|-|y_2|)\right|^2= |x_1-x_2|^2+\left(|y_1|-|y_2| \right)^2 \\ \leq&|x_1-x_2|^2+\left(|y_1-y_2| \right)^2 = \left|(x_1-x_2,y_1-y_2)\right|^2.
\end{align*}

By the constancy theorem, $ {\pi_{\circ}}_{\#}C=Q_0\a{\RR^{m-1} \times \RR^{+}}$ where $Q_0$ is a positive integer. Since $\partial {\pi_{\circ}}_{\#}T=Q\a{\RR^{m-1} \times \{0\}}$, it follows that $Q=Q_0$.

Since $\pi_{\circ}$ is $1$-Lipschitz
\begin{equation}
\Theta(C,0)=\frac{|C|(\bB_1)}{\omega_m} \geq \frac{|{\pi_{\circ}}_{\#}C|(\pi(\bB_1))}{\omega_m}=\frac{Q}{2}.
\end{equation}
This proves the desired inequality. The rest of this proof will be spent studying the equality case. We note that if $C$ is an open book then trivially $\Theta(C,0)=Q/2$. Thus it remains to show that if $\Theta(C,0)=Q/2$ then $C$ is an open book.

We consider the differential form 
\begin{equation}
\omega:=dr \wedge de_1 \wedge ... \wedge de_{m-1}
\end{equation}
where given $(x,y) \in \RR^{m-1} \times \RR^{n+1}$ we take $r=|y|.$
By definition $d\omega=0$ outside of $V$ and $\left \|\omega \right\|\leq 1$. Let $\Omega$ be the volume form in $\RR^{m-1} \times [0,\infty)$. We notice that $\omega = \pi_{\circ}^{\#}\Omega$ and $\|\omega\|=1$ outside of $V$.

We have that 
\begin{equation}
|C|(\bB_1)=\frac{Q\omega_{m}}{2}= \int_{\pi_{\circ}(\bB_1)}\langle \overrightarrow{\Omega},\overrightarrow{{\pi_{\circ}}_{\#}(C)}\rangle d|{\pi_{\circ}}_{\#}(C)|=\int_{\bB_1} \langle \omega,\overrightarrow{C} \rangle d|C|.
\end{equation}
Given that $\left\|\omega\right\|=1$ and $\left\|\overrightarrow{C}\right\|=1$ we must have that
$\langle \omega,\overrightarrow{C} \rangle \leq 1$. Since the equality holds this implies that $\langle\omega(p), \overrightarrow{C}_p \rangle=1$ for $|C|$ almost every $p$ and thus $\omega$ is a calibration for $\overrightarrow{C}$.

We see that this implies that \begin{equation*}
    \overrightarrow{C}_p= \left(\frac{\sum_{k=0}^n x_{m+k}e_{m+k}}{r}\right) \wedge e_1 \wedge ... \wedge e_{m-1}
\end{equation*} for almost every $p$. This is because  both of them have norm $1$ and their inner product is equal to $1$, by Cauchy Schwartz they must be a multiple of each other (which implies they are equal up to a sign, but once we know this we conclude that they must have the same orientation).

In order to conclude we consider a new map $
\sigma: \RR^{m-1} \times \left(\RR^{n+1} \setminus \left\{0 \right\}\right) \rightarrow \RR^{n+1}$ defined as
\begin{equation*}
\sigma(x,y)=\frac{y}{|y|}  \; \; \textup{where} \; (x,y) \in \left( \RR^{m-1} \times  \left(\RR^{n+1} \setminus \left\{0 \right\}\right)\right).
\end{equation*}

We note that since $T_p(C)=V \oplus \RR^{+}p$ for almost every $p \in C$ we must have that the tangential gradient of $\sigma$ along $C$ vanishes (i.e., $D_{C}\sigma \equiv 0$). This implies that $\sigma \res C$ is locally constant and thus every connected component of $C$ outside of $V$ agrees with a half-plane passing through $V$. We easily conclude from this that $C$ is an open book. See the proof of  \ref{t:classificationcones}.
\end{proof}
\end{section}
\begin{section}{Convex barrier assumption}
In this section, we study the convex barrier assumption and we classify the tangent cones we obtain in that setting.

We recall the definition of the convex barrier in Euclidean space.
\begin{assumption}[Convex Barrier] \label{convexbarrier2} Let $\Omega \subset \RR^{m+n}$ be a domain such that $\partial \Omega$ is a $C^2$ uniformly convex submanifold of $\RR^{m+n}$. We say that $\sum_{i=1}^N Q_i \a{\Gamma_i}$ has a convex barrier if $Q_i$ are positive integers and $\Gamma_i \subset \partial \Omega$ are disjoint $C^2$ closed oriented submanifolds of $\partial \Omega$. In this setting we consider $T$ an area minimizing current with $\partial T=\sum_i Q_i \a{\Gamma_i}.$ 
\end{assumption}
The main theorem of this section will be \begin{theorem}\label{t:convexbarrierdensity}
Let $T$ be as in either Assumption \ref{convexbarrier} (convex barrier) or Assumption \ref{a:convexbarrierlocal} (local convex barrier), then every tangent cone to $T$ at a boundary point $q$ is an open book. Moreover, we have:
\begin{enumerate}[a)]
    \item if $q \in \Gamma_i$ in the first case, then $\Theta(T,q)=Q_i/2$,

    \item if $q \in \Gamma$ in the second case, then $\Theta(T,q)=Q/2$.
\end{enumerate}
\end{theorem}

\subsection{The convex hull property}

In Euclidean ambient space under \ref{convexbarrier2} we must have 
\begin{equation*}
\textup{spt}(T)\subset \textup{ConvexHull}(\cup_{i=1}^N \Gamma).
\end{equation*}
However, this property does not necessarily hold in arbitrary Riemannian manifolds. Indeed some global conditions of the geometry are required to reach a conclusion of that nature.

The following lemma \ref{l:convexbarriernonnegativsectional} illustrates a wider context, that of manifolds non-positive sectional curvature,  where the mentioned property holds. The argument does not rely on the differentiable structure and thus it is also translatable to the context of metric spaces with non-positive sectional curvature. For another relevant reference on the convex barrier property see  Chapter 4 \cite{simon2014introduction} Lemma 6.1 and Theorem 6.2. 

\begin{lemma}[Convex barrier with curvature conditions]\label{l:convexbarriernonnegativsectional} Let $\Sigma$ be a manifold with non-positive sectional curvature. Let $\Gamma$ be an integral current in $\Sigma$ with $\partial \Gamma=0$. Let $T$ an area minimizing current in $\Sigma$ with $\partial T=\Gamma$. Then
\begin{equation*}
\textup{spt}(T) \subset \textup{ConvexHull}\left(\textup{spt}(\Gamma)\right).
\end{equation*}
\end{lemma}
\begin{proof}
Given a geodesic ball $\bB_r^{\Sigma}=\left\{x: d(x,p)<r\right\}$ we take the projection $f_{r}$ from $\Sigma$ to $\bB_r^{\Sigma}$. This is well-defined since $\Sigma$ is uniquely geodesic. Since the convex hull of a set is the intersection of balls which contain it, it is enough to show for every geodesic ball $\bB_r^{\Sigma}$, if $\textup{spt}(\Gamma) \subset \bB_r^{\Sigma}$, then $\textup{spt}(T) \subset \bB_r^{\Sigma}$.

We wish to show that for every ball $B_r$ we have 
\begin{equation*}
\mathbf{M}((f_r)_{\#}(T)) \le \mathbf{M}(T)
\end{equation*} 
with equality iff $\textup{spt}(T) \subset \bB_r^{\Sigma}$. 

We will show that $f_r$ is $1$-Lipschitz and that if $A \subset \partial \bB_s^{\Sigma}$ then
\begin{equation*}
\HH^{m}(f_r(A)) \leq \left(\frac{r}{s}\right)^m \HH^{m}(A).
\end{equation*}

Assume $d(p,x) \geq d(p,y)$. Let $x'$ be such that $d(p,x')=\max \left\{d(p,y),r\right\}$. We take a comparison triangle $p_0$, $x_0$, $y_0$ in Euclidean space and $x'_0$ is such that $|x'_0-p_0|=d(x',p_0)$. By basic Euclidean geometry, $|x'_0-y_0| \leq |x_0-y_0|$. By comparison $d(x',y)\leq |x'_0-y_0|$ and thus $d(x',y) \leq d(x,y)$. 

We reduce then to show the $1$-Lipschitz property of $f_r$ to the case where $d(p,x)=d(p,y)$.

Given two points $x,y$, we take the two geodesics parameterized by arclength $\alpha$ and $\beta$ connecting $p$ to $x$ and $p$ to $y$. We must have $f_r(x)=\alpha(r)$ and $f_r(y)=\beta(r)$.

The non-positive sectional curvature condition implies that $d(\alpha(t),\beta(t))$ is a convex function and thus using $d(\alpha(0),\beta(0))=0$ we get
\begin{align*}
 d(f_r(x),f_r(y))= d(\alpha(r),\beta(r)) \leq& \left(\frac{s-r}{s} \right)d(\alpha(0),\beta(0))+ \left(\frac{r}{s}\right)d(\alpha(s),\beta(s))\\=&\left(\frac{r}{s}\right)d(x,y).
\end{align*}

This implies the desired mass bound.

Assume that
\begin{equation*}
\mathbf{M}((f_r)_{\#}(T))= \mathbf{M}(T)
\end{equation*} 
then we must have that for every
$s$ with $r \leq s$ 
\begin{equation*}\mathbf{M}((f_r)_{\#}(T))=\mathbf{M}((f_s)_{\#}(T))
\end{equation*}
this can only happen if $\textup{spt}(T) \subset \bB_r^{\Sigma}$.
\end{proof}
\subsection{A local version of the convex barrier}
The condition in Assumption \ref{convexbarrier2} is global and we will instead consider local versions of this condition. We introduce some notation on wedges.
\begin{definition}
Given an $(m-1)$-dimensional plane $V \subset \mathbb{R}^{m+n}$ we denote by $\mathbf{p}_V$ the orthogonal projection onto $V$. Given additionally a unit vector $\nu$ normal to $V$ and an angle $\vartheta \in\left(0, \frac{\pi}{2}\right)$ we then define the wedge with spine $V$, axis $\nu$ and opening angle $\vartheta$ as the set
$$
W(V, \nu, \vartheta):=\left\{y:\left|y-\mathbf{p}_V(y)-(y \cdot \nu) \nu\right| \leq(\tan \vartheta) y \cdot \nu\right\} .
$$
\end{definition}

Lemma 2.3 of \cite{delellis2021allardtype}, which we cite below, gives us the desired local geometrical information.
\begin{lemma}\label{l:convexbarrierlocal}Let $T$, $\sum_{i=1}^N Q_i\a{\Gamma_i}$, $\Omega$ be as in \ref{convexbarrier2}. Then there is a $0<\vartheta<\frac{\pi}{2}$ (which is independent of the point in $\cup_{i=1}^N\Gamma_i$) such that the convex hull of $\cup_{i=1}^N\Gamma$ satisfies
$$
\textup{ConvexHull}(\cup_{i=1}^N\Gamma_i) \subset \bigcap_{q \in \Gamma}\left(q+W\left(T_q \Gamma, \nu(q), \vartheta\right)\right)
$$
\end{lemma}
The former condition forces an additional property for area minimizing cones. 
\begin{definition}[Convex barrier cone]\label{d:convexbarriercone}We say that an $m$-dimensional area minimizing cone $C$ in $\RR^{m+n}$ satisfies the convex barrier condition if, up to rotation, $\partial C= Q\a{\RR^{m-1} \times \{0\}}$ and $\exists \theta>0$:
 \begin{equation*}
 \textup{spt}(C) \subseteq \{p=(x,y)\in \RR^m \times \RR^n:|y|\leq \left(\tan \theta\right) x_m \}.
 \end{equation*}
\end{definition}
Clearly if $T$ satisfies assumption \ref{convexbarrier2} then, by the conclusions of \ref{l:convexbarrierlocal} we must have that for every $i$, a tangent cone to $T$ at $\Gamma_i$ must satisfy \ref{d:convexbarriercone} for $Q=Q_i$.

We introduce an analogue of Assumption 2.4 in \cite{delellis2021allardtype} for this setting:

\begin{assumption}\label{a:convexbarrierlocal}
Let $T, \Gamma, \Sigma$ be as in  \ref{A:general}. Assume that there exists $\nu: \Gamma \rightarrow \mathbb{S}^{n+1} \cap T\Sigma$ a map such that $\nu(q) \perp T_q \Gamma$ and $\nu(q) \in T_q \Sigma$.
We say $T$ satisfies the local convex barrier assumption if 
\begin{equation*}
 \operatorname{spt}(T) \subset \bigcap_{q \in \Gamma}\left(q+W\left(T_q \Gamma, \nu(q), \vartheta\right)\right) \cap \Sigma.
\end{equation*}
\end{assumption}
As seen by Lemma \ref{l:convexbarrierlocal}, convex barrier implies local convex barrier. Moreover, local convex barrier still implies that every tangent cone is a convex barrier cone.

\begin{subsection}{Convex barrier cones}  
We wish to classify the convex barrier cones defined in \ref{d:convexbarriercone}. We show this in the following theorem:
\begin{theorem}[Classification of area minimizing cones with convex barrier]\label{t:ClassificationTgtConesConvexBarrier}
Assume that $C$ is an area minimizing cone in $\RR^{m+n}$ 
with 
\begin{equation*}
\partial C=Q\a{\RR^{m-1} \times \left\{0\right\}}
\end{equation*}
and for some $\theta>0$
\begin{equation*}
\textup{spt}(C) \subseteq \{p=(x,y)\in \RR^m \times \RR^n:|y|\leq \left(\tan \theta\right) x_m \}=:W.
\end{equation*}
Then $C$ is an open book with each half-plane $H_i$ satisfying $H_i \subset W$.
\end{theorem}
Notice that the set $W$ plays the role of the convex barrier hypothesis at the level of the tangent cone.

We notice that $\textup{spt}(C)  \cap \left\{x_m=0\right\} \cap W \subset \RR^{m-1}$ and thus $|C|(\left\{x_m=0\right\})=0.$ Thus the above theorem will be an immediate consequence of the following slightly more general theorem:
\begin{theorem}\label{t:classificationcones}
Assume that $C$ is an area minimizing cone in $\RR^{m+n}$ 
with 
\begin{equation*}
\partial C=Q\a{\RR^{m-1} \times \left\{0\right\}}
\end{equation*}
and
\begin{equation}\label{e:convexbarrier}
\textup{spt}(C) \subseteq \left\{x_m \geq 0 \right\}
\end{equation}
Then $C \res \left\{x_m>0\right\}$ is an open book. This means, there exist half-planes $H_1, ..., H_N$ and positive integer multiplicities $Q_i$ such that $\partial \a{H_i}=\a{V}$ and $H_i \subset \left\{x_m>0\right\}$, such that:
\begin{equation}
C \res \left\{x_m>0\right\}=\sum_{i=1}^N Q_i \a{H_i}.
\end{equation}
In particular, if $|C|(\left\{x_m=0 \right\})=0$, then $C$ is an open book and $\sum_{i=1}^N Q_i=Q$.
\end{theorem}

Allard shows a version of the classification of tangent cones with a convex barrier for the multiplicity $Q=1$ case in his boundary regularity paper \cite{AllB}. We will show that his ideas extend to higher multiplicities $Q>1$.

We remind the notation used in the alternative proof of the classification of tangent cones
$\sigma: \RR^{m-1} \times \left(\RR^{n+1} \setminus \left\{0 \right\}\right) \rightarrow \RR^{n+1}$ defined as
\begin{equation}\label{e:mapsigma}
\sigma(x,y)=\frac{y}{|y|}  \; \; \textup{where} \; (x,y) \in \left( \RR^{m-1} \times  \left(\RR^{n+1} \setminus \left\{0 \right\}\right)\right).
\end{equation}

The following is a modification of a general result about varifolds from Allard in his boundary regularity paper Section 5 \cite{AllB}.

\begin{theorem}\label{t:allardconvexvarifoldslemma} Fix a $2$-dimensional plane $\pi_0 \subset \RR^{m+n}$ and assume that:
\begin{enumerate}[i)]
    \item $\mathcal{V}$ is an $m$-dimensional varifold which is stationary on $\RR^{m+n} \setminus \pi_0^{\perp}$.
    \item $\Theta(\mathcal{V},x) \geq c_0 >0$ for some positive number $c_0$ and for $|\!|\mathcal{V}|\!|$-almost every $x \in \textup{spt}(\mathcal{V}) \setminus \pi_0^{\perp}$.
    \item $\mathcal{V}$ is a cone (i.e. $\left(\lambda_{0,r}\right)_{\#}\mathcal{V}=\mathcal{V}$).
    \item The orthogonal projection of $\textup{spt}(\mathcal{V})$ onto $\pi_0$ is contained on a half-space. This means there is a vector $e \in \pi_0 \cap \mathbb{S}^{m+n-1}$ such that 
    \begin{equation}\label{e:convexbarriervarifold}
        \mathbf{p}_{\pi_0}(\textup{spt}(\mathcal{V}))\subset \{x:0 \leq \langle x,e \rangle \}.
    \end{equation}
\end{enumerate}
The following conclusions holds:
\begin{enumerate}[a)]
\item There exist finitely many $v_1,..., v_{k(\rho)} \in \pi_0 \cap \mathbb{S}^{m+n-1}$ such that 
\begin{equation}
 \left( \bB_1 \cap \textup{spt}(\mathcal{V}) \right) \setminus B_{\rho}(\pi_0^{\perp})  \subseteq \bigcup_{i=1}^{k(\rho)} \pi_0^ {\perp} + \RR^+ v_i
\end{equation}
for every $\rho>0$.
\item  We denote $(x,y)=z$, $x^{\perp}=(-x_2,x_1)$ if $x=\mathbf{p}_{\pi_0}(z)$.  For $\mathcal{V}$ a.e. $(z,\pi)$ where $z \in \RR^{m+n} \setminus{\pi_0^{\perp}}$ we have $\mathbf{p}_{\pi}(x^{\perp})=0.$
\end{enumerate}
\end{theorem}
We remind the notation used in the alternative proof of the classification of tangent cones
$\sigma: \RR^{m-1} \times \left(\RR^{n+1} \setminus \left\{0 \right\}\right) \rightarrow \RR^{n+1}$ is defined as
\begin{equation*}
\sigma(x,y)=\frac{y}{|y|}  \; \; \textup{where} \; (x,y) \in \left( \RR^{m-1} \times  \left(\RR^{n+1} \setminus \left\{0 \right\}\right)\right).
\end{equation*}
\begin{corollary}\label{c:allardconvexvarifoldslemma}
Assume that
\begin{enumerate}[i)]
    \item $\mathcal{V}$ is an $m$-dimensional varifold which is stationary on $\RR^{m+n} \setminus \left( \RR^{m-1} \times \left\{0\right\}\right)$.
    \item $\Theta(\mathcal{V},x) \geq c_0 >0$ for some positive number $c_0$ and for $|\!|\mathcal{V}|\!|$-almost every $x \in \textup{spt}(\mathcal{V}) \setminus \left(\RR^{m-1} \times \left\{0\right\}\right)$.
    \item $\mathcal{V}$ is a cone (i.e. $\left(\lambda_{0,r}\right)_{\#}\mathcal{V}=\mathcal{V}$).
    \item The varifold $\mathcal{V}$ is contained in a hemispace
    \begin{equation*}
      \textup{spt}(\mathcal{V})\subset \{x: x_m \geq 0\}.
    \end{equation*}
\end{enumerate}
Then for $\mathcal{V}$ a.e. $(x,\pi)$ where $x \in \RR^{m+n} \setminus \left\{x_m=0\right\}$ 
\begin{equation*}
D_{\pi}\sigma(x)=0.
\end{equation*}
\end{corollary}
\begin{proof}[Proof of  Corollary \ref{c:allardconvexvarifoldslemma}]
The proof of this Corollary immediate by applying the Theorem \ref{t:allardconvexvarifoldslemma} for the planes $\pi_{0,k}=\RR e_{m} \oplus \RR e_{m+k}$ where $1 \leq k \leq n$. Since $\pi_{0,k}^{\perp} \subset \left\{x_m=0\right\}$ we have that for $\mathcal{V}$ a.e. $(x,\pi)$ with $x_m>0$, $\mathbf{p}_{\pi}(-x_{m+k}e_m+x_me_{m+k})=0$ which implies $D_{\pi}\sigma(x)=0$.

\end{proof}

\begin{proof}[Proof of Theorem \ref{t:ClassificationTgtConesConvexBarrier}]

\textbf{We prove the decomposition of $C \res \left\{x_m>0 \right\}$ first with, not necessarily positive, integer multiplicities:}

We apply Corollary \ref{c:allardconvexvarifoldslemma}. This implies that $|C|$ almost everywhere in  $C \res \left\{x_m>0 \right\}$ we have
\begin{equation*}
D_{C}\sigma \equiv 0.
\end{equation*}
By Almgren \cite{Alm}, De Lellis - Spadaro \cite{DS3,DS4,DS5}, the area minimizing cone $C$ is regular up to a set of codimension $2$. For every regular point $p \in \textup{spt}(C) \cap \left\{x_m >0\right\}$ there is a unique half-plane $H_p$ which passes through $p$ and $\RR^{m-1} \times \left\{0\right\}$. Since the tangential derivative of $\sigma$ along $C$ is zero, by unique continuation $C \res H_p= Q_p \a{H_p}$ where $Q_p \neq 0$. We choose the orientations of $H_p$ so that $\partial \a{H_p}=\a{V}.$

We choose distinct half-planes $H_i$ which satisfy that property. Suppose there are at least $N$ of them. Then
\begin{equation*}
|C|(\bB_1) \geq |C|(\left\{x_{m}=0\right\} \cap \bB_1) + \sum_{i=1}^N |C|(H_i \cap \bB_1) \geq \frac{N \omega_{m}}{2}.
\end{equation*}
Since $N$ needs to be bounded, there are only finitely many possible such planes and thus we get
\begin{equation*}
C \res \left\{ x_m >0\right\}=\sum_{i=1}^N Q_i\a{H_i}.
\end{equation*}

\textbf{We show that the multiplicities must be all positive and conclude}

It remains to show all $Q_i$ have the same sign. If this was not the case assume without loss of generality that $Q_1>0$ and $Q_2<0$. 

 We define $S:= \a{H_1}-\a{H_2}$ and $T:= (Q_1-1)\a{H_1}+(Q_2+1)\a{H_2}+\sum_{i=3}^k Q_i\a{H_i} + C \res \left\{x_{m}=0 \right\}$. Since $S+T=C$ and $|\!|S|\!|+|\!|T|\!|=|\!|C|\!|$ it must be that both $S$ and $T$ must be area minimizing. 

By definition $\partial S=0$, since we oriented $H_1$ and $H_2$ so that $\partial \a{H_1}=\partial \a{H_2}=\a{V}$. We know that $S=\a{\RR^{m-1}}\times\left(\a{L_1} -\a{L_2}\right)$ where $L_1$, $L_2$ are half-lines in $\RR^{n+1}$. The only way that $S$ is an area minimizing cone is that $\a{L_1}-\a{L_2}$ is an area minimizing cone with no boundary $\RR^{n+1}$. This can only happen when $L_1=-L_2$ (i.e. they form a full line), which is impossible in the setting we are considering because both lines have positive $x_m$ coordinate with opposite orientations.

We can also see that $Q_i$ must be positive. If $Q_i$ are all negative, we study the tangent cones at points in $V$. By the Almgren stratification, we can take a point in $V$ such that the tangent cone splits $V$. The tangent cone $C'$ at a point $p \in V$ must be 
\begin{equation*}
  C'=  \RR^{m-1} \times \a{S}
\end{equation*}
where $S$ is a $1$d area minimizing cone, which either has flat support or is a sum of half-lines all with the same orientation. By assumption, since $Q_i$ are negative the orientation of the half-lines is opposite that of $V$ and thus the support would need to be flat. The flat case is impossible since the support of $C'$ must be on $x_m \geq 0$ and include a half-plane with positive $x_m$ coordinate.

Thus, we obtain the desired decomposition. When $|C|(\left\{x_m=0\right\})=0$, we must have $\sum_{i=1}^N Q_i=Q$ so that it takes the right boundary data.
\end{proof}

We will use the theory of general varifolds and the monotonicity formula for stationary (not necessarily integral) varifolds. See Simon's lecture notes \cite{simon2014introduction} for a reference on the theory of general varifolds.

Assume $\pi_0= \{(x_1,x_2,0,...,0) \in \RR^{m+n} \}$ and we denote $z=(x,y) \in \pi_0 \times \pi_0^{\perp}$. The following lemma implies \ref{t:allardconvexvarifoldslemma}:
\begin{lemma}\label{l:allardwinding}Assume $\mathcal{V}$ is a varifold which is stationary on $\RR^{m+n} \setminus \pi_0^{\perp}$ and is a cone (conditions (1) and (3) of Theorem \ref{t:allardconvexvarifoldslemma}). Consider the following Radon measure $\mu$ on $\mathbb{S}^1:=\{x \in \pi_0: |x|=1\}$
\begin{equation} 
\int \varphi d\mu =\int_{\bB_1 \setminus \pi_0^{\perp}} \varphi \left( \frac{x}{|x|}\right) |x|^{-2} |\mathbf{p}_{\pi}(x^{\perp})|^2d\mathcal{V}(z,\pi)
\end{equation}
    where $(x,y)=z$, $x^{\perp}=(-x_2,x_1)$. Then $\mu=c \HH^1$ for some $c \in \RR$. 
    
    In addition, if $\mu=0$ and we have conditions (2) and (4) in Theorem \ref{t:allardconvexvarifoldslemma} then:
    \vspace*{-0.2cm}
\begin{enumerate}
\item $\mathbf{p}_{\pi}(x^{\perp})=0$ for $\mathcal{V}$ a.e. $(z,\pi)$ where $z \in \RR^{m+n} \setminus{\pi_0^{\perp}}$.
\item The set $\mathbf{p}_{\pi_0}(\textup{spt}(\mathcal{V}) \setminus B_{\rho}(\pi_0^{\perp})) \cap S^{1}$ is finite and it has at most $\frac{2^m \Theta(\mathcal{V},0)}{\rho^m c_0}$ points.
\end{enumerate}  
\end{lemma}
\begin{proof}[Proof of Theorem \ref{t:allardconvexvarifoldslemma}]
The condition \eqref{e:convexbarriervarifold} implies that the support of $\mu$ must be contained in the arc of $\mathbb{S}^1$ with positive $e$ component. Thus $\mu \equiv 0$ and we conclude Theorem \ref{t:allardconvexvarifoldslemma}.
\end{proof}
In order to conclude we prove the remaining lemma for varifolds 
 \begin{proof}[Proof of Lemma \ref{l:allardwinding}]
 To prove that $\mu=c\HH^1$ we must show that if $\varphi \in C(S^1)$ and has zero average, i.e. $\int_{\mathbb{S}^1} \varphi d\HH^1=0$, then 
 \begin{equation*}
     \int \varphi d\mu=0.
 \end{equation*}
 We consider $\varphi$ as a $2\pi$ periodic function on $\RR$ and we take $\psi(t):=\int_0^t \varphi(\tau)d\tau$ to be  a primitive of $\varphi$ on $\RR$.
 Since $\varphi$ has zero average, $\psi$ is also $2\pi$ periodic and thus it induces a map $\psi: \mathbb{S}^1 \rightarrow \RR$.
 We extend $\psi$ to a $0$-homogeneous function on $\RR^2$. We get then that 

\begin{equation}
    \nabla \psi(x)=\varphi\left(\frac{x}{|x|}\right)\frac{x^{\perp}}{|x|^2}.
\end{equation}
We need to show that
\begin{equation}\label{eq:allardmeasure}
\int_{\bB_1 \setminus \pi_0^{\perp}} \varphi \left( \frac{x}{|x|}\right) \left|\mathbf{p}_{\pi}\left( \frac{x^{\perp}}{|x|}\right)\right|^2d\mathcal{V}(z,\pi)=0.
\end{equation}
We will deduce this from the fact that $\mathcal{V}$ is stationary (i.e. $\int \textup{div}_{\pi}X(z)d\mathcal{V}(z,\pi)=0$ for every vector field $X \in C^1(\RR^n)$).  We will want \eqref{eq:allardmeasure} to get the identity by using the first variation formula over a specific vector field. The fact that the integral is computed on $\bB_1(0) \setminus \pi_0^{\perp}$ is a technical issue that we will resolve by introducing a cutoff. In order to get the gradient of $\psi$, it is reasonable to consider the vector field $\psi(x)x^{\perp}$ (in fact this is what Allard does).

We consider a vector field of the form 
\begin{equation*}
X(z)=\alpha(|z|)\beta(|x|)\psi(x)x^{\perp}.
\end{equation*}
Where $\alpha, \beta \in C^{\infty}(\RR)$ are defined with the following properties:
$\alpha \equiv 1$ for $t \leq 1-\delta$, $\alpha \equiv 0$ for $t \geq 1 - \delta/2$, and it is decreasing.
$\beta \equiv 1$ for $t \geq \delta$, $\beta \equiv 0$ for $t<\delta/2$, and it is increasing with $\|{\beta}'\|_{C^0} \lesssim \delta^{-1}$. Of course, for each $\delta$ we will have a different vector field, but we will abuse notation and not write the $\delta$ dependence. 

\textbf{We will compute $\textup{div}_{\pi}(X)$.} We fix an orthonormal basis $v_1,...v_m$ for $\pi$ and compute in this coordinates the following:

\begin{align*}
\textup{div}_{\pi}(X)=(\mathrm{I})+(\mathrm{II})+(\mathrm{III})+(\mathrm{IV}),
\end{align*}
where 
\begin{equation*}
(\mathrm{I}):=\sum_{i=1}^m \alpha'(|z|)\left(\frac{z}{|z|} \cdot v_i\right) \beta(|x|)\psi(x)\left(x^\perp \cdot v_i\right),
\end{equation*}
\begin{equation*}
(\mathrm{II}):=\sum_{i=1}^m \alpha(|z|)\beta'(|x|)\left(\frac{x}{|x|} \cdot v_i\right) \psi(x)\left(x^\perp \cdot v_i\right),
\end{equation*}
\begin{equation*}
(\mathrm{III}):=\sum_{i=1}^m \alpha(|z|)\beta(|x|) \left(\nabla \psi(x) \cdot v_i \right) \left(x^\perp \cdot v_i \right),
\end{equation*}
\begin{equation*}
(\mathrm{IV}):=\alpha(|z|)\beta(|x|)\psi(x)\sum_{i=1}^{m}\left(D_{v_i}x^\perp \cdot v_i\right).
\end{equation*}
We deal with every term separately.

\textbf{As for the first term:} Since $\mathcal{V}$ is a cone, $z \in \pi$ for $\mathcal{V}$ a.e. $(z,\pi)$.
This means $\sum_{i=1}^{m}(z\cdot v_i)v_i=z$ and we can rewrite 
\begin{equation*}
(\mathrm{I})=\frac{\alpha'(|z|)}{|z|}\beta(|x|)\psi(x) (z\cdot x^{\perp}).
\end{equation*}
But $z \cdot x^{\perp}=0$, so $(\mathrm{I})=0$.

\textbf{As for the second term:}
Here we just use $|\beta'|\lesssim \delta^{-1}$, $\textup{spt}(\beta') \subseteq \{|x|<\delta\}$ and $|x^{\perp}\cdot v_i| \leq \delta$ to conclude:

\begin{equation*}
    \int |(\mathrm{II})|d\mathcal{V} \lesssim |\!|\mathcal{V}|\!|\left(B_1 \cap \left(B_{\delta} (\pi_0^{\perp}) \setminus \pi_0^{\perp} \right)\right).
\end{equation*}
In particular, since 
\begin{equation*}
\bigcap_{\delta>0} B_{\delta}(\pi_0^{\perp})=\varnothing.
\end{equation*}
We get $\lim_{\delta \rightarrow 0} \int |(\mathrm{II})|d\mathcal{V} =0$.

\textbf{As for the third term:} We recall
\begin{equation*}
    \nabla \psi(x)=\varphi\left(\frac{x}{|x|}\right)\frac{x^{\perp}}{|x|^2}.
\end{equation*}
Thus
\begin{align*}
(\mathrm{III})=& \frac{1}{|x|^2}\,\alpha(|z|)\beta(|x|)\,\varphi\left(\frac{x}{|x|}\right)\sum_{i=1}^{m}\left(x^{\perp}\cdot v_i\right)^2= \frac{1}{|x|^2}\,\alpha(|z|)\beta(|x|)\,\varphi\left(\frac{x}{|x|}\right)|\mathbf{p}_{\pi}(x^{\perp})|^2 \\
=& \alpha(|z|)\beta(|x|)\,\varphi\left(\frac{x}{|x|}\right)\left|\mathbf{p}_{\pi}\left(\frac{x^{\perp}}{|x|}\right)\right|^2.
\end{align*}
This implies that 
\begin{equation*}
\lim_{\delta \rightarrow 0} \int (\mathrm{III})d\mathcal{V}=\int_{\bB_1(0) \setminus \pi_0^{\perp}} \varphi \left( \frac{x}{|x|}\right) \left|\mathbf{p}_{\pi}\left( \frac{x^{\perp}}{|x|}\right)\right|^2d\mathcal{V}(z,\pi).
\end{equation*}

\textbf{As for the last term:} We have

\vspace*{0.2cm}

\begin{equation}Dx^{\perp}= -e_1\otimes e_2 +e_2 \otimes e_1=
\begin{pmatrix}
0 & -1 & 0_{(m+n-2)\times 1}\\
1 & 0 & 0_{(m+n-2)\times 1}\\
0_{1\times (m+n-2)} & 0_{1\times (m+n-2)} & 0_{(m+n-2) \times (m+n-2)}
\end{pmatrix}
\end{equation}
The former matrix is antisymmetric matrix.
Notice \begin{equation*}
    \sum_{i=1}^m D_{v_i}X^{\perp} \cdot v_i= \left < Dx^{\perp}: \left( \sum_{i=1}^m v_i \otimes v_i \right)\right>.
\end{equation*} Given that the matrix $\sum_{i=1}^{m} v_i \otimes v_i$ is symmetric and the fact that the Hilbert-Schmidt product of a symmetric matrix and an antisymmetric matrix vanishes, $(\mathrm{IV})=0.$

Since $\mathcal{V}$ is a stationary varifold,
\begin{equation*}
\int \textup{div}_{\pi}X(z)d\mathcal{V}(z,\pi)=0.
\end{equation*} We conclude taking $\delta \rightarrow 0$ that \eqref{eq:allardmeasure} holds.

\textbf{We conclude the theorem studying the case $\mu \equiv 0$.}
Notice that if $\mu=0$ then $\mathbf{p}_{\pi}(x^{\perp})=0$ for $\mathcal{V}$ a.e. $(z,\pi)$ where $z \in \bB_1 \setminus {\pi_0^{\perp}}$ and since $\mathcal{V}$ is a cone, $\mathbf{p}_{\pi}(x^{\perp})=0$ for $\mathcal{V}$ a.e. $(z,\pi)$ where $z \in \RR^{m+n} \setminus {\pi_0^{\perp}}$ .

We need to bound the cardinality of the set
\begin{equation*}
    \mathbf{p}_{\pi_0}(\textup{spt}(\mathcal{V}) \setminus B_{\rho}(\pi_0^{\perp})) \cap \mathbb{S}^{1}:=F.
\end{equation*}
Pick $k$ distinct points $\xi_1,...,  \xi_k \in F \subseteq \mathbb{S}^1$. We consider $k$ bump functions $\chi_1,... ,\chi_k \in C^\infty(\mathbb{S}^1)$, identically $1$ in a neighborhood of $\xi_i$ and such that $\sum_i \chi_i=1$.

We define the truncated varifolds $\mathcal{V}_i$ by
\begin{equation*}
\int \varphi(z,\pi)d\mathcal{V}_i(z,\pi)= \int \varphi(z,\pi)\chi_i\left(\frac{x}{|x|} \right)d\mathcal{V}(z,\pi).
\end{equation*}
We claim that each $\mathcal{V}_i$ is stationary on $\RR^{m+n} \setminus \pi_0^{\perp}$:

Given a vector field $X \in C^{\infty}_{c}(\RR^{m+n}\setminus \pi_0^{\perp})$ then
\begin{align*}
\int \textup{div}_{\pi}X(z)d\mathcal{V}_i(z,\pi)&= \int \textup{div}_{\pi}\left(X(z)\chi_i\left(\frac{x}{|x|} \right)\right)d\mathcal{V}(z,\pi)\\&- \int \sum_j \left(X(z) \cdot v_j \right) D_{v_j} \left(\chi_i \left(\frac{x}{|x|}\right)\right)d\mathcal{V}(z,\pi).
\end{align*}
The first integral on the right hand side vanishes because $\mathcal{V}$ is stationary. For the second integral on the right hand side we observe that
\begin{equation*}
 D_{v_j} \left(\chi_i \left(\frac{x}{|x|}\right)\right)=\varphi_i\left(\frac{x}{|x|}\right)\frac{x^\perp}{|x|^2}
\end{equation*}
for an appropiate function $\varphi_i \in C^{\infty}_c(\mathbb{S}^1)$.
In particular, we have
\begin{align*}
\sum_j \left(X(z) \cdot v_j \right) D_{v_j} \left(\chi_i \left(\frac{x}{|x|}\right)\right)= \,&\varphi_i\left(\frac{x}{|x|}\right)\frac{1}{|x|^2} \sum_j \left(X(z)\cdot v_j\right) \left(x^{\perp} \cdot v_j \right)\\= \,&\varphi_i \left( \frac{x}{|x|} \right) \frac{1}{|x|^2}\left( \mathbf{p}_{\pi}(x^{\perp}) \cdot X(z) \right).
\end{align*}
This term must be zero since we already establish that $\mathbf{p}_{\pi}(x^{\perp})=0$
for $\mathcal{V}$ a.e. $(z,\pi)$ where $z \in \RR^{m+n} \setminus \pi_0^{\perp}$. This shows that $\mathcal{V}_i$ is stationary.

Now we consider $\zeta_i \in \textup{spt}(\mathcal{V}) \cap \bB_1 \setminus B_{\rho}(\pi_0^\perp)$ such that
\begin{equation*}
\frac{\mathbf{p}_{\pi}(\zeta_i)}{|\mathbf{p}_{\pi}(\zeta_i)|}=\xi_i
\end{equation*}
Then $\Theta(\mathcal{V}_i, \zeta_i)=\Theta(\mathcal{V},\zeta_i) \geq c_0$ and since $B_{\rho}(\zeta_i) \subseteq \RR^{m+n} \setminus \pi_0^{\perp}$, by the monotonicity formula 
\begin{equation*}
    |\!|\mathcal{V}_i|\!|(B_\rho(\zeta_i)) \geq c_0\omega_m \rho^{m}.
\end{equation*}
Now 
\begin{equation*}
|\!|\mathcal{V}|\!|(B_2(0)) \geq \sum_i |\!|\mathcal{V}_i|\!|(B_2(0)) \geq \sum_i|\!|\mathcal{V}_i|\!|(B_{\rho}(\zeta_i)) \geq k \omega_m c_0 \rho^m.
\end{equation*}

Since $|\!|\mathcal{V}|\!|(B_2(0))=\omega_m 2^{m}\Theta(\mathcal{V},0)$, the last inequality gives us the desired bound.
 \end{proof}
\end{subsection}
\end{section}

\begin{section}{Main reductions}
In this section we will reduce the proofs of the main theorems in this paper to a proof of an excess decay type lemma. We show the uniqueness of the tangent cone (Theorem \ref{t:uniquetangentcone}), the Hölder continuity of the multivalued normal (Corollary \ref{c:normal}), and the decomposition Theorem \ref{T:decomposition-1sided}.
\begin{subsection}{Proof of small excess uniqueness of the tangent cone}
\begin{definition}We say an open book $C$ is admissible for $Q\a{\Gamma}$ at a point $p \in \Gamma$ if
$C=\sum Q_i \a{H_i}$ where $\sum Q_i=Q$ and $\partial H_i= T_p(\Gamma).$  Additionally, we require that for every $i$  $H_i \subset T_p(\Sigma).$
\end{definition}
\begin{assumption}[Main Reductions Assumption] \label{a:currentandcone} Let $T$ and $\Gamma$ be as in \ref{A:general} with $|\!|T|\!|(\bB_2) \leq \left(2^{m-1}Q+1\right)\omega_m$, $\mathbf{A}_{\Gamma}+\mathbf{A}_{\Sigma} \leq 1$. We assume furthermore that $C$ is an admissible open book.
\end{assumption}

\begin{remark}[Choices of parameters]\label{rmk:parameters}
There will be $Q$ versions of excess decay depending the number of sheets of the cone. The four relevant parameters will be $\theta_N \in \RR^+$, $\eta_N \in (0,1/2)$, $\varepsilon_N \in \RR^+$, $\gamma(Q,m,n,\overline{n})$. The parameter $\gamma(Q,m,n,\overline{n})$ is purely dimensional and depends on the comparison between excess and Dirichlet energy and not on $N$.
The constant $\theta_N$ is chosen so  $\theta_N \ll \min_{N+1 \leq i \leq Q}\varepsilon_i$ and that $\theta_N \ll \gamma(Q,m,n,\overline{n})$.  The parameters $\eta_{N}$ and $\varepsilon_{N}$ are chosen to ensure excess decay on a cone with $N$ sheets with parameter $\theta_N.$ Additionally, we require that $\varepsilon_{N} \ll \gamma(Q,m,n,\overline{n}).$
\end{remark}
\begin{lemma}[$N$th excess decay with small angle and small second fundamental forms] \label{lem:saexcessdecay} Let $T$ and $C$ be chosen as in Assumption \ref{a:currentandcone}. For every $\theta_N \in \RR^+$ there exist $\varepsilon_N \in \RR+$,  $\eta_N \in (0,1/2)$ such that if 
\begin{equation*}
\mathbb{E}(T,C,\mathbf{B}_1) \leq \varepsilon_N \alpha(C)^2 \: \: \textup{and} \: \: \mathbf{A}_{\Gamma}+\mathbf{A}_{\Sigma}^2 \leq \varepsilon_{N} \mathbb{E}(T,C,\mathbf{B}_1) 
\end{equation*}
(when $N=1$ the first inequality becomes $\mathbb{E}(T,C,\bB_1) \leq \varepsilon_1$)
there exist a radius $r_0 \in [\eta_{N},1/2]$ and an admissible open book $C'$ such 
that
\begin{equation}\label{eq:angletilt}
\mathcal{G}(C',C)^2 \leq \gamma(Q,m,n,\overline{n})\mathbb{E}(T,C,\bB_1)
\end{equation}
and
\begin{equation*}
    \mathbb{E}(T,C',\bB_{r_0}) \leq \theta_N \min\left\{\alpha(C')^2,\mathbb{E}(T,C,\bB_1) \right\}.
\end{equation*}
 \end{lemma}

 \begin{remark}
 Under the hypothesis above, \eqref{eq:angletilt} remains valid up to modifying the constant $\gamma(Q,m,n,\overline{n})$, with the following
\begin{equation*}
    C'=\sum_{i,j} Q_{i,j}\a{H_{i,j}}
\end{equation*}
 has $\sum_j Q_{i,j}=Q_i$ and the angle condition 
 
\begin{equation}\label{eq:anglecondition}
   \angle \left(H_i, H_{i,j}\right) \leq \gamma(Q,m,n,\overline{n}) \mathbb{E}(T,C,\bB_1)^{1/2}
    \end{equation}
\end{remark}

In particular, we know that one of the following two statements holds:
\begin{enumerate}
    \item $C'$ has the same number of sheets than $C$ and \begin{equation}
        \alpha(C') \geq \alpha(C)-\gamma(Q,m,n,\overline{n})^{-1/2}\mathbb{E}(T,C,\bB_1)^{1/2}.
    \end{equation}
    \item $C'$ has more sheets than $C$ (this case will never happen when $C$ consists of $Q$ multiplicity $1$ sheets).
    \end{enumerate}

 Our excess decay needs a control on the angle of the cone, which we require for our graphical approximations to be valid later.  We require a combinatorial argument that controls the angle. 
 If excess is large relative to the angle we will prune the cone by collapsing the two closest sheets into one and adjusting the multiplicities appropriately. This is collected in the following proposition:
\begin{proposition}\label{prop:anglecases} Let $T,\Gamma,C$ be as in \ref{a:currentandcone}.
Given positive parameters $\left \{ \varepsilon_i \right\}_{1 \leq i \leq Q}$, there exists $\varepsilon>0$ such that if $\mathbb{E}(T,C,\bB_1)<\varepsilon$ then we have the following:
\begin{itemize}
    \item An integer $1 \leq N' \leq N$ and an admissible cone $C'$ with $N'$ sheets.
    \item A partition function $h:\left\{ 1,..., N\right\} \rightarrow \left\{A \subseteq \left\{ 1,..., N\right\} \right\}$ (i.e. $h(i)$ are disjoint, possibly empty sets and their union is $\left\{ 1,..., N\right\}$).
    \item  $C'=\sum \sum_{j \in h(i)} Q_j \a{H_i}$. In particular, if $N=N'$ then $C'=C.$
    \item \begin{equation*}
        \mathbb{E}(T,C',\bB_1) \leq C(m)\left(\min_{1 \leq k \leq Q} \varepsilon_k\right)^{-Q}  \mathbb{E}(T,C,\bB_1).
    \end{equation*}
    \item We have
    \begin{equation*}
    \mathbb{E}(T,C',\bB_1) \leq \varepsilon_{N'}\alpha(C')^2
    \end{equation*}
    where we remind that if $C'$ is a single half-plane then $\alpha(C')=1.$
    \item
    \begin{equation*}
    \mathcal{G}(C',C)^2 \leq C(m)\left(\min_{1 \leq k \leq Q} \varepsilon_k\right)^{-Q} \mathbb{E}(T,C,\bB_1).
    \end{equation*}
    \end{itemize}
\end{proposition}
\begin{proof}
We outline an iterative procedure. Let $C_0=C$. In step $k$ if we have 
\begin{equation*}
\mathbb{E}(T,C_k,\bB_1) \leq \varepsilon_{N-k} \alpha(C_k)^2
\end{equation*}
we take $C'=C_k$, $N'=N-k$ and we stop.

Otherwise assume that
\begin{equation*}
\mathbb{E}(T,C_k,\bB_1) > \varepsilon_{N-k} \alpha(C_k)^2.
\end{equation*}

Then we select two half-planes of $C_k$, $H_i$ and $H_j$, that form an angle of $\alpha(C_k)$. We remove $H_j$ and change the multiplicity of $H_i$ to $Q_i+Q_j$. This cone is defined to be $C_{k+1}$.
We must have
 \begin{equation*}
     \mathbb{E}(T,C_{k+1},\bB_1) \leq c(m) \frac{1}{\delta_{N-k}} \mathbb{E}(T,C_{k},\bB_1).
 \end{equation*}

If the process never stops, we define $C'=C_N$ and $N'=1$. We must have that
\begin{equation*}
        \mathbb{E}(T,C',\bB_1) \leq c(m) \prod_{N'+1 \leq  k \leq N} \frac{1}{\varepsilon_k} \mathbb{E}(T,C,\bB_1).
    \end{equation*}
This implies that if $\varepsilon$ is chosen small enough then in the case we didn't stop we also have 
\begin{equation*}
    \mathbb{E}(T,C',\bB_1) \leq \varepsilon_1 \alpha(C')^2.
\end{equation*}
The last inequality of the proposition is a straightforward consequence of the procedure.
\end{proof}

We are now ready to prove the uniqueness of the tangent cone. 

 \begin{definition}
 We define the modified excess as
 \begin{equation*}
 \hat{\mathbb{E}}(T,C,\bB_r(p)):=\max\left\{\mathbb{E}(T,C,\bB_r(p)), \kappa^{-1} \mathbf{A}_{\Gamma}r ,\kappa^{-1}\mathbf{A}_{\Sigma}^2r^2 \right\}
 \end{equation*}
 The dependence on $\kappa$ is implicit in the notation; in Theorem \ref{t:sexcessuniqueness} we will choose $\kappa$ as $\kappa_{\textup{decay}}$. \end{definition}
  \begin{theorem}[Small Excess Uniqueness of the Tangent Cone]\label{t:sexcessuniqueness}
Let $T$ and $C$ be as in \ref{a:currentandcone}.
There exist $\varepsilon_{\textup{decay}}>0$, $\kappa_{\textup{decay}}>0$ such that if
\begin{equation*}
\hat{\mathbb{E}}(T,C,\bB_1(p))<\varepsilon_{\textup{decay}},
\end{equation*}
then the tangent cone, $C_p$, at $p$ is unique and is an admissible open book.
Here, $\hat{\mathbb{E}}$  is defined using $\kappa=\kappa_{\textup{decay}}$.
Moreover, there exists $\alpha(Q,m,n,\overline{n})>0$ such that for every $\rho<1/2$,
\begin{equation*}
\hat{\mathbb{E}}(T,C_p,\bB_{\rho}(p)) \leq   \rho^{\alpha}\hat{\mathbb{E}}(T,C,\bB_1(p)).
\end{equation*}
We also have 
\begin{equation*}
\mathcal{G}(C_p,C)^2+ \dist_{H}(C_p,C)^2 \leq c(Q,m,n,\overline{n})\hat{\mathbb{E}}(T,C,\bB_1(p)).
\end{equation*}
 \end{theorem}
 \begin{proof}

Before proceeding we apply, we apply the pruning proposition \ref{prop:anglecases} to the cone $C$, obtaining a cone $\hat{C}$ which satisfies 
\begin{equation*}
\mathbb{E}(T,\hat{C},\bB_1) \leq C(m)\left(\min_{1 \leq k \leq Q} \varepsilon_k\right)^{-Q} \mathbb{E}(T,C,\bB_1).
\end{equation*}
 Since $\mathcal{G}(\hat{C},C) \lesssim \mathbb{E}(T,C,\bB_1)^{1/2}$, if we prove the theorem to hold for $\hat{C}$ then it will hold for $C.$ By choosing $\varepsilon$ small enough, we may assume that the starting cone $C$ is pruned (without relabeling) and thus it satisfies the conclusions of Proposition \ref{prop:anglecases}.

\textbf{We will start by proving the following iterative excess decay starting with a pruned cone $C$}.
There exists a sequence of admissible open books $\mathbf{C}^l$ and radii $r_{l+1}r_{l}^{-1} \in [\eta_{N_l},1/2]$, where $\mathbf{C}_0:=C$ and $r_0:=1$, such that 

\begin{itemize}
\item The number of distinct sheets $N_l$ of cone $\mathbf{C}^l$ is non-decreasing. 
      \item      For every $l$ we have $\hat{\mathbb{E}}(T,\mathbf{C}_l,\bB_{r_l})\leq \varepsilon_{N_l}\alpha(\mathbf{C}_l)^2.$

    \item 
    If $\hat{\mathbb{E}}(T,\mathbf{C}_l,\bB_{r_l})> \mathbb{E}(T,\mathbf{C}_l,\bB_{r_l})$
    then 
    $\mathbf{C}^{l+1}:=\mathbf{C}^l$ and $r_{l+1}=\frac{1}{2}r_l$.
   
    \item 
    If $\hat{\mathbb{E}}(T,\mathbf{C}_l,\bB_{r_l})=\mathbb{E}(T,\mathbf{C}_l,\bB_{r_l})$ we apply excess decay \ref{lem:saexcessdecay}. Here $r_{l+1}r_l^{-1} \in [\eta_{N_l},1/2]$ and 
    \begin{equation*}
        \mathcal{G}(\mathbf{C}^{l+1},\mathbf{C}^l) \leq  \gamma(Q,m,n,\overline{n})^{1/2}\mathbb{E}(T,C,\bB_{r_l})^{1/2}.
    \end{equation*}
    \item For every non-negative integer $l$
    \begin{equation*}
        \hat{\mathbb{E}}(T,\mathbf{C}_{l+1},\bB_{r_{l+1}}) \leq \frac{1}{2}\hat{\mathbb{E}}(T,\mathbf{C}_{l},\bB_{r_l}).
    \end{equation*}
\end{itemize}

We prove we can do the above.
We chose $\kappa \ll \min_{1 \leq k \leq Q} \varepsilon_k$.

 \textbf{The case of large second fundamental forms:}
 If
  \begin{equation*}
\hat{\mathbb{E}}(T,\mathbf{C}_l,\bB_{r_l})>C(Q,m,n,\overline{n}) \mathbb{E}(T,\mathbf{C}_l,\bB_{r_l}) 
\end{equation*} we chose $r_{l+1}:=\frac{1}{2}r_l$, 
$\mathbf{C}_{l+1}:=\mathbf{C}_{l}$,  then 
\begin{equation*}
\hat{\mathbb{E}}(T,\mathbf{C}_{l+1},\bB_{r_{l+1}}) \leq \frac{1}{2}\hat{\mathbb{E}}(T,\mathbf{C}_l,\bB_{r_l}). 
\end{equation*}
In the last inequality we use Lemma \ref{l:excessboundscales}.

\textbf{The case of small second fundamental forms:}
We must have 

\begin{equation*}
\mathbb{E}(T,\mathbf{C}_l,\bB_{r_l}) \leq \varepsilon_N \alpha(\mathbf{C}_l)^2 \: \: \textup{and} \: \: \mathbf{A}_{\Gamma}r_l+\mathbf{A}_{\Sigma}^2r_l^2 \leq \varepsilon_{N} \mathbb{E}(T,C,\bB_{r_l}).
\end{equation*}

We are in the hypothesis to apply  Lemma \ref{lem:saexcessdecay}. This gives us the cone $\mathbf{C}_{l+1}$, which satisfies
\begin{equation*}
\mathbb{E}(T,\mathbf{C}_{l+1},\bB_{r_{l+1}}) \leq \theta_{N} \mathbb{E}(T,\mathbf{C}_l,\bB_{r_l}). 
\end{equation*}

If $\mathbf{C}_{l+1}$ has more sheets than $\mathbf{C}_l$ then
\begin{equation*}
\mathbb{E}(T,\mathbf{C}_{l+1},\bB_{r_{l+1}}) \leq \theta_{N}\alpha(\mathbf{C}_{l+1})^2 \leq \varepsilon_{N_{l+1}}\alpha(\mathbf{C}_{l+1})^2
\end{equation*}
the last step by our choice of parameters \ref{rmk:parameters}.

In the case that $\mathbf{C}_{l+1}$ and $\mathbf{C}_l$ have the same number of sheets we have that
\begin{align*}
\alpha(\mathbf{C}_{l+1}) &\geq \alpha(\mathbf{C}_l)-C(Q,m)\mathcal{G}(\mathbf{C}_{l+1},\mathbf{C}_l) \\ &\geq \alpha(\mathbf{C}_l)- C(Q,m)\gamma(Q,m,n,\overline{n})\mathbb{E}(T,\mathbf{C}_{l},\bB_{r_l})^{1/2} \\&\geq (1-\varepsilon_{N} \gamma(Q,m,n,\overline{n})) \alpha(\mathbf{C}_l)  \geq \frac{1}{\sqrt{2}}\alpha(\mathbf{C}_l).
\end{align*}

This implies that
\begin{equation*}
\mathbb{E}(T,\mathbf{C}_{l+1},\bB_{r_{l+1}}) \leq \frac{1}{2}\mathbb{E}(T,\mathbf{C}_{l},\bB_{r_{l}}) \leq \frac{1}{2}\varepsilon_{N_l} \alpha(\mathbf{C}_l)^2 \leq \varepsilon_{N_{l+1}}\alpha(\mathbf{C}_{l+1})^2
\end{equation*}
since $N_l=N_{l+1}$.
Additionally since $r_{l+1}r_l^{-1}< \frac{1}{2}$ the second fundamental form terms must also shrink by a factor of $1/2$ and thus we must have that
\begin{equation*}
\hat{\mathbb{E}}(T,\mathbf{C}_{l+1},\bB_{r_{l+1}}) \leq \frac{1}{2}\hat{\mathbb{E}}(T,\mathbf{C}_{l},\bB_{r_l}).
\end{equation*}

\textbf{Convergence of the cones $\mathbf{C}_l$ and uniqueness of the tangent cone}:

We must have 
\begin{equation*}
\mathcal{G}(\mathbf{C}_{l+1},\mathbf{C}_l) \leq A \hat{\mathbb{E}}(T,\mathbf{C}_l,\bB_{r_l})^{1/2} \leq A\frac{1}{2^l} \mathbb{E}(T,C,\bB_r)^{1/2}
\end{equation*}
for a for a universal parameter $A$ depending on all the other parameters.
This implies that there exists a limit $C_p$ in $\mathcal{G}$ of the sequence $\mathbf{C}_l$ and 
\begin{equation*}
\mathcal{G}(C_p,\mathbf{C}_l) \leq 2A \frac{1}{2^l}\mathbb{E}(T,C,\bB_r)^{1/2}.
\end{equation*}

This implies 
\begin{equation*}
\mathbb{E}(T,C_p,\bB_{r_l}) \leq C(Q,m,n,\overline{n},\min_{1 \leq k \leq Q} \varepsilon_k)\frac{1}{2^l} \mathbb{E}(T,C,\bB_r)^{1/2}.
\end{equation*}

Given $\rho$ we can choose $l$ such that $\rho \in (r_{l+1},r_l]$. Let $\eta:=\min_{1 \leq i \leq Q} \eta_i$. Then

\begin{equation*}
\hat{\mathbb{E}}(T,C_p,\bB_{\rho})\leq C(m) \frac{1}{\eta^{m+2}} \hat{\mathbb{E}}(T,C_p,\bB_{r_l}) \leq C(\eta,\varepsilon_1) \frac{1}{2^l} \hat{\mathbb{E}}(T,C,\bB_r) 
\end{equation*}
We notice that if $\rho<1/2r$ then using that $\eta^{l+1}r \leq r_{l+1} < \rho$ and defining $\alpha:=\textup{log}(2)/\textup{log}(\eta)$ we get
\begin{equation*}
\frac{1}{2^{l+1}} \leq  \left(\frac{\rho}{r}\right)^{\alpha}.
\end{equation*}

This implies as desired
\begin{equation*}
\hat{\mathbb{E}}(T,C_p,\bB_{\rho}) \leq C(Q,m,n,\overline{n},\min_{1 \leq k \leq Q} \varepsilon_k, \eta) \left(\frac{\rho}{r}\right)^{\alpha}\hat{\mathbb{E}}(T,C,\bB_r(p)). 
\end{equation*}
In order to get the estimate on the intermediate scales see Proposition
\ref{l:excessboundscales}.

The last inequality implies the uniqueness of the tangent cone. Indeed, let $r_j \rightarrow 0$ a sequence of scales such that $(\lambda_{r_j})_{\#}T \rightharpoonup C'$. Then 
\begin{equation*}
\mathbb{E}(C',C_p,\bB_1)= \lim_{j \rightarrow \infty} \mathbb{E}((\lambda_{r_j})_{\#}T,C_p,\bB_1)=0.
\end{equation*}
This implies $C'=C_p.$
\end{proof}
\end{subsection}
\subsection{Consequences of the excess decay lemma}

As an immediate consequence we get an upgrade on the classification of the tangent cones and thus on the minimum density.
\begin{theorem}[Density jump] Let $C$ be an area minimizing cone with $\partial C= Q\a{\{0\} \times \RR^{n+1}}$. There exists $\theta(Q,m,n)>0$ such that if $\Theta(C,0)<Q/2+\theta$ then $\Theta(C,0)=Q/2$ (and thus $C$ is an open book).
\end{theorem}
\begin{proof}
Suppose by contradiction there exists a sequence of area minimizing cones $C_j$ with $\Theta(C_j,0) \rightarrow Q/2$. Up to a subsequence we can assume $C_j \rightarrow C$ for $C$ an area minimizing cone. We must have then that $\Theta(C,0)=Q/2$
 and thus $C$ is an open book, by Theorem \ref{t:ClassificationTgtConesNConvex}. 
 
 We must have that for $j$ large enough $\mathbb{E}(C_j,C,\bB_1)<\varepsilon$ for $\varepsilon$ in \ref{t:sexcessuniqueness}. This implies that the tangent cone to $C_j$ at $0$ for some $j$ large enough must be unique and an open book. Thus $C_j$ must be an open book for $j$ large enough.
 \end{proof}

\begin{theorem}[Uniqueness of the tangent cone]\label{t:uniquetangentcone}
Let $T$, $\Gamma$, $\Sigma$ be as in  Assumption \ref{A:general} or the local convex barrier \ref{a:convexbarrierlocal}. Assume in the first case that $\Theta(T,0) \leq Q/2+\theta$. Then $\Theta(T,0)=Q/2$ and the tangent cone at is a unique open book.
\end{theorem}
\begin{proof}
In the first case, since $\Theta(T,0)<Q/2+\theta$ then $\Theta(T,0)=Q/2$. Thus in this case, every tangent cone must be an open book. Analogously, in the second case, by Theorem \ref{t:convexbarrierdensity} every tangent cone at $0$ must be an open book. We take a tangent cone $C$. We must then have that 
\begin{equation*}
\hat{\mathbb{E}}(T,C,\bB_{r_j}) \rightarrow 0
\end{equation*}
which means that there exists $j$ for which $\hat{\mathbb{E}}(T,C,\bB_{r_j})<\varepsilon$ with $\varepsilon$ as in \ref{t:sexcessuniqueness}. This implies that $T$ has a unique tangent cone at $0$ by \ref{t:sexcessuniqueness}. 
\end{proof}
Let $N\Gamma$ be the normal bundle of $\Gamma.$
\begin{definition}\label{d:multivaluednormalmap}
Given $p \in \Gamma$ let $C_p=\sum_{i=1}^N Q_i\a{H_i}$ be the tangent cone at $p$. We define the normal map $\eta:\Gamma \rightarrow \mathcal{A}_Q(N\Gamma)$ given by 
\begin{equation*}
    \eta(p)=\sum_{i=1}^N Q_i \a{\mathbf{p}_{T_p(\Gamma)^{\perp}}(H_i)}.
\end{equation*}
\end{definition}
\begin{corollary}[Multivalued Normal Map Hölder Estimate]
\label{c:normal}
There exists $\varepsilon$ such that if \begin{equation*}
\hat{\mathbb{E}}(T,C,\bB_1(p))<\varepsilon
\end{equation*}
then the following quantitative estimate holds $\forall q_1, q_2 \in B(p,1/16)$ 
\begin{equation*}
\mathcal{G}(\eta(q_1),\eta(q_2))^2 \leq c(Q,m,n,\overline{n}) \hat{\mathbb{E}}(T,C,\bB_1)|q_1-q_2|^{\alpha}.
\end{equation*}
Here, $\hat{\mathbb{E}}$  is defined using $\kappa=\kappa_{\textup{decay}}$.

Moreover, let $U:=\left\{x \in \Gamma: \Theta(x)<Q/2+\theta \right\}$, which forms a relatively open set of $\Gamma$, and let $K\subset U$ a compact subset of $U$. Then 
$\eta \in C^{\alpha}(K).$
\end{corollary}
\begin{proof}
By 
\ref{l:excessboundscales}, in combination with \ref{prop:anglecases} if $\varepsilon$ and $\kappa$ are small enough then for every $q_1 \in \bB_{1/4}(p)$
\begin{equation*}
\hat{\mathbb{E}}(T,C_{q_1},\bB_{1/2}(q_1)) \leq C(Q,m,n,\overline{n}) \hat{\mathbb{E}}(T,C,\bB_1).
\end{equation*}
The open book $C_{q_1}$ is defined by translating $C$ so that it passes through $p$ and tilting it so its spine is $T_{q_1}(\Gamma)$ and it is still contained on $T_{q_1}(\Sigma)$.

Up to taking $\varepsilon$ smaller we must have that for $\mathbf{C}_{q_1}$ the tangent cone at $q_1$, for every $0<r_1<1/4$
\begin{equation*}
    \hat{\mathbb{E}}(T,\mathbf{C}_{q_1},\bB_{r_1}) \leq C(Q,m,n,\overline{n}) r_1^{\alpha}\hat{\mathbb{E}}(T,C,\bB_1).
\end{equation*}

We define $C_{q_1}^{q_2}$ as the cone obtained by translating and tilting $\mathbf{C}_{q_1}$ so that it is centered at $q_2$ and has a spine tangent to $\Gamma$ and it is still contained on $T_{q_2}(\Sigma)$. If $q_1$ and $q_2$ are in $\bB_{1/16}(p)$ and $r_1=2|q_1-q_2|$ then
\begin{equation*}
\hat{\mathbb{E}}(T,C^{q_2}_{q_1},\bB_{r_1}(q_1)) \leq C(Q,m,n,\overline{n})r_1^{\alpha}\hat{\mathbb{E}}(T,C,\bB_1). 
\end{equation*}
Thus if $\varepsilon$ is small enough then 
$\hat{\mathbb{E}}(T,C^{q_2}_{q_1},\bB_{r_1}(q_1))<\varepsilon_{\textup{decay}}$ where $\varepsilon_{\textup{decay}}$ is the one needed to run theorem \ref{t:sexcessuniqueness}. This allows us to apply the theorem and obtain the desired estimate 
\begin{equation*}
\mathcal{G}(\eta(q_1),\eta(q_2))^2 \lesssim \mathcal{G}(C^{q_2}_{q_1},C_{q_1})^2\lesssim |q_1-q_2|^{\alpha}\hat{\mathbb{E}}(T,C,\bB_1)
\end{equation*}

The qualitative conclusion follows by compactness and continuity. By the quantitative statement we must have that the normal map is a continuous map of $U$. We must show that 
\begin{equation*}
    \frac{\mathcal{G}(\eta(p),\eta(q))}{|p-q|^{\alpha}}
\end{equation*} is bounded in $K \times K \subset \Gamma \times \Gamma$. It is clearly bounded away from the diagonal by continuity. We can cover the diagonal by balls on which the quantitative result applies and conclude by compactness.
\end{proof}
 \begin{theorem}[Decomposition lemma] \label{T:decomposition-1sided}
Let $T$ and $\Gamma$ be under Assumption \ref{A:general} and $q \in \Gamma$. Assume that $C$ is a non- flat open book. There exists $\varepsilon_{decomp}>0$ such that the following holds. If $C=\sum_{i=1}^N Q_i\a{H_i}$
\begin{equation}
\hat{\mathbb{E}}(T,C,\bB_1(p))<\varepsilon_{decomp}\alpha(C)^2.
\end{equation}
There exist $T_1,T_2,..,T_N$ area minimizing currents on $\bB_{1/32} \cap \Sigma$ with
\begin{itemize}
\item $\partial T_i=Q_i \a{\Gamma}$ 
\item \begin{equation}
T\res\bB_{1/32}(p)=\sum_{i=1}^N T_i
\end{equation}
\item The supports of $T_i$ only intersect at $\Gamma$.
\end{itemize}
\end{theorem}
\begin{theorem}[Qualitative decomposition theorem]\label{thm:qualitativedecomp}
Let $T$ be an area minimizing current as in Assumption \ref{A:general}. Further assume that $\Theta(T,0)=Q/2$. Let $C=\sum_{i=1}^N Q_i \a{H_i}$ be the unique tangent cone to $T$ at $0$ (where the representation of $C$ is such that the half-planes are distinct). Then there exists $\rho>0$ and area minimizing currents $T_1,T_2,...,T_N$ in $\bB_{\rho}$ such that
\begin{equation*}
T \res \bB_{\rho}=\sum_{i=1}^N T_i 
\end{equation*}
where the supports of $T_i$ only intersect at $\Gamma$, $\partial T_i \res \bB_{\rho}=Q_i \a{\Gamma}$, and the (unique) tangent cone at $0$ of $T_i$ is $Q_i \a{H_i}$.
\end{theorem}
The proof of Theorem \ref{thm:qualitativedecomp} follows immediately from Theorem \ref{T:decomposition-1sided} at some scale small enough from Theorem \ref{t:sexcessuniqueness}.
\begin{proof}[Proof of Theorem \ref{T:decomposition-1sided}]
Let $C_q$ be the tangent at $q$. If $\mathbf{A}_{\Gamma}$ is small enough then the projection $\mathbf{p}_{\Gamma}$ is well-defined.

By assuming that $\hat{\mathbb{E}}(T,C,\bB_1)<\varepsilon_{Holder}$. We start by coupling the estimates from the proof of \ref{c:normal} with the $L^2$-$L^{\infty}$ height bound of De Lellis-Minter-Skorobogatova (Theorem 3.2 \cite{de2023fineIII}). 

If $q \in \textup{spt}(T) \cap \bB_{1/16}$ we have
\begin{equation}\label{e:heightdecomposition}
\frac{\dist(q,C_{\mathbf{p}_{\Gamma}}(q))}{|q-\pi(q)|} \leq C(Q,m,n,\overline{n})|q-\pi(q)|^{\alpha/2}\hat{\mathbb{E}}(T,C,\bB_r(p))^{1/2}.
\end{equation}

We will also have that $\forall q_1,q_2 \in B(p,1/16)$
\begin{equation*}
    \mathcal{G}(\eta(q_2),\eta(q_1))\leq C(Q,m,n,\overline{n})\varepsilon^{1/2} \alpha(C).
\end{equation*}
We also have 
\begin{equation*}
    \mathcal{G}(C_p,C) \leq C(Q,m,n,\overline{n}) \hat{\mathbb{E}}(T,C,\bB_1)^{1/2} \leq \varepsilon_{decomp}\alpha(C).
\end{equation*}
This implies that the tangent cones $C_q$ have at least as many sheets as $C$.

We  parameterize $\bB_1$ as a subset of the normal bundle of $\Gamma$. We consider the function $d:\bB_1 \setminus \Gamma \rightarrow \RR$ 
\begin{equation*}
d(q):=\frac{\dist(q,C_{\pi(q)})}{|q-\pi(q)|}.
\end{equation*}

As a consequence of \eqref{e:heightdecomposition} we get $d(q) \ll \alpha(C_p)$ on $\textup{spt}(T) \cap \bB_{1/16}.$

This means that $\textup{spt}(T)\setminus \Gamma$ is split into $N$ disjoint open sets, each of which is non-empty. If one of the sets was empty, one of the half-planes in the support of $C$ could not be in the support of the tangent cone to $T$ at $0$.

This allows us to decompose $T$ into $N$ currents as desired.
\end{proof}

\begin{theorem}
Let $T$, $\Gamma$ and $\Sigma$ be as in \ref{a:currentandcone}. There exists $\varepsilon$ such that if either
\begin{enumerate}
\item 
\begin{equation*}
    \max \left\{\mathbf{E}(T,C,\bB_1),\kappa^{-1} \mathbf{A}_{\Gamma},\kappa^{-1}\mathbf{A}_{\Sigma}^2 \right\}<\varepsilon
\end{equation*}
\item \begin{equation*}
    \max \left\{\frac{|\!|T|\!|(\bB_1)}{\omega_{m}}-Q/2,\kappa^{-1} \mathbf{A}_{\Gamma},\kappa^{-1}\mathbf{A}_{\Sigma}^2 \right\}<\varepsilon
\end{equation*}
\end{enumerate}
then there exists $C'$ such that
\begin{equation*}
\hat{\mathbb{E}}(T,C',\bB_{1/2})<\varepsilon_{decomp}.
\end{equation*} 
In the first case, we have $\textup{spt}(C') \subset C$ and if 
\begin{equation*}
\int_{\bB_1 \setminus \bB_{1/4}}\dist(x,\textup{spt}(T))^2d|C|<\varepsilon
\end{equation*}
then we have $\textup{spt}(C')=C$.
\end{theorem}
\begin{proof}
We argue in both cases by compactness, by taking a sequence of currents $T_j$ which satisfies the assumptions of the theorem with a vanishing sequence of parameters $\varepsilon=\varepsilon_j$. We assume after rescaling the radius is $1$ and up to rotation $T_0(\Gamma)=\RR^{m-1}\times \left\{0\right\}$. Up to subsequence, $T_j \rightarrow T$, $\Gamma_j \rightarrow \RR^{m-1}\times \left\{0\right\}$, $\Sigma_j \rightarrow \RR^{m+\overline{n}}$.

In the first case up to subsequence also $C_j \rightarrow C$. We must have that $\mathbf{E}(T,C,\bB_1)=0$. This implies $\textup{spt}(T) \subset \textup{spt}(C)$ and thus there exists $C'$ such that $\mathbb{E}(T,C',\bB_1)=0$ which is a contradiction since $\hat{\mathbb{E}}(T_j,C',\bB_1)<\varepsilon_{\textup{decay}}$ for $j$ large enough.

In the second case $T$ must have density constantly equal to $Q/2$ which means that it must be a cone by the monotonicity formula and it must be an open book by \ref{t:ClassificationTgtConesNConvex}. We must then have $\hat{\mathbb{E}}(T_j,T,\bB_1)<\varepsilon_{\textup{decay}}$ for $j$ large enough which is a contradiction.
\end{proof}
We consider for a plane $\pi$ the tilt excess
\begin{equation*}
    \overline{\mathbf{E}}(T, \bB_r(p), \pi) := \frac{1}{\omega_m r^m} \int_{\bB_r(p)} \frac{|\vec{T} (x) - \vec{\pi}|^2}{2} \, d\|T\| (x),
\end{equation*}
\begin{theorem}\label{T:excessdecayforflat}
Let $T,\Gamma,\Sigma$ be as in Assumption \ref{A:general} and assume $\Theta(T,p)=Q/2$. Assume that $T$ has a flat one-sided tangent cone $Q\a{\pi}$ at $p$. There exists $\varepsilon(Q,m,n,\overline{n})>0$ such that if
\begin{equation*}
\max\left\{\overline{\mathbf{E}}(T,\bB_1(p),\pi), \mathbf{A}_{\Gamma},\mathbf{A}_{\Sigma}^2\right\}<\varepsilon
\end{equation*}
then there exists $C(Q,m,n,\overline{n})$ and $\alpha(Q,m,n,\overline{n})>0$ such that
\begin{equation*}
\overline{\mathbf{E}}(T,\bB_r(p),\pi) \leq C(Q,m,n,\overline{n}) r^{\alpha}, \mbox{ for every }r\in (0,1).
\end{equation*}
\end{theorem}
\begin{proof}
It is enough to show that there exists $\varepsilon$ such that if \begin{equation*}
\max\left\{\overline{\mathbf{E}}(T,\bB_1(p),\pi), \mathbf{A}_{\Gamma},\mathbf{A}_{\Sigma}^2\right\}<\varepsilon
\end{equation*}
then
\begin{equation*}
    \hat{\mathbb{E}}(T,\bB_{1/2}(p),\pi)<\varepsilon_{\textup{decay}}.
\end{equation*}
This follows by compactness. Since $Q\a{\pi}$ is a tangent cone, we obtain the decay towards $Q\a{\pi}$.

We thus establish for $r< 1/8$ that
\begin{equation*}
\overline{\mathbf{E}}(T,\bB_{r}(p),\pi) \lesssim \mathbf{E}(T,\bB_{2r}(p),\pi)+ \mathbf{A}_{\Gamma}+\mathbf{A}_{\Sigma}^2 \lesssim \hat{\mathbb{E}}(T,\bB_{4r}(p),Q\a{\pi})\lesssim r^{\alpha}.
\end{equation*}
The first inequality is the classical $L^2$- tilt excess bound. The second inequality is the strong excess to $L^2$ bound. The last inequality is the decay to the tangent cone from Theorem \ref{t:uniquetangentcone}.
\end{proof}

\end{section}
\begin{section}{Lipschitz approximation}
In this section, we construct the Lipschitz approximation and conclude with some technical estimates regarding measure theoretic excess and its relationship to $L^2$ excess.

We introduce an assumption that enables the construction of a suitable Lipschitz approximation defined on the ball $\bB_1$.

\begin{assumption}\label{A:Lipschitzapproximation}Let $T$, $\Gamma$ and $\Sigma$
 be as in Assumption \ref{A:general}, but instead considered within the ball $\bB_8$. Suppose $C$ is an admissible open book. We assume for a small constant $\varepsilon(Q,m,n,\overline{n})$ that
 \begin{equation*}
 \mathbb{E}(T,C,\bB_8)+ \mathbf{A}_{\Gamma}+\mathbf{A}_{\Sigma}^2 \leq \varepsilon \alpha(C)^2.
\end{equation*}
Moreover, we assume that for every $p \in \Gamma \cap \bB_2$ and all $0<r<1/8$
\begin{equation*}
|\!|T|\!|(\bB_r(p)) \leq \left(Q/2 +1/8\right) \omega_m r^m.\end{equation*}
 \end{assumption}
 The Lipschitz approximation will be defined on the ball $\bB_1.$
\begin{subsection}{Lipschitz approximation away from the spine}
We cite the relevant results on the Lipschitz approximation from \cite{de2023fineIII} and outline the necessary modifications. The following result follows from the $L^2-L^\infty$ height bound with respect to parallel planes (Theorem 3.2 of \cite{de2023fineIII}).

\begin{lemma}[Crude splitting]\label{l:crudesplitting}
Let $T$ and $\Sigma$ be as in Assumption \ref{A:general}. Let $C$ be an open book with $\partial C=Q\a{V}$ with $0 \in V$. Assume that  $\partial T \res \left(\bB_4 \setminus B_{1/32}(V)\right)=0$. There are constants $\delta(Q,m,n,\overline{n})>0$ and $\rho(Q,m,n,\overline{n})>0$ with the following property. Assume that
\begin{equation*}
\int_{\bB_4 \setminus B_{1/32}(V)}\dist(p,C)^2
+\mathbf{A}_{\Sigma}^2<\delta^2 \alpha(C)^2
\end{equation*}
Then
\begin{enumerate}
\item The sets $W_i:= \left(\bB_4 \setminus B_{1/32}(V) \right) \cap \left\{\dist(. ,H_i) <\rho\sigma\right\}$ are pairwise disjoint.
\item \begin{equation*}
    \textup{spt}(T) \cap \bB_{3} \setminus B_{1/16}(V)\subseteq \cup_i W_i.
\end{equation*}
\end{enumerate}
\end{lemma}
The proof of this lemma is the same as in Lemma 8.5 in \cite{de2023fineIII} with the additional information that the boundary does not appear at the desired scales.

We will also denote
$W_i':=\bB_{5/2}\setminus B_{1/20}(V) \cap \left\{ \dist(.,H_i)<2\rho\sigma\right\}$, $\Omega_i'=\bB_{5/2} \setminus B_{1/20}(V)$ and we suppose that $\rho$ is such that $W_i'$ are also disjoint (by possibly taking $\delta$ smaller).
\begin{proposition}\label{p:lipschitzapprox} Let $T$, $C$, $W_i$ as in Lemma \ref{l:crudesplitting}. Consider $\Omega_i:= \bB_2 \cap H_i \setminus B_{1/16}(V)$  and $\mathbf{\Omega_i}=\bB_3\cap \mathbf{p}_{H_i}^{-1}(\Omega_i).$ Set $T_i:=T \res \left(W_i \cap \mathbf{\Omega_i} \right)$ 
and 
\begin{equation*}
E_i:= \int_{\bB_3 \setminus B_{1/32}(V)}\dist(p,H_i)^2 d|\!|T_i|\!|(p).
\end{equation*}

Assume that $\left(\pi_{\circ}\right)_{\#}(T \res \bB_4 \setminus B_{1/32}(V))=Q \a{\bB_4 \cap \left\{x_m>1/32\right\}}$ and $|\!|T|\!|(\bB_4) \leq \left(Q/2+1/4\right)4^{m}\omega_m$.

Then, there exist non-negative integers $Q_1',...,Q_k'$ such that $\sum Q_i'=Q$ and the following properties hold:
\begin{enumerate}
\item $\partial T_i \res \mathbf{\Omega}_i=0.$;
\item $\left(\mathbf{p}_{H_i}\right)_{\#}(T_i)=Q_i'\a{\Omega_i}$;
\item The following estimate holds 
\begin{equation*}
\dist^2(q,H_i)=|\mathbf{p}_{H_i}^{\perp}(q)|^2 \lesssim E_i + \mathbf{A}_{\Sigma}^2 \; \forall q \in \textup{spt}(T_i) \cap \mathbf{\Omega_i}:
\end{equation*}
\item For all $i$ with $Q_i' \geq 1$, there are Lipschitz multi-valued maps $u_i: \Omega_i \rightarrow A_{Q_i'}(H_i^{\perp})$ and closed sets $K_i \subseteq \Omega_i$ such that $\textup{gr}(u_i) \subseteq \Sigma$, $T_i \res \mathbf{p}_{H_i}^{-1}(K_i)=\mathbf{G}_{u_i}\res \mathbf{p}_{H_i}^{-1}(K_i)$ and the following estimates hold
\begin{equation*}
\left\|u_i\right\|_{\infty}^2+\left\|Du_i\right\|_{L^2}^2 \leq CE_i + C\mathbf{A}_{\Sigma}^2
\end{equation*}
\begin{equation*}
\textup{Lip}(u_i) \leq C\left(E_i+\mathbf{A}_{\Sigma}^2\right)^{\gamma}
\end{equation*}
\begin{equation*}
|\Omega_i \setminus K_i| + |\!|T|\!|(\mathbf{\Omega_i} \setminus \mathbf{p}_{H_i}^{-1}(K_i))\leq C\left(E_i+\mathbf{A}_{\Sigma}^2\right)^{1+\gamma}
\end{equation*}
\item $Q_i'=0$ if and only if $T_i=0$;
\item Finally, if in addition
\begin{equation*}
\mathbb{E}(T,C,\bB_8)< \delta^2 \alpha(C)^2
\end{equation*}
then $Q_i'=Q_i$ for every $i$.
\end{enumerate}
\end{proposition}
\begin{proof}
The proof of the common statements follows \cite{de2023fineIII}. The additional properties we establish are the bound $\sum Q_i' \geq Q$ and (6). 

The bound $\sum Q_i' \geq Q$ is an immediate consequence (coupled with the information coming from the same proposition in \cite{de2023fineIII}) of the assumption on the circular projection
\begin{equation*}
\left(\pi_{\circ}\right)_{\#}(T \res \bB_4 \setminus B_{1/32}(V))=Q \a{\bB_4 \cap \left\{x_m>1/32\right\}}.
\end{equation*}
It remains to show (6).

We know that
\begin{equation*}
|T_i|(\mathbf{\Omega}_i)-Q_i|\Omega_i| \lesssim \left(E_i +\mathbf{A}_{\Sigma}^2\right)^{1+\gamma}.
\end{equation*}
In particular 
\begin{equation*}
|T_i|(\mathbf{\Omega}_i)-Q_i|\Omega_i|<<1/2 |\Omega_i|.
\end{equation*}

Assume that $Q_i'>Q_i$ for some $i$, then $|T|(\Omega_i)-|C|(\Omega_i') \geq c(m)$ if excess is small enough.

Assume that $|T|=|T|^1+|T|^2$ and $|C|=|C|^1+|C|^2$ with a transport plan $\sigma$ between $|T|^1$ and $|T|^2$. We recall that $W_i'$ is a fattening of $W_i$ and $\Omega_i'$ is a fattening of $\Omega_i$ and thus exiting $\Omega_i'$ from $\Omega_i$ means traveling a constant distance.

Since \begin{equation*}
    \int \dist(x,V)^2d|T|^2+\int \dist(x,V)^2d|C|^2<<1
\end{equation*} then we must have $||T|^1-|T||(\Omega_i')<<1$ and $||C|^1-|C||(\Omega_i')<<1$. Thus 

\begin{equation*}
c(m) \leq |T|(\Omega_i)-|C|(\Omega_i') \lesssim |T|^1(\Omega_i)-|C|^1(\Omega_i')\leq
\sigma(\mathbf{\Omega_i} \cap W_i,4B \setminus \mathbf{\Omega_i}' \cap W_i') .
\end{equation*}
Since $|x-y| \geq c(m)\alpha(C)$ for $x \in \mathbf{\Omega_i} \cap W_i$ and $y \in 4B \setminus \mathbf{\Omega_i}' \cap W_i' $ we have
\begin{equation*}
c(m) \alpha(C) \leq \int |x-y|^2 d\sigma(x,y) \leq \mathbb{E}(T,C,\bB_4).
\end{equation*}
This is a contradiction to the assumption that $Q_i'>Q_i$ for some $i$ which implies $Q_i'=Q_i$ and thus we conclude (6).
\end{proof}

Notice that we don't separate between the case of a single and multiple planes (as in \cite{de2023fineIII}) since we define $\alpha(C):=1$ when $C$ is a single half-plane.
\end{subsection}

\begin{subsection}{Whitney regions and dyadic cubes}

To construct the Lipschitz Approximation, we define a suitable Whitney region. We assume $V=\RR^{m-1} \times \left\{0\right\}$. We let $L_0$ be a cube with side length $\frac{2}{\sqrt{m-1}}$ centered at $0$.

We define the regions as in \cite{de2023fineIII}. Let $R:=\left\{p: \mathbf{p}_{V}(p) \in L_0 \: \textup{and} \: 0<|\mathbf{p}_{V^{\perp}}(p)| \leq 1 \right\}$. The cubes $\mathcal{G}_l$ are cubes of sidelength $\frac{2^{1-l}}{\sqrt{m-1}}$ obtained by subdividing $L_0$ into $2^{l(m-1)}$ cubes. The integer $l$, the generation of the cube, is denoted by $l(L)$. If $L \subset L^{\prime}$ and $\ell\left(L^{\prime}\right)=\ell(L)+1$, we then call $L^{\prime}$ the parent of $L$, and $L$ a child of $L^{\prime}$. When $\ell\left(L^{\prime}\right)>\ell(L)$, we say that $L^{\prime}$ is an ancestor of $L$ and $L$ is a descendant of $L^{\prime}$.

For every cube $L$ we denote by $y_L$ it's center and $\mathbf{B}(L)$ the ball $\mathbf{B}_{2^{2-l(L)}}(y_L)$ (in $\RR^{m+n}$) and by $\mathbf{B}^h(L)$ the set $\mathbf{B}(L) \setminus B_{2^{-5-l(L)}}(V)$.

We define for $L \in \mathcal{G}_l$
\begin{equation*}
R(L):=\left\{p: \mathbf{p}_{V}(p) \in L \: \textup{and} \: 2^{-l-1} \leq |p_{V^{\perp}}(p)| \leq 2^{-l} \right\}.
\end{equation*}
We also define $\lambda L$ as the cube concentric with $L$ with side length $\lambda \frac{2^{1-l}}{\sqrt{m-1}}$ and
\begin{equation*}
\lambda R(L):=\left\{p: \mathbf{p}_{V}(p) \in \lambda L \: \textup{and} \: \lambda^{-1}2^{-l-1} \leq |p_{V^{\perp}}(p)| \leq \lambda 2^{-l} \right\}.
\end{equation*}
We also define $\lambda L_i:=\lambda R(L) \cap H_i$.
We define the circular projection $\pi_{\circ}: \RR^{m+n} \rightarrow \RR^{m-1} \times \RR^{+}$ as $\pi_{\circ}(q):=(\mathbf{p}_{V}(q),|\mathbf{p}_{V^{\perp}}(q)|)$.
We remind the basic properties of the regions.

\begin{lemma}
Consider the collection of cubes $\mathcal{G}$ introduced above and its elements L. Then the following properties hold:
\begin{enumerate}

\item Given any pair of distinct $L, L^{\prime} \in \mathcal{G}$ the interiors of $R(L)$ and $R\left(L^{\prime}\right)$ are pairwise disjoint and $R(L) \cap R\left(L^{\prime}\right) \neq \emptyset$ if and only if $L \cap L^{\prime} \neq \emptyset$ and $\left|\ell(L)-\ell\left(L^{\prime}\right)\right| \leq 1$, while the interiors of $L$ and $L^{\prime}$ are disjoint if $\ell(L) \leq \ell\left(L^{\prime}\right)$ and $L^{\prime}$ is not an ancestor of $L$.
\item The union of $R(L)$ ranging over all $L \in \mathcal{G}$ is the whole set $R$.
\item  The diameters of the sets $L, R(L), \lambda L, \lambda R(L), L_i, \lambda L_i$, and $\mathbf{B}^h(L)$ are all comparable to $2^{-\ell(L)}$ and, with the exception of $L, \lambda L$, all comparable to the distance between an arbitrarily element within them and $V$; more precisely, any such diameter and distance is bounded above by $C 2^{-\ell(L)}$ and bounded below by $C^{-1} 2^{-\ell(L)}$ for some constant $C$ which depends only on $m$ and $n$.
\item There is a constant $C=C(m, n)$ such that, if $\mathbf{B}^h(L) \cap \mathbf{B}^h\left(L^{\prime}\right) \neq \emptyset$, then $\left|\ell(L)-\ell\left(L^{\prime}\right)\right| \leq$ $C$ and $\operatorname{dist}\left(L, L^{\prime}\right) \leq C 2^{-\ell(L)}$. In particular, for every $L \in \mathcal{G}$, the subset of $L^{\prime} \in \mathcal{G}$ for which $\mathbf{B}^h(L)$ and $\mathbf{B}^h\left(L^{\prime}\right)$ have nonempty intersection is bounded by a constant.
\item $\sum_{L \in \mathcal{G}_{\ell}} \mathcal{H}^{m-1}(L)=C(m)$ for any $\ell$ and therefore, for any $\kappa>0$,
\begin{equation*}
\sum_{L \in \mathcal{G}} 2^{-(m-1+\kappa) \ell(L)} \leq C(\kappa, m) .
\end{equation*}

\end{enumerate}
\end{lemma}
We cite the Layer subdivision of \cite{de2023fineIII}, which will determine a sequence of subcones for a given open book $C$.

\begin{lemma}[Layer subdivision]\label{l:layersubdivision} For every integer $N \geq 2$ and every $0<\delta \leq 1$, there is $\eta=\eta(\delta, N)>0$ with the following properties. Let $C$ be an open book. Then, there is $\kappa \in \mathbb{N}$ and subcollections $I(0) \supsetneq I(1) \supsetneq \cdots \supsetneq I(\kappa)$ of the integers $\{1, \ldots, N\}$, each of cardinality at least 2 and with $I(0)=\{1, \ldots, N\}$, so that the numbers
$$
m(s):=\min _{i<j \in I(s)} \angle\left(H_i,H_j\right)
$$

$$
\begin{aligned}
d(s) & :=\max _{i \in I(0)} \min _{j \in I(s)} \angle\left(H_i,H_j\right) \\
M(s) & :=\max _{i<j \in I(s)} \angle\left(H_i,H_j\right)
\end{aligned}
$$
satisfy the following requirements
\begin{enumerate}
\item $M(\kappa)=M(0)$
\item $\eta M(\kappa) \leq m(\kappa)$
\item $d(s) \leq \delta m(s)$ and $\eta d(s) \leq m(s-1)$ for every $1 \leq s \leq \kappa$;
\item $m(s-1) \leq \delta m(s)$ for every $1 \leq s \leq \kappa$.
\end{enumerate}
If $M(\kappa)<\delta$, we extend the above by $\overline{\kappa}=\kappa+1$ and $I(\overline{\kappa}):=\min \left\{ I(\kappa)\right\}$. Otherwise $\overline{\kappa}:=\kappa.$
\end{lemma}

 \begin{remark}\label{rmk:multiplicities} We can assign multiplicities to the sequence of cones $C_0 \supsetneq C_1 \supsetneq... \supsetneq C_{\overline{k}}$, given by the layering procedure Lemma \ref{l:layersubdivision}.
We write 
\begin{equation*}
C_{k}=\sum_{i \in I(k)}Q_i^{k}\a{H_i}.
\end{equation*}
We define inductively starting from assigning the multiplicities for $I(0)$. We define for $i \in I(k+1)$, the multiplicity $Q_i^{k+1}$ as the sum of $Q_j^k$ over $j \in I(k)$ such that
\begin{equation*}
\angle \left(H_j, H_i\right)= \min_{l \in I(k+1)} \angle \left(H_l, H_i\right).
\end{equation*}
By construction, this minimum will only be achieved by a single sheet of the cone which means that the multiplicity arrangement is well-defined. 
 \end{remark}
The multiplicities in remark \ref{rmk:multiplicities} are defined to be used for Lemma \ref{l:reverseexcessbound} but in the other statements the cones are considered without having an assigned multiplicity.

\begin{assumption}[Selection of parameters]
We fix $\tau(m,n,\overline{n},Q)$ small enough. The parameter $\overline{\delta}=\overline{\delta}(m,n,\overline{n},Q,\tau)$ must be chosen much smaller than $\tau$. Finally $\varepsilon=\varepsilon(Q,m,n,\overline{\delta},\tau)$ in Assumption \ref{A:Lipschitzapproximation} will be smaller $\overline{\delta}.$
\end{assumption}
The parameters are chosen to satisfy lemma 8.3 of \cite{de2023fineIII}.
We define for the open books $C_k$ given by 
\begin{equation*}
    \mathbf{E}(L,k):=2^{(m+2)l(L)} \int_{\bB^h(L)}\dist^2(q,C_{k})d|\!|T|\!|(q).
\end{equation*} Given an open book $C_k$ we define 
\begin{equation*}
\mathbf{s}(k):=\min_{i<j \in I(k)}\dist(H_i \cap B_1,H_j \cap B_1)
\end{equation*}
and $\mathbf{s}(k)=\overline{\delta}$ if $I(k)$ is a singleton.
We will have a division into outer, central and inner cube as in \cite{de2023fineIII} with the addition the respective balls don't intersect the boundary. 
\begin{definition}\label{d:cubestypes} Let $L \in \mathcal{G}$. We say that $L$
\begin{enumerate}
\item $L$ is an \textit{interior cube} if $\mathbf{B}^h(L') \cap \Gamma=\varnothing$
for every ancestor $L'$ of $L$  (including $L$) and 
\begin{equation}\label{eq:densitycondition}
|T|(\mathbf{B}(L)) \leq (Q/2+1/2)2^{m(2-l(L))}.
\end{equation}
\item $L$ is a \textit{boundary stopping cube} if $\mathbf{B}^h(L) \cap \Gamma \neq \varnothing$ or $|T|(\mathbf{B}(L)) >(Q/2+1/2)2^{m(2-l(L))}$ and it's parent is an interior cube.
\item A cube $L$ is a $k$-type cube if it is an interior cube and $\mathbf{E}(L',0)/\mathbf{s}(k)^2 \leq \tau^2$ for every ancestor $L'$ of $L$ (including $L$), but it is not a $(k-1)$-type cube (if $k=0$ this condition is not apply). 
\item A cube $L$ is an \textit{outer cube} if it is a $0$-type cube. We denote by $\mathcal{G}^o$ the outer cubes. A cube $L$ is a \textit{central} cube if it is $k$-type for some $1 \leq k \leq \overline{\kappa}.$ We denote by $\mathcal{G}^c$ the central cubes. 
\item A cube $L$ is a \textit{inner cube} if it is an interior cube, it is neither an outer nor a central cube, but its parent is an outer or a central cube.
\end{enumerate}
\end{definition}
We note that for every interior cube ${\pi_{\circ}}_{\#}\left(T\res_{\mathbf{B}^h(L')}\right)=Q\a{\pi_{\circ}(\mathbf{B}^h(L')))}$.
If the density condition \eqref{eq:densitycondition} fails then we must have  $\dist(\Gamma,y_{L}) \leq C(m) \ell(L)$ since otherwise
\eqref{eq:densitycondition} would be obtained as a consequence of the monotonicity formula at a boundary point.

We define the following regions that will be useful for the non-concentration estimates.

\begin{itemize}
    \item The \textit{boundary region} denoted $R^b$ which is the union of $R(L)$ over all cubes which are not interior cubes.
    \item The \textit{outer region} denoted $R^o$, is the union of $R(L)$ over all outer cubes.
    \item The \textit{central region} denoted $R^c$, is the union of $R(L)$ over all central cubes.

    \item The \textit{inner region} denoted $R^i$, is the union of $R(L)$ over all interior cubes which are not outer or central cubes.
    \end{itemize}
    \end{subsection}
   \begin{subsection}{Local lipschitz approximation and coherent outer approximation} 
We conclude this section with local Lipschitz approximations and a coherent outer approximation defined on the outer region. The proofs are the same as in \cite{de2023fineIII} since the definition of these regions is such that they don't encounter the boundary and satisfy the appropriate upper density estimate.

We obtain the local approximations by  \ref{p:lipschitzapprox} for the current $T_{y_L,2^{-\ell(L)}}$ and the cone $C_{k(L)}$. The domains of definition will be $\Omega(L):=\left(\bB_{2^{1-\ell(L)}}(y_L) \setminus \overline{B}_{2^{-4\ell(L)}}(V)\right)\cap C_{k(L)}.$  We also denote $\Omega_i(L):=\Omega(L) \cap H_i$, $u_{L,i}$ the corresponding multi-valued approximation (obtaining by translating and rescaling to use the appropriate lemma) $\mathbf{\Omega}_i(L)$ the sets $2^{\ell(L)}\mathbf{\Omega}_i+y_L$ and by $K_i(L)$ the sets $2^{-\ell(L)}K_i+y_L$, where $\mathbf{\Omega}_i$ and $K_i$ are given by proposition \ref{p:lipschitzapprox}. When $Q_{L,i}=0$, the map $u_{L,i}$ does not exist and we will not change the notation to take that into consideration.
\begin{proposition}\label{p:localapprox} 
Let $T, \Sigma, \Gamma, C$ as in Assumption \ref{A:Lipschitzapproximation}.
There exists $1<\lambda=\lambda(m) \leq \frac{3}{2}$ and $\overline{C}=\overline{C}(Q,m,n,\overline{n})$ such that the local approximations satisfy the following properties.
\begin{enumerate}
\item $Q_{L,i}=Q_i$ for every $L \in \mathcal{G}^o$.
\item $\sum_{i=1}^N Q_{L, i}=Q$ for every $L \in \mathcal{G}^o \cup \mathcal{G}^c$.
\item For every $L \in \mathcal{G}^o \cup \mathcal{G}^c$ we have $\operatorname{spt}(T) \cap \lambda R(L) \subset \bigcup_i \boldsymbol{\Omega}_i(L)$ and
$$
2^{2 \ell(L)}\left|q-\mathbf{p}_{H_i}(q)\right|^2 \leq \bar{C}\left(\mathbf{E}(L)+2^{-2 \ell(L)} \mathbf{A}_{\Sigma}^2\right) \quad \forall q \in \operatorname{spt}(T) \cap \boldsymbol{\Omega}_i(L) ;
$$
\item 
For $L \in \mathcal{G}^o \cup \mathcal{G}^c$, if we set
$$
T_{L, i}:=T\left\llcorner\boldsymbol{\Omega}_i(L) \cap\left\{\operatorname{dist}\left(\cdot, H_i\right)<\bar{C} 2^{-\ell(L)}\left(\mathbf{E}(L)+2^{-2 \ell(L)} \mathbf{A}_{\Sigma}^2\right)^{1 / 2}\right\}\right.
$$
(with $\bar{C}$ larger than the constant in the previous estimate, then, for $K_i(L)=2^{-\ell(L)} K_i+$ $y_L \subset \Omega_i(L)$ as above,
$$
T_{L, i}\left\llcorner\mathbf{p}_{H_i}^{-1}\left(K_i(L)\right)=\mathbf{G}_{u_{L, i}}\left\llcorner\mathbf{p}_{H_i}^{-1}\left(K_i(L)\right),\right.\right.
$$
gr $\left(u_{L, i}\right) \subset \Sigma$, and the following estimates hold:
$$
\begin{aligned}
2^{2 \ell(L)}\left\|u_{L, i}\right\|_{\infty}^2+2^{m \ell(L)}\left\|D u_{L, i}\right\|_{L^2}^2 & \leq \bar{C}\left(\mathbf{E}(L)+2^{-2 \ell(L)} \mathbf{A}_{\Sigma}^2\right) \\
\operatorname{Lip}\left(u_{L, i}\right) & \leq C\left(\mathbf{E}(L)+2^{-2 \ell(L)} \mathbf{A}_{\Sigma}^2\right)^\gamma \\
\left|\Omega_i(L) \backslash K_i(L)\right|+\left\|T_{L, i}\right\|\left(\boldsymbol{\Omega}_i(L) \backslash \mathbf{p}_{H_i}^{-1}\left(K_i(L)\right)\right) & \leq \bar{C} 2^{-m \ell(L)}\left(\mathbf{E}(L)+2^{-2 \ell(L)} \mathbf{A}_{\Sigma}^2\right)^{1+\gamma} .
\end{aligned}
$$
\end{enumerate}
\end{proposition}

\begin{definition}
Let $T, \Sigma, \Gamma, C$ be as in Proposition \ref{p:localapprox}. Fix $L \in \mathcal{G}^o$. We define
\begin{equation*}\overline{\mathbf{E}}(L):=\max \left\{\mathbf{E}\left(L^{\prime}\right): L^{\prime} \in \mathcal{G}^o \: \textup{with} \: R(L) \cap R\left(L^{\prime}\right) \neq \emptyset\right\}
\end{equation*}
\end{definition}

\begin{proposition}[Coherent outer approximation] Let $T, \Gamma, \Sigma, C$ be as in Assumption \ref{A:Lipschitzapproximation}. We define
\begin{equation*}
R_i^o:= \cup_{L \in \mathcal{G}^o}L_i \equiv H_i \cap \cup_{L \in \mathcal{G}^o} R(L).
\end{equation*}
Then, there exist Lipschitz multivalued maps $u_i: R_i^o \rightarrow A_{Q_i}(H_i^{\perp})$ and closed subsets $\overline{K_i}(L) \subset L_i$ satisfying the following properties:
\begin{enumerate}
\item $gr(u_i) \subset \Sigma$ and $T_{L,i} \res \mathbf{p}_{H_i}^{-1}(\overline{K}_i(L))=\mathbf{G}_{u_i} \res \mathbf{p}_{H_i}^{-1}(\overline{K}_i(L))$ for every $L \in \mathcal{G}^o$, where $T_{L,i}$ is defined as in the proposition above.
\item \begin{equation*}
2^{2l(L)}\left\|u_i\right\|_{L^\infty(L_i)}^2+2^{ml(L)}\left\|Du_i\right\|_{L^2(L_i)}^2 \lesssim \overline{\mathbf{E}}(L)+ 2^{-2l(L)}\mathbf{A}_{\Sigma}^2
\end{equation*}
\begin{equation*}
\left\|Du_i \right\|_{L^{\infty}(L_i)} \lesssim \left(\overline{\mathbf{E}}(L)+\mathbf{A}_{\Sigma}^2\right)^{\gamma}
\end{equation*}
\begin{equation*}
|L_i \setminus K_i(L)| + |\!|T_{L,i}|\!|(\mathbf{p}_{H_i}^{-1}(L_i \setminus K_i(L))\lesssim 2^{-ml(L)}\left(\overline{\mathbf{E}}(L)+2^{-2l(L)}\mathbf{A}_{\Sigma}^2\right)^{1+\gamma}
\end{equation*}
\end{enumerate}
\end{proposition}

\begin{proposition}[Blowup]\label{p:blowup}
Let $T, \Sigma, \Gamma, C$ satisfy Assumption \ref{A:Lipschitzapproximation}. Then, for every $\sigma, \varsigma>0$ there are constants $C=C(m,n,Q,\overline{\delta},\tau)>0$ and $\varepsilon=\varepsilon(m,n,Q,\overline{\delta},\tau,\sigma,\varsigma)>0$ such that the following holds:
\begin{enumerate}
\item $R \setminus B_{\sigma}(V) \subset R^o$ 
\item $u_i$, $R_i:=\left(R \setminus B_{\sigma}(V)\right) \cap H_i$ then 
\begin{equation*}
\int_{R_i}|Du_i|^2 \leq C \sigma^{-2}\mathbf{E}(T,C,\bB_4)+\mathbf{A}_{\Sigma}^2
\end{equation*}
 \item If additionally $\mathbf{A}_{\Sigma}^2 \leq \varepsilon^2 \mathbf{E}(T,C,\mathbf{B}_4)$, then there is a map $w_i: R_i \rightarrow A_{Q_i}(H_i^{\perp})$ which is Dir-minimizing and such that 
 \begin{equation*}
 d_{W^{1,2}}(v_i,w_i) \leq \varsigma
 \end{equation*}
where $d_{W^{1,2}}$ is the $W^{1,2}$ distance defined in \cite{DS1}.
\end{enumerate}
\end{proposition}
\end{subsection}\begin{subsection}{Technical estimates on measure theoretic excess}
In the following lemma, we use the definitions of the different type of cubes as \ref{d:cubestypes}, but do not require the global Assumption \ref{A:Lipschitzapproximation}. We still consider the definition of types of cubes as above with respect to the cone $C$ and their associated Lipschitz approximation. 
\begin{lemma}\label{l:reverseexcessbound}Let $T, \Gamma, \Sigma$ be as in Asumption \ref{A:general}, but instead within $\bB_4$. Assume that every cube that intersects $\bB_{2^{-10}}$ and their ancestors are outer cubes whose multiplicities of their local Lipschitz approximations agree with those of $C$. Then
\begin{equation}
\mathbb{E}(T,C,\bB_1) \leq  C(Q,m,n,\overline{n},\tau,\overline{\delta}) \left(\mathbf{E}(T,C,\bB_4) + \mathbf{A}_{\Sigma}^2+\mathbf{A}_{\Gamma}^2\right) 
\end{equation}
\end{lemma}
\begin{proof}
Let $\varphi_{L}$ be a partition of unity such that $\varphi_{L}$ is supported in $\mathbf{B}^h(L)$.
We can split into cubes $L$
\begin{equation*}
    \mathbb{E}(T,C,\bB_1) \leq \sum_{L \textup{ interior}} d(\varphi_{L} |T|,\varphi_{L} C) +C\mathbf{A}_{\Gamma}^2.
\end{equation*}

In the boundary term we are also including the cubes such that the density 

Now we further split by estimating

\begin{equation}\label{eq:reverseexcessmain}
\sum_{L \; \textup{interior}}d(\varphi_{L} |T|,\varphi_{L} C) \leq (\mathrm{I})+(\mathrm{II}) +(\mathrm{III})
\end{equation}with
\begin{align*}
&(\mathrm{I})=\sum_{L \textup{ inner  cube}} C \ell(L)^{m+2} , \; (\mathrm{II})=\sum_{k}\sum_{L \; k-\textup{type}}d(\varphi_{L} |T|,\varphi_{L} C_L) \; \\&(\mathrm{III})=\sum_{k}\sum_{L \; k-\textup{type}} d(\varphi_{L}C_{L},\varphi_{L}C).
\end{align*}
In the above equations we define the cone $C_{L}$ as the multiplicity arrangement of $C$ given by constructing the local Lipschitz approximations over $C$.

\textbf{We consider the term $(\mathrm{I})$:}

The bound for $\mathrm(I)$ is an immediate consequence of the definition. Indeed for any interior cube $L$ we have 
\begin{equation*}
l(L)^{m+2} \leq C 2^{-(m+2)\ell(L)} \mathbf{E}(L),
\end{equation*}
which implies 
\begin{equation*}
(\mathrm{I}) \leq C \mathbf{E}(T,C,\bB_4).
\end{equation*}

    \textbf{We consider the term $(\mathrm{II}):$}
We obtain, in a way similar to Remark \ref{rmk:strongexcessgraphicalbound},
for $L$ of $k$-type the following inequality
\begin{equation}\label{eq:reverseexcess2}
    d(\varphi_{L} |T|,\varphi_{L}|C_{L}|) \leq C 2^{-(m+2)\ell(L)}\left(\mathbf{E}(L)+2^{-2\ell(L)}\mathbf{A}_{\Sigma}^2\right).
\end{equation}
We can conclude 
\begin{equation*}
(\mathrm{II}) \leq C\mathbf{E}(T,C,\bB_4) +C \mathbf{A}_{\Sigma}^2.
\end{equation*}
 \textbf{We consider the term $(\mathrm{III}):$}
By definition of $k$-type, every ancestor to $L$ must have a bound on the excess $\mathbf{E}(L',0)/s(k)^2 \leq c(k)\tau^2$ and thus the $L^2-L^{\infty}$ height bound is valid with respect to $C_k$ for $L$ and every ancestor of $L$. This coupled with the fact that the regions over which the Lipschitz approximations for $L$ and it's father overlap, implies that they can all be modelled over the cone $C_k$ with the same multiplicities given by
\begin{equation*}
Q_i^k=\sum_{j \in J(i)}Q_{L,j}
\end{equation*}
where 
\begin{equation*}
J(i):=\left\{j \in I(0):  \langle \left(H_j, H_i\right)= \min_{l \in I(k)} \langle \left(H_j, H_l\right) \right\}.
\end{equation*}

Thus, we can bound the distance of the cones by the angle between them
\begin{equation*}
d(\varphi_{L}C_{L},\varphi_{L}C) \leq C \mathbf{s}(k)^2\ell(L)^{m+2} 
\end{equation*}
for any cube of $k$-type. We wish to bound then the contribution of all the cubes of $k$-type.

 We call first $k$-type a central cube which doesn't have an ancestor of $k$-type.

It is easy to see that \begin{equation*}
\sum_{L\, k-\textup{type}}\mathbf{s}(k)^2 \ell(L)^{m+2} \leq \sum_{L\; \textup{ first } k-\textup{type}}\mathbf{s}(k)^2 \ell(L)^{m+2}.
\end{equation*}
This is because the sum of $\ell(L)^{m+2}$ of the descendants of a fixed cube $R$ is comparable to $\ell(R)^{m+2}$. 

Moreover for a first $k$-type cube $L$, by definition of the layer subdivision and being the first of $k$-type (and thus having a father which has large excess relative to $C_{k-1}$) implies
\begin{equation*}
\mathbf{s}(k)^2 \ell(L)^{m+2} \leq C \mathbf{s}(k-1)^2 \ell(L)^{m+2} \leq C 2^{-(m+2)\ell(L)}\mathbf{E}(L).
\end{equation*}

Now by packing the last estimate over all cubes, we get that 
\begin{equation*}
\sum_{L\; \textup{ first } k-\textup{type}}\mathbf{s}(k)^2 \ell(L)^{m+2} \leq C \mathbf{E}(T,C,\bB_4)
\end{equation*}
and thus $(\mathrm{II}) \leq C \mathbf{E}(T,C,\bB_4)$.
The proof of the remaining estimate of $(\mathrm{III})$ follows analogously to the one of $(\mathrm{II})$ by reducing the bound to summing over first $k$-type cubes and using the stopping condition over those. 
\end{proof}
\begin{lemma}\label{l:excessboundscales}
Let $T, \Gamma, \Sigma, C$ be as in Assumption \ref{A:Lipschitzapproximation} for the ball $\bB_1.$ Let $\eta \in (0,1/2)$ and assume that $\varepsilon$ is sufficiently small with respect to $\eta$. For every
$\rho \in [\eta,1/2]$:
 \begin{equation*}
\mathbb{E}(T,C,\bB_{\rho}) \leq C(Q,m,n,\overline{n},\eta) \left(\mathbb{E}(T,C,\bB_1) +\mathbf{A}_{\Sigma}^2+\mathbf{A}_{\Gamma}^2\right).
\end{equation*}

Moreover for every $p \in \bB_{1/4}(0)$ and $\rho \in [\eta,1/2]$:
 \begin{equation*}
\mathbb{E}(T,C_p,\bB_{\rho}(p)) \leq C(Q,m,n,\overline{n},\eta)\left(\mathbb{E}(T,C,\bB_1) +\mathbf{A}_{\Sigma}^2+\mathbf{A}_{\Gamma}^2\right).
\end{equation*}
The open book $C_p$ is obtained by a suitable translation and tilting so that the spine of $C_p$ is tangent to $\Gamma$ at $p$ and $C_p \subset T_p(\Sigma)$.
\end{lemma}

\begin{proof}
The first part of the statement is a consequence of the previous lemma when $\rho \in [\eta,1/4]$. In order to get it to go to $[\eta,1/2]$ we do need to take the constant of Assumption \ref{A:Lipschitzapproximation} smaller (and to construct the Lipschitz approximations in regions that get closer to the boundary of the ball).

The second part follows analogously as in Lemma \ref{l:reverseexcessbound} with some minor technical modifications. When bounding excess relative to a ball centered at $p \neq 0$, the bump function used in the definition of the excess is taken to be centered at $p$ so this requires us to take this into consideration when estimating the transport plan. 
\end{proof}
\end{subsection}
\end{section}
\begin{section}{Simon's non-concentration estimates}
In this section we prove Simon's non-concentration estimates - the key technical machinery we require in this work. 

Theorem 11.2 of \cite{de2023fineIII} goes back to Leon Simon's work on the uniqueness of cylindrical tangent cones. The bound on the error term of the monotonicity formula implies, with delicate technical work, the Simon non-concentration inequality. In our setting, we will also have an error term coming from the fact that we have a boundary. This technique has also been used in \cite{krummel2023analysisI} within the same setting. We will follow \cite{de2023fineIII} taking care of the additional error terms.
Throughout this section let $r=\frac{1}{3\sqrt{m-1}}$. 

\begin{theorem}[Simon's Estimates]\label{t:Simon'sEstimates}
Let $T, \Gamma, \Sigma$ and $C$ as in assumption \ref{A:Lipschitzapproximation} for $\varepsilon$ small enough. For any $r \leq \frac{1}{3\sqrt{m-1}}$ we have
\begin{equation}
\int_{\bB_r} \frac{|q^{\perp}|^2}{|q|^{m+2}} d|\!|T|\!|\leq C(Q,m,n,\overline{n}) \left(\mathbf{E}(T,C,\bB_4)  + \mathbf{A}_{\Gamma}+
\mathbf{A}_{\Sigma}^2\right)
\end{equation}
here $q^{\perp}:=q-\p_{\overrightarrow{T}}(q)$ at $\HH^m$ a.e. $q \in \textup{spt}(T)$ ($\p_{\overrightarrow{T}}$ is the orthogonal projection to the span of $\overrightarrow{T}(q)$).  We also have 
\begin{equation}
\int_{\bB_r}|\mathbf{p}_{V} \circ \mathbf{p}_{T}^{\perp}|^2 d|\!|T|\!| \leq C(Q,m,n,\overline{n}) \left(\mathbf{E}(T,C,\bB_4)  + \mathbf{A}_{\Gamma}+
\mathbf{A}_{\Sigma}^2\right).
\end{equation}
\end{theorem}
\begin{corollary}\label{cor:nonconcentration}
Let $T, \Gamma, \Sigma$ and $C$ as in Assumption \ref{A:Lipschitzapproximation} for $\varepsilon$ small enough. Then for all $\kappa \in (0,m+2)$,
\begin{equation}
\int_{\bB_r}\frac{\dist(q,C)^2}{|q|^{m+2-\kappa}} \leq C_{\kappa} \left(\mathbf{E}(T,\bB_4,C)+\mathbf{A}_{\Gamma}+\mathbf{A}_{\Sigma}^2\right)
\end{equation}
\end{corollary}
The proof of Corollary \ref{cor:nonconcentration} is an immediate consequence of Theorem \ref{t:Simon'sEstimates} and the following Lemma \ref{l:firstvariationestimate}.
\begin{lemma}\label{l:firstvariationestimate}
Let $T, \Gamma, \Sigma$ and $C$ as in assumption \ref{A:Lipschitzapproximation} for $\varepsilon$ small enough. Then for every $\kappa>0$ we have 
\begin{equation}
\int_{\bB_1} \frac{\dist(q,C)^2}{|q|^{m+2-\kappa}}d|\!|T|\!|(q) \leq C_{\kappa} \int_{\bB_1} \frac{|q^{\perp}|^2}{|q|^{m+2}}d|\!|T|\!|(q)+ C_{\kappa} \mathbf{E}(T,C,\bB_1)+\mathbf{A}_{\Gamma}+\mathbf{A}_{\Sigma}^2.
\end{equation}
\end{lemma}

\begin{proof}[Proof of Theorem \ref{t:Simon'sEstimates}]
By the monotonicity formula from Corollary \ref{c:monotonerror}
\begin{align}\label{e:monotonformula}
\int_{\bB_\rho} \frac{|q^{\perp}|^2}{|q|^{m+2}}d|\!|T|\!|(q)= 
\frac{|\!|T|\!|(B_\rho(0))}{\omega_m {\rho}^m}-Q/2-&\frac{1}{m}\int_{B_{\rho}}(q^{\perp}\cdot \overrightarrow{H}_{T}(q))\left(|q|^{-m} - \rho^{-m}\right) \nonumber\\& -Q \int_0 ^{\rho} t^{-m-1} \int_{\Gamma \cap B_t} \left(q \cdot \overrightarrow{\eta} \right)d\HH^{m-1}(q)dt
\end{align}
We let $H_i$ be the half planes $C$ and $Q_i$ their respective multiplicities. We have then that 
\begin{equation*}
\frac{Q\omega_mr^m}{2}= \sum Q_i |H_i \cap B_r|.
\end{equation*}

We later use the following formula, which holds for any Radon measure $\mu$ and any non-negative function $f$:
\begin{equation*}
\int f(|q|) d\mu= \int_0^{\infty}f(r)\frac{d}{dr}\left(\mu(B_r)\right)dr.
\end{equation*}

We start fixing a smooth $\chi$, monotone, non-increasing, such that $\mathbbm{1}_{[0,r]} \leq \chi \leq \mathbbm{1}_{[0,2r)}$.

Define $\Gamma(t):=-\int_s^{\infty} \frac{d}{ds}
(\chi(s)^2)s^mds$. We take a weighted version of the identity \eqref{e:monotonformula} by multiply both sides by $\rho^m$ take the derivative on $\rho$ and then integrate over $\chi(\rho)^2$. We can then get estimate (11.10) in \cite{de2023fineIII} with a boundary term. We explain carefully how to deal with the boundary terms.

We bound

\begin{align}
\label{eq:nonconcentration1}\int_{\bB_r} \frac{|q^{\perp}|^2}{|q|^{m+2}}d|\!|T|\!|(q) \leq C\left [ \int \chi^2(|q|)d|\!|T|\!|(q)- \sum_{i=1}^N Q_i\int_{H_i} \chi^2(|q|)d|H_i| \right ]\\\label{eq:nonconcentration2}+ C \int \frac{|\Gamma(|q|)q^{\perp}\cdot\overrightarrow{H}_T|}{|q|^m}d|\!|T|\!|(q)\\\label{eq:nonconcentration3}+CQ \int_0^{4} \chi^2(|\rho|) \rho^{-1} \int_{\Gamma \cap \bB_{\rho}}(q \cdot \overrightarrow{\eta})d\HH^{m-1}d\rho \\\label{eq:nonconcentration4}+CQ\int_0^{4} \chi^2(|\rho|) \rho^{m-1}\int_0^{\rho} t^{-m-1}\int_{\Gamma \cap \bB_t}(q \cdot \overrightarrow{\eta}) \, d\HH^{m-1}dt d\rho \end{align}
The term \eqref{eq:nonconcentration2} is dealt analogously.
The terms \eqref{eq:nonconcentration3}, \eqref{eq:nonconcentration4} are dealt by observing that 

\begin{equation*}
\int_{\Gamma \cap \bB_{\rho}}|q\cdot \overrightarrow{\eta}| d\HH^{m-1} = \int_{\Gamma \cap \bB_{\rho}}|\mathbf{p}_{T_{q}(\Gamma)^{\perp}}(q) \cdot \overrightarrow{\eta}| d\HH^{m-1} \lesssim \rho^{m-1} \sup_{q \in \Gamma \cap \bB_{\rho}} |\mathbf{p}_{T_{q}(\Gamma)^{\perp}}(q)| \lesssim \rho^{m+1}\mathbf{A}_{\Gamma}.
\end{equation*}

Thus the terms \eqref{eq:nonconcentration3}, \eqref{eq:nonconcentration4} are bounded by $\mathbf{A}_{\Gamma}$.
As for the term \eqref{eq:nonconcentration1}, we take the vector field
\begin{equation*}
X(q)=\chi(|q|) p_{V}^{\perp}(q).
\end{equation*}
The first variation formula gives us 
\begin{equation*}
\int \textup{div}_{\overrightarrow{T}}(X)=\mathbf{O}(\mathbf{A}_{\Sigma}^2+\mathbf{A}_{\Gamma}).
\end{equation*}

This gives us 
\begin{align*}
\int \chi^2(|q|) d|T|(q)- \int \chi^2(|q|)d|C|(q) \leq C \mathbf{A}_{\Sigma}^2 + C\mathbf{A}_{\Gamma}+ C \int _{B_{2r}}|x^{\perp}|^2 d|\!|T|\!|(q)\\+\int \chi(|q|)x \cdot \nabla_{V^{\perp}}\chi(|q|)d|\!|C|\!|(q).
\end{align*}

The proof then proceeds almost identically as in \cite{de2023fineIII} since we have the same estimates for inner, central and outer regions. The only term that remains are the boundary regions which can be bounded by $\mathbf{A}_{\Gamma}.$
\end{proof}
\begin{proof}[Proof of Lemma \ref{l:firstvariationestimate}]
See the proof of lemma 11.6 in \cite{de2023fineIII}. It follows from a first variation argument on the vector field 

\begin{equation*}
X(q):=\dist^2(q,C) \left(\max(r,|q|)^{-m-2-\kappa}-1\right)_{+}q
\end{equation*}
where $g_{+}(q):=\max(g(q),0).$

The only required modification is controlling the boundary term
\begin{equation*}
\int \textup{div}_{\overrightarrow{T}}(X)d|\!|T|\!|= Q \int_{\Gamma} (X \cdot \overrightarrow{\eta}) d\HH^{m-1} - (X \cdot H_{T})d|\!|T|\!|.
\end{equation*}
We estimate as in the boundary error term in the proof of Theorem \ref{t:Simon'sEstimates}
\begin{equation*}
    \int_{\Gamma} (X \cdot \overrightarrow{\eta}) d\HH^{m-1} \leq \int_{\Gamma \cap \bB_1} |q|^{-m-\kappa}|\mathbf{p}_{T_{q}(\Gamma)^{\perp}}(q)\cdot \overrightarrow{\eta}| d\HH^{m-1} \leq C_{\kappa} \mathbf{A}_{\Gamma}.
\end{equation*}
\end{proof}
\begin{theorem}\label{t:nonconcentration}
Let $T,\Gamma,\Sigma, C$ be as in Assumption \ref{A:Lipschitzapproximation}. 
For every $\alpha>0$
\begin{enumerate}
    \item For every cube $L$ in the cubical decomposition from the previous section
\begin{equation*}
\mathbf{E}(L) \lesssim_{\alpha} 2^{-\alpha\ell(L)}  \left(\mathbf{E}(T,C,\bB_4)+\mathbf{A}_{\Gamma}+\mathbf{A}_{\Sigma}^2\right)
\end{equation*}
\item 
    For every $\sigma$ and $r \leq \frac{1}{6\sqrt{m-1}}$:
\begin{equation*}
  \frac{1}{r^{m+2}}  \int_{\bB_r \cap B_{\sigma r}(V)}\dist(q,C)^2 d|\!|T\!|  \leq C_{\alpha}\sigma^{3-\alpha}  \left(\mathbf{E}(T,C,\bB_4)+\mathbf{A}_{\Gamma}+\mathbf{A}_{\Sigma}^2\right).
\end{equation*}
\item Moreover, 
 if $\bB_r \setminus B_{\sigma r}(V) \subset R^{out}$ then 
\begin{align*}
\frac{1}{r^{m+2}}\int_{\bB_r} \dist(q,C)^2 d|\!T\!| \leq C(\alpha,Q,m,n,\overline{n})\left( \mathbf{E}(T,C,\bB_4) + \mathbf{A}_{\Gamma}+\mathbf{A}_{\Sigma}^2 \right) \sigma^{3-\alpha} \\+ \frac{1}{r^{m+2}}\sum_{i}\int_{\bB_r \cap H_i \setminus B_{\sigma r}(V)} |u_i|^2 \\+C(Q,m,n,\overline{n})\left( \mathbf{E}(T,C,\bB_4)+r^2\mathbf{A}_{\Sigma}^2 \right)\left(\mathbf{E}(T,C,\bB_4)\sigma^{-m-2}+r^2\mathbf{A}_{\Sigma}^2\right)^{\gamma}
\end{align*}
\end{enumerate}
\end{theorem}
\begin{proof}
The first inequality is a consequence of Theorem \ref{t:Simon'sEstimates} Simon's non-concentration estimate at suitable boundary points.

We prove the second inequality. We start covering $\bB_r \cap B_{\sigma r} (V)$ with balls of size comparable to $\sigma r$. For every such ball we have \begin{equation*}
    \int_{B_i} \dist(q,C)^2 \leq \left(\sigma r\right)^{m+2-\alpha}\left(\mathbf{E}(T,C,\bB_4)+\mathbf{A}_{\Gamma}+\mathbf{A}_{\Sigma}^2\right).
\end{equation*}
Part of the error, which is already accounted for corresponds to the fact that the earlier non-concentration estimates require the cone to be admissible. We thus have to tilt and translate the cone which introduces an error of $\mathbf{A}_{\Gamma}^2 +\mathbf{A}_{\Sigma}^2$. 

By taking a Besicovitch cover the balls will have bounded overlap. The number of balls will be thus bounded by $C(m)\sigma^{-m+1}$. This implies the second inequality in the lemma.

We estimate the remaining part of the integral in order to conclude the third inequality:

\begin{align*}
&r^{m+2}\mathbf{E}(T,C,\bB_r \setminus B_{\sigma r}(V)) \leq \sum_{L:  \sigma/2 \leq 2^{-\ell(L)} \leq 4r} \sum_{i} \int_{\mathbf{\Omega}(L)}\dist(p,H_i)^2d|T_{i,L}|\\& \leq \sum_{L:  \sigma/2 \leq 2^{-\ell(L)} \leq 4r} \sum_{i} \int_{\Omega_i(L)^2}|u_i|^2+C2^{-(m+2)\ell(L)}\left(\mathbf{E}(L)+2^{-2\ell(L)}\mathbf{A}_{\Sigma}^2\right)^{1+\gamma}\\ &\leq \sum_{i} \int_{\bB_{4r} \cap H_i \setminus B_{\sigma}(V)} |u_i|^2 +  C\left( \mathbf{E}(T,C,\bB_4)+r^2\mathbf{A}_{\Sigma}^2 \right)\left(\mathbf{E}(T,C,\bB_4)\sigma^{-m-2}+r^2\mathbf{A}_{\Sigma}^2\right)^{\gamma}.
\end{align*}
This type of estimate comes from Proposition \ref{p:localapprox}. There is an additional issue we deal with here which is that there is a slight mismatch between the current on the cylindrical region $\mathbf{R}(L)$ and where the graphical approximation is defined. This issue does appear in the above estimate and is addressed in Lemma 8.16 and Remark 8.17 of \cite{de2023fineIII}. 
\end{proof}
\end{section}
\begin{section}{The boundary linear problem}
In this section, we study the boundary linear problem of multi-valued Dirichlet minimizing functions on the half-space. We show this problem to have similar properties to the interior problem of Dirichlet minimizing functions treated in \cite{DS1}. We show the $1$-homogeneous Dir-minimizers to be linear and a suitable decay lemma. We conclude the section with a technical result which will later allow us to conclude that blowups are Dir-minimizing up to the boundary.

\begin{subsection}{Preliminaries}
We refer to \cite{DS1} for the detailed set up for the linear problem. 

Let $\RR_{+}^m=\left\{(y,t) \in \RR^{m-1} \times \RR^+\right\}$. Let $H= \{x: x_m=0\}$ and if $x \in H$ then $\bB_r^+(x)$ denote the half ball $\bB_r(x) \cap \RR^m_+$. We denote $\partial \bB_r^+(x)= \partial \bB_r(x) \cap \{x_m >0\}$. 
Given $\Omega \subset \RR_+^m$ be a Lipschitz domain there is a well-defined continuous trace operator 
\begin{equation*}
    \circ|_{\partial \Omega}:W^{1,2}(\Omega,\mathcal{A}_Q(\RR^n)) \rightarrow  L^2(\partial \Omega, \mathcal{A}_Q(\RR^n)).
\end{equation*}

\begin{definition}
Let $\Omega$ be a Lipschitz domain. We say $u$ is Dir-minimizing in $\Omega$ if
\begin{equation*}
    u \in W^{1,2}(\overline{\Omega},\mathcal{A}_Q(\RR^n))
\end{equation*}
and for every $v \in  W^{1,2}(\overline{\Omega},\mathcal{A}_Q(\RR^n)), \; \textup{we have} \; u \res \partial \Omega = v\res \partial \Omega$  \begin{equation*}
\int_ {\Omega} |Du|^2 \leq \int_{\Omega}|Dv|^2.
\end{equation*}
\end{definition}
The boundary linear problem is the following:
\begin{assumption}[Linear problem]\label{a:linearproblem}
Given a Dir-minimizing function $u \in W^{1,2}(\Omega,\mathcal{A}_Q(\RR^n))$. The boundary linear problem concerns functions $u$ satisfying $u \res \partial \Omega \cap \bB_1=Q\a{0}$.
\end{assumption}

In this work it will be enough for us to consider the linear problem on the half space (i.e., the following):

\begin{assumption}[Linear problem on the half space]\label{a:linearproblemhalfspace}
Given a Dir-minimizing function $u \in W^{1,2}(\bB_1^+,\mathcal{A}_Q(\RR^n))$. The boundary linear problem on the half space consists of functions $u$ such that $u \res H\cap \bB_1=Q\a{0}$.
\end{assumption}

In \cite{Jonas3} Hirsch proved Hölder regularity for Dir minimizers with Hölder regular boundary data on $C^1$ domains. A relevant consequence of his work for us is:
\begin{theorem}\label{t:Höldercontinuitydir}
    Let $u$ be as in Assumption \ref{a:linearproblem} for a $C^1$ domain $\Omega$. There exists $\alpha(Q,m,n)>0$ such that every $r \in (0,1)$,
    $u \in C^{0,\alpha}(\overline{\bB_r \cap \Omega} )$.
\end{theorem}
His work can handle any Hölder continuous boundary data over the boundary of a domain, and even thought the main statement of \cite{Jonas3} is written in a global way, the proof of it is purely local and implies the above theorem.

We follow the main properties of the frequency function as in \cite{DS1}
\begin{definition}
For $x \in H \cap \bB_1$, $0<r<1-|x|$, we define the Dirichlet energy, $L^2$ spherical height and the boundary frequency function:
\begin{equation*}
    D(x,r)=\int_{\bB_r^+(x)} |Du|^2, \; H(x,r)= \int_{\partial \bB_r^+(x)} |u|^2 \; \textup{and} \; I(x,r)= \frac{rD(x,r)}{H(x,r)}
\end{equation*}
\end{definition}

\begin{proposition}\label{Prop:Identities}
Let $u$ be as in  Assumption \ref{a:linearproblemhalfspace}. For every $x \in H \cap \bB_1$ and a.e. $r \in (0,1-|x|)$
\begin{equation*}
(m-2) \int_{\bB_r^+(x)}|Du|^2= r \int_{\partial \bB_r^+(x)} |Du|^2 - 2r\int_{\partial \bB_r^+(x)} \sum_{i} |\partial_{\nu}u_i|^2
\end{equation*}
\begin{equation*}
\int_{\bB_r^+(x)}|Du|^2= \int_{\partial \bB_r^+(x)} \sum\langle  \partial_v u_i , u_i\rangle
\end{equation*}
\begin{proof}
The proof is analogous to the corresponding version for classical Dir-minimizing functions (see Proposition  3.2. \cite{DS1}).
\item  For the first formula we use the same specific families of inner variations, which are diffeomorphisms that move $H$ tangentially (and thus preserve the zero boundary data condition). These families are of the form $\varphi(x)=\phi(|x|)x$.
\item For the second formula we use the same specific families of outer variations. The competitors are of the form 
\begin{equation*}
    \psi_{\varepsilon}(x)=\sum\a{u_i(x)+\varepsilon \phi(|x|)u_i(x)},
\end{equation*} which preserve the zero boundary data condition along $H$ and thus they are a valid $1$-parameter family of competitors.
\end{proof}

\end{proposition}
We can prove with the above the monotonicity of the frequency.
\begin{theorem} \label{t:monotonfrequency} The frequency function for a function $u$ satisfying Assumption \ref{a:linearproblemhalfspace} is monotone. For any $x \in H \cap \bB_1$ either there exists $\rho$ such that $u|_{B_{\rho}(x)}\equiv 0$ or $I(x,r)$ is an absolutely continous non decreasing positive function of $r$ on $(0,1-|x|)$.
\end{theorem}

\begin{corollary}\label{c:monotonfrequency}
We have $I(0,r) \equiv r$ for $r \in (0,1)$ if and only if $u$ is $\alpha$-homogeneous i.e. there is $f$ such that
\begin{equation*}
    u(y)=|y|^{\alpha}f\left(\frac{y}{|y|}\right)
\end{equation*}
Similar statements at different centers and scales can be obtained by suitable rescaling.
\end{corollary}
\begin{corollary}\label{c:consequencesmonotonfrequency}
We set $x=0$ and consider radii up to $1$.
The following estimates hold:
\begin{enumerate}
    \item For almost every $r\leq s \leq t$,
    \begin{equation*}
        \frac{d}{d\tau}|_{\tau=s} \left[\textup{ln}\left( \frac{H(\tau)}{\tau^{m-1}} \right) \right]=\frac{2I(s)}{s}
    \end{equation*}
    and
    \begin{equation*}
        \left(\frac{r}{t}\right)^{2I(t)} \frac{H(t)}{t^{m-1}} \leq \frac{H(\tau)}{r^{m-1}} \leq       \left(\frac{r}{t}\right)^{2I(r)} \frac{H(t)}{t^{m-1}};
    \end{equation*}
    \item if $I(t)>0$, then
    \begin{equation} \label{eq:DirBound}
    \frac{I(r)}{I(t)} \left( \frac{r}{t} \right)^{2I(t)} \frac{D(t)}{t^{m-2}} \leq \frac{D(r)}{r^{m-2}}  \leq \left( \frac{r}{t} \right)^{2I(r)} \frac{D(t)}{t^{m-2}}.
    \end{equation}
\end{enumerate}
\end{corollary}

The proof of Theorem \ref{t:monotonfrequency}, Corollary \ref{c:monotonfrequency} and Corollary \ref{c:consequencesmonotonfrequency} follow analogously doing the same computations as in the interior setting in the Memoir \cite{DS1}.
For example for the proof of Theorem \ref{t:monotonfrequency} we analogously get 
\begin{equation*}
    I'(x,r)=\frac{(2-m)D(x,r)+rD'(x,r)}{H(x,r)}-2r\frac{D(x,r)^2}{H(x,r)^2}
\end{equation*} for a.e. $0<r<1-|x|$ by using identities \ref{Prop:Identities}, which allows us to conclude $I'(x,r)$ is non-negative by Cauchy Schwartz.

\begin{proposition}
The frequency function
\begin{equation*}
    I: H \cap \bB_1 \rightarrow \RR^+ \; \textup{defined by} \; I(x):=\lim_{r \rightarrow 0}I(x,r)
\end{equation*}
is an upper semicontinuous function on $H \cap \bB_1$.
\end{proposition}
\begin{proof}
We take a sequence $x_k \rightarrow x$ and $r_k \rightarrow r$ we further assume $r_k< \frac{1}{2}\left(1-|x| \right)$. For $k$ large enough $r_k<1-|x_k|$, which makes the frequency well-defined.

We know that
\begin{equation*}
    I(x_k,r_k) \rightarrow I(x,r).
\end{equation*} 
By the monotonicity formula \begin{equation*}
I(x,r) \geq \limsup_{k \rightarrow \infty} I(x_k)
\end{equation*}
 for every $r \in (0, 1-|x|)$. By taking $r \rightarrow 0$ we get the desired upper semicontinuity.
\end{proof}

\begin{proposition}\label{p:DirEnergyBound}
If $x \in H$ and $I(x)=\alpha$ then for every radius $r<\frac{1-|x|}{2}$ we have that 
\begin{equation*} 
\int_{\bB_r^+(x)}|Du|^2 \leq C_xr^{m-2+2{\alpha}}.
\end{equation*} 
\end{proposition}
\begin{proof}
We plug $t=r_x:=\frac{1-|x|}{2}$ in \eqref{eq:DirBound} which gives us for
$r \leq r_x$
\begin{equation*}
    \int_{\bB^+_r(x)}|Du|^2 \leq r^{m-2} \times \left(\frac{r}{r_x}\right)^{2I_x(r)} \int_{\bB^+_{r_x}(x)}|Du|^2 \leq C_x r^{m-2+2\alpha}.
\end{equation*}
In the last step we used the monotonicity formula.
\end{proof}\end{subsection}
\begin{subsection}{Blowup of Dir-minimizing $Q$ valued functions}

Let $y \in H \cap \bB_1$ assume that $\textup{Dir}(f,B_{\rho}(y))>0$ for every $\rho \leq 1-|y|$. The blowups of $f$ at $y$ are defined as:
\begin{equation*}
    f_{y,\rho}(x)= \frac{\rho^{\frac{m-2}{2}}f(\rho x+y )}{\sqrt{\textup{Dir}(f,B_{\rho}(y))}}
\end{equation*}

\begin{theorem} 
 Let $f \in W^{1,2}(\bB^+_1,\mathcal{A}_Q)$ be Dir-minimizing with zero boundary data and $\textup{Dir}(f,B_{\rho})>0$ for every $\rho \leq 1$. Then for any sequence $\rho_k \rightarrow 0$ there exists a subsequence, not relabeled, such that $f_{0,\rho}$ coverges locally uniformly and in weak $W^{1,2}(\RR_+^m)$ to a Dir-minimizing function $g:\RR^m_+ \rightarrow \mathcal{A}_Q(\RR^n)$ with the following properties:
 \begin{enumerate}
     \item $\textup{Dir}(g,\bB_1^+)=1$ and $g| \Omega$ is Dir-Minimizing for any bounded $\Omega$.
     \item $g(x)=|x|^{\alpha}g\left(\frac{x}{|x|}\right)$, where $\alpha=I_{0,f}(0)>0$ is the frequency of $f$ at $0$.
 \end{enumerate}
\end{theorem}
The proof is analogous to the one from \cite{DS1}.
\begin{lemma}{(Cylindrical Blowups)}

Let $g \in W^{1,2}(\RR^{m^+},\mathcal{A}_Q(\RR^n))$ be an $\alpha$ homogeneous and Dir-Minimizing function with zero boundary value, $\textup{Dir}(g,\bB_1^+)>0$ and $\beta=I_{z,g}(0)$. Suppose $z=\frac{e_1}{2}$. Then, the tangent functions $h$ to $g$ at $z$ satisfy
\begin{equation*}
h(x_1,x_2,...,x_m)=\hat{h}(x_2,...,x_m)
\end{equation*} where $\hat{h} \in W^{1,2}(\RR^{m-1}_+,\mathcal{A}_Q(\RR^n))$ is a Dir-Minimizing Q valued function with zero boundary values on $\RR^{m-1}_+$
\end{lemma}
The proof is analogous to the interior version of this statement. We do require a small modification for the construction of the competitor, which we will do as part of the next lemma.
We recall here the interpolation result that we will use to construct the competitor in next section.
\begin{proposition}\label{Interpolation}
There is a constant $C=C(m,n,Q)$ with the following property. For every $g \in W^{1,2}(\partial  \bB_1,\mathcal{A}_Q)$, there is $h \in W^{1,2}(\bB_1,\mathcal{A}_Q)$ with $h|_{\partial \bB_1}=g$ and
\begin{equation*}
    \textup{Dir}(h,\bB_1) \leq C \textup{Dir}(g,\partial \bB_1)+C \int_{\partial \bB_1} \mathfrak{g}(g,Q\a{0})^2.
\end{equation*}
\end{proposition}
This is Corollary 2.16 of \cite{DS1}.
\end{subsection}
\begin{subsection}{Classification of 1 homogeneous Dir-minimizers}

\begin{lemma}\label{Lem:Invariance}Let $h$ be a $\alpha$ homogeneous Dir-minimizing function on $\RR^m_+$ with zero boundary value. Suppose that $I(z)=\alpha$ for $z=e_1/2$. Then $h(s,x)=\hat{h}(x)$ for a Dir-minimizing function $\hat{h}$ on $\RR^{m-1}_+$.
\end{lemma}
\begin{proof}
\textbf{We show that $y \rightarrow h(y-z)$ must be $\alpha$ homogeneous.}

We know that for $r \leq t$
\begin{equation*}
\frac{I(z,r)}{I(z,t)}\left(\frac{r}{t}\right)^{2I(z,t)}\frac{D(z,t)}{t^{m-2}}\leq \frac{D(z,r)}{r^{m-2}}.
\end{equation*}
Moreover by the $\alpha$ homogeneity 
\begin{equation*}
\int_{\bB_r(0)}|Dh|^2=r^{m-2+2\alpha}\int_{\bB_1(0)}|Dh|^2.
\end{equation*}
This implies that $D(z,t) \gtrsim t^{m-2+2\alpha}$ for $t \geq 2$. Moreover we also know that for $r \leq 1$
\begin{equation*}
D(z,r) \lesssim r^{m-2+2\alpha}.
\end{equation*}
This implies 
\begin{equation*}
\left(\frac{t}{r}\right)^{2I(z,t)-2\alpha} \lesssim \frac{I(z,t)}{I(z,r)} \lesssim \frac{I(z,t)}{\alpha}.
\end{equation*}
For this inequality to hold we must have $I(z,t)\equiv\alpha$ since when $r \rightarrow 0$ the left hand side blowsup otherwise.

As for the interior we can get invariance of the function on the $z$ direction. If $x \in \RR^{m-1}_+$ then for $s>0$
\begin{equation*}
    h(se_1+x')=(2s)^{\alpha}h \left (\frac{e_1}{2}+\frac{x'}{2s} \right)=h\left (\frac{e_1}{2}+x'\right)=\hat{h}(x').
\end{equation*} For $s<0$ the equality is analogous. This gets us invariance in the direction of $z$. 

\textbf{We show that $\hat{h}$ is Dir-minimizing. }If it is not, there exists a competitor $h'$ on a $m-1$ dimensional half ball $\bB^+$:

\begin{equation*}
\int_{\bB^+}|D\tilde{h}|^2 < \int_{\bB^+}|D\hat{h}|^2 \; \textup{and} \; \tilde{h} \res \partial \bB^+=\hat{h}\res \partial \bB^+.
\end{equation*}
where $h'$ also has zero boundary value along $H$.

We can as in the proof of lemma 3.24 in \cite{DS1} use $\tilde{h}$ to construct a competitor for $h$ in a sufficiently large half cylinder. The proof is analogous but we still give a briefs sketch on the construction of the competitor.

We assume $\bB^+$ is centered at $0$ and is of radius $1$ (up to rescaling and translating). Given $L \in \RR$ we take a competitor $\left\{\tilde{h} \neq h'\right\} \subset [-L,L] \times \bB_{1}^{+}$.
The competitor $h'$ is chosen so that it agrees with $\tilde{h}$ in $[-L+1,L-1] \times \bB_1{+}$. 

The transition regions are two half-cylinders of the form $[0,1] \times \bB_1^{+}$. In order to get the interpolating functions over the cylinders, which will be independent of $L$, we can use the interpolation result \ref{Interpolation} to get the desired interpolating function in the transition regions. The fact that the unit ball and the half-cylinder are bi-lipschitz equivalent allows to use the lemma despite it is stated originally for the unit ball.

When $L$ is large enough the gain obtained by using a less energy competitor for the slices dominates the loss of the interpolating function. The precise computation is the same as in lemma 3.24 in \cite{DS1}.
\end{proof}
\begin{lemma}\label{l:classificationof1homogeneous}
Let $u \in W^{1,2}(\RR_+^m,\mathcal{A}_Q(\RR^n))$ be Dir-minimizing with zero boundary value along $H$. Assume $u$ is $\alpha$-homogeneous. Then $\alpha \geq 1$. If $\alpha=1$, then $u$ is linear: $ \exists v_1,...,v_Q \in \RR^n$ such that
\begin{equation*}
u=\sum_{i=1}^Q\a{v_ix_m}.
\end{equation*}
\end{lemma}
\begin{proof}
In order to prove $\alpha \geq 1$, we argue by contradiction. Suppose $\alpha<1$. Given $z \neq 0$ by homogeneity $I(rz)=I(z)$ for all $r>0$. By upper semicontinuity of the frequency $I(z)<I(0)<1$. We take a blowup at $z$ and get an $I(z)$ homogeneous Dir mimizing function on $\RR_+^{m-1}$ with zero boundary data. We can repeat this procedure to get a Dir-minimizing $\beta$ homogeneous function with $\beta <1$ on $\RR^+$ with zero boundary data. The only Dir minimizers with zero boundary data in $1$d are linear functions so $\beta=1$ which reaches a contradiction.

We now show that if $\alpha=1$, then $u$ must be linear. By upper semicontinuity of the frequency $I(z)\leq I(0)=1$ for every $z \in H$. By the last proposition,  $I(z) \geq 1$ for every $z \in H$. Thus $I(z)=1$ for every $z \in H$. 

By Lemma \ref{Lem:Invariance}, the function $u$ is invariant in all directions parallel to the boundary by Lemma \ref{Lem:Invariance}. Thus there exists $u'$ a $1$ dimensional Dir-minimizing function on $\RR^+$ $u(y,t)=u'(t)$. Since all $1$ dimensional Dir-minimizing functions are linear, $u'$ must be linear:
\begin{equation*}
u'(t)=\sum_{i=1}^Q \a{v_it}
\end{equation*} w
for $v_i \in \RR^n$.
Thus we get the desired expression for $u$.
\end{proof}
\end{subsection}
\subsection{A decay lemma}
\begin{proposition}\label{p:heightdecay}
If $I(0)=\alpha$ then
\begin{equation*}
\int_{\bB_r^+}|u|^2 \leq\frac{1}{m+1}r^{m+2\alpha} \int_{\partial\bB_1^+}|u|^2.
\end{equation*}
In particular, if $I(0)>1$ then
\begin{equation*}
\lim_{r \rightarrow 0}\frac{1}{r^{m+2}}\int_{\bB_r^+}  |u|^2=0.
\end{equation*}
\end{proposition}
\begin{proof}
We have that
\begin{equation*}
\int_{\partial \bB_r^+}|u|^2 \leq r^{m-1+2\alpha} \int_{\partial \bB_1^+}|u|^2
\end{equation*}
from  Corollary \ref{c:consequencesmonotonfrequency}.
We integrate over the radius and get since $\alpha \geq 1$
\begin{equation*}
\int_{\bB_r^+}|u|^2 \leq \left(\int_0^r s^{m+1} \int_{\partial \bB_1^+}|u|^2\right) r^{2(\alpha-1)}= \frac{1}{m+1}r^{m+2\alpha} \int_{\partial\bB_1^+}|u|^2.
\end{equation*}
This gives us the desired estimate. The conclusion at the end of the proposition is immediate from the fact that $\alpha>1$ implies the exponent on $r$ is greater than $m+2$.
\end{proof}

The following lemma will allow us to get decay to either a sequence of linear functions or to zero.
\begin{lemma}\label{l:dirminmizersdecaylinearguy}
    Let $u$ be as in Assumption \ref{a:linearproblem}. Then for every sequence of radii $r_j \to 0$, there exists a subsequence, not relabeled, and linear maps $L_j$, each with zero boundary value, such that the following properties hold. We require for every $j$
    \begin{equation*}
        \int_{\bB_1} |D L_j|^2 \leq \int_{\bB_1} |Du|^2.
    \end{equation*}
    If $L_j = \sum_{i=1}^{N} \a{v_{i,j} x_n}$, we define the angle
   \begin{equation*}
    \alpha(L_j) := \min_{\substack{i_1, i_2 \\ v_{i_1,j} \neq v_{i_2,j}}} |v_{i_1,j} - v_{i_2,j}|
\end{equation*}
    If $L_j = Q\a{L'_j}$, then $\alpha(L_j) := 1$. One of the following cases holds:
    \begin{enumerate}
        \item $\liminf_{j \to \infty} \alpha(L_j) > 0$ and
        \begin{equation*}
            \lim_{j \to \infty} \frac{1}{r_j^{m+2}} \int_{\bB_{r_j}} \mathcal{G}(u, L_j)^2 = 0,
        \end{equation*}
        \item $L_j=Q\a{L'}$ and
        
        \begin{equation*}
            \lim_{j \to 0} \frac{1}{r_j^{m+2}} \int_{\bB_{r_j}} \left|u\ominus Q \a{L'}\right|^2 = 0.
        \end{equation*}
    \end{enumerate}
\end{lemma}
\begin{proof}
We start by proving the lemma first for the case that $u$ has zero average. If $I(u)>1$ or $u \equiv 0$ we have by Proposition \ref{p:heightdecay} that
\begin{equation*}
\lim_{j \rightarrow 0} \frac{1}{r_j^{m+2}}\int_{\bB_{r_j}} |u|^2=0
\end{equation*}
so we conclude.

If $I(u)=1$, given the sequence $r_j$, we can choose a subsequence, not relabeled, so that
\begin{equation*}
u_{0,r_j} \rightarrow L
\end{equation*}
is linear average free and with zero boundary value. This follows from Lemma \ref{l:classificationof1homogeneous}, since $L$ must be Dir minimizing $1$-homogeneous with zero boundary value function and thus it is linear. 
We have
\begin{align*}
\int_{\bB_{1}}\mathcal{G}(u_{0,r_j},L)^2&=\int_{\bB_1}\mathcal{G}\left(\frac{r_j^{m-2/2}}{\sqrt{\textup{Dir}(u,\bB_{r_j})}},u(r_jx),L(x)\right)^2dx\\&=\frac{1}{r_j^{m+2}}\frac{r_{j}^m}{\textup{Dir}(u,\bB_{r_j})} \int_{\bB_{r_j}}\mathcal{G}\left(u,\frac{\sqrt{\textup{Dir}(u,\bB_{r_j})}}{{r_j^{m/2}}}L\right)^2.
\end{align*}

We denote \begin{equation*}
    c_j:=\frac{\sqrt{\textup{Dir}(u,\bB_{r_j})}}{{r_j^{m/2}}}.
\end{equation*}
Thus
\begin{equation*}
\lim_{j \rightarrow 0} \frac{1}{r_j^{m+2}}\frac{1}{c_j^2}\int_{ \bB_{r_j}}\mathcal{G}(u,c_jL)^2 =0.
\end{equation*}

We notice that 
\begin{equation*}
\frac{\sqrt{\textup{Dir}(u,\bB_r)}}{r^{m/2}} \leq \sqrt{\textup{Dir}(u,\bB_1)}.
\end{equation*}
since the left hand side is monotone for a Dir-minimizing function with frequency at least $1$ (See Lemma \ref{p:DirEnergyBound}). This implies $c_j \leq \sqrt{\textup{Dir}(u,\bB_1)}$ and thus if $L_j=c_jL$, then $\int_{\bB_1} |DL_j|^2 \leq \int_{\bB_1}|Du|^2.$
We must have by definition that $\alpha(L_j)=c_j\alpha(L).$ 

We must show that if 
\begin{equation*}
\liminf_{j \rightarrow \infty}\alpha(L_j)=0
\end{equation*}
then up to a subsequence
\begin{equation*}
\lim_{j \rightarrow 0} \frac{1}{r_j^{m+2}}\int_{\bB_{r_j}}|u|^2=0.
\end{equation*}
This comes from
\begin{equation*}
\lim_{j \rightarrow 0}\frac{1}{r_j^{m+2}}\int_{\bB_{r_j}}|u|^2 \leq C\lim_{j \rightarrow 0}\frac{1}{r_j^{m+2}}\int_{\bB_{r_j}}\mathcal{G}(u,c_jL)^2 +C\lim_{j \rightarrow 0}\alpha(L_j)^2=0.
\end{equation*}
This comes from the triangle inequality coupled with the fact that
\begin{equation*}
\frac{1}{r_j^{m+2}}\int_{\bB_{r_j}} |c_jL|^2 \leq C c_j^2\alpha(L)^2=\alpha(L_j)^2.
\end{equation*}
Thus we conclude with the proof of the lemma for the case of zero average.

We extend it for the case when $u$ is not average free. We denote $\eta \circ u$  the average of $u$. In the first case, we define
\begin{equation*}
L_j=c_j L\oplus Q\a{D(\eta \circ u)(0)}. 
\end{equation*}
In the second case, we define 
\begin{equation*}
    L_j=Q\a{D(\eta \circ u)(0)}.
\end{equation*}
We obtain the same properties for the average
\begin{equation*}
\int_{\bB_1}|D(\eta \circ u)(0)|^2 \leq \int_{\bB_1}|D(\eta \circ u)|^2
\end{equation*}
and
\begin{equation*}
\lim_{j \rightarrow \infty} \frac{1}{r_j^{m+2}}\int_{\bB_{r_j}}|\eta \circ u(x)-D(\eta \circ u)(0)x|^2dx=0.
\end{equation*}
We obtain the desired conclusion by summing the average free estimate with the estimate with the average.
\end{proof}
\begin{subsection}{Local Dir-minimizers away from the halfplane}
This subsection proves a useful result for later. Namely a function with zero boundary data on $\bB_1^+$ which is Dir-minimizing away from the spine is Dir-minimizing on $\bB_1^+$ with zero boundary value along $H$.
\begin{theorem}\label{t:localdirminimizers}
Let $u \in W^{1,2}(\bB_1^+,\mathcal{A}_Q(\RR^n))$ have zero boundary value along $H$.
Assume that $u$ is Dir-minimizing on every subdomain of $\bB_1^+$ with smooth boundary. Then $u$ is Dir-minimizing in $\bB_1^+$ with zero boundary value along $H$.
\end{theorem}

\begin{proof}
    Assume by contradiction $u$ is not Dir-minimizing in $\bB_1^+$. Then there exists $v$ with $u|_{\partial \bB_1^+}=v |_ {\partial \bB_1^+}$ and 
    \begin{equation*}
\int_{\bB_1^+} |Dv|^2< \int_{\bB_1^+}|Du|^2.
    \end{equation*}
We note that $v$ must also have zero boundary data along $H$.

If $\mathbf{\xi}$ is the bi-Lipschitz embedding from \cite{DS1} then the above inequality is equivalent to
 \begin{equation*}
 \int_{\bB_1^+}|D\mathbf{\xi}(v)|^2< \int_{\bB_1^+}|D\mathbf{\xi}(u)|^2.
     \end{equation*}
We interpolate between $u$ and $v$ so that they agree on a fattened neighborhood of the boundary. We consider $\alpha$ be a bump function increasing and smooth with $\alpha(t)=0, \forall t \in[0,1/2]$, $\alpha(t)=1$ for every $t \in[1,\infty)$. The bump function we care about will be $\phi_{\epsilon}(x)=\alpha(\dist(x,\partial \bB_1^+)/\epsilon)$. It is supported in 
\begin{equation*}
U_{\epsilon}:=\left\{x: \dist(x, \partial \bB_1^+)> \epsilon \right\}.
\end{equation*}

We define the embeddings $\overline{v}=\xi(v)$ and $\overline{u}=\xi(u)$.
The consider thus \begin{equation*}
    \overline{v}_{\epsilon}=\left( \Phi_{\epsilon}\mathbf{\xi}(v)+(1-\Phi_{\epsilon})\mathbf{\xi}(u)\right).
\end{equation*}
We define the interpolations $v_{\epsilon}$ as $v_{\epsilon}=\rho(\overline{v}_{\epsilon})$.

For every $\epsilon$, $v_{\epsilon}$ agrees with $u$ at $\bB_1^+ \setminus U_{\epsilon}$. We consider $V_{\epsilon}$ a subdomain of $\bB_1^+$ with smooth boundary such that $U_{\epsilon} \subset V_{\epsilon}$.

By the Dir-minimizing property on $V_{\epsilon}$
\begin{equation*}
\int_{\bB_1^+} |Du|^2 \leq \int_{\bB_1^+} |D\xi(v_{\epsilon})|^2.
\end{equation*}

We will show that 
\begin{equation*}
\int |D \overline{v}|^2= \lim_{\epsilon \rightarrow 0} \int|D \overline{v}_{\epsilon}|^2.
\end{equation*} This will allow us to conclude that $u$ is Dir-minimizing in $\bB_1^+$.

We compute
\begin{equation*}
D \mathbf{\xi}(v_{\epsilon})= D(\xi \circ \rho)(v_{\epsilon}) \cdot (\Phi_{\epsilon} D\xi(v)+ (1-\Phi_{\epsilon})D\xi(u)+ D\Phi_{\epsilon}(\xi(v)-\xi(u)).
\end{equation*}

We will show strong convergence in $L^2$ of the above $\epsilon \rightarrow 0$ to 
\begin{equation*}
D(\xi \circ \rho)(\xi(v)) \cdot D\xi(v).
\end{equation*}

It is enough to show that  $D(\xi \circ \rho)(v_{\epsilon})$ converges strongly in $L^2$ to $D(\xi \circ \rho)(\xi(v))$ and $\Phi_{\epsilon} D\xi(v)+ (1-\Phi_{\epsilon})D\xi(u)+ D\Phi_{\epsilon}(\xi(v)-\xi(u))$ converges weakly in $L^2$ to $D\xi(v)$.

For the first term we have
\begin{equation*}
\int_{\bB_1^+} |D (\mathbf{\xi} \circ \rho)(\overline{v}_{\epsilon}) - D (\mathbf{\xi} \circ \rho)(\overline{v})| ^2\lesssim |\left\{\overline{v}_{\epsilon} \neq \overline{v}\right\}|\end{equation*}.

In the above, we use that $\xi \circ \rho$ is bi-lipschitz. The $L^2$ convergence follows from construction since the right hand side of the inequality are neighborhoods of the boundary that converge to zero.

As for the weak $L^2$ convergence, it is an immediate consequence of the weak $L^2$ convergence of  $\Phi_{\epsilon}$ to $1$.

We conclude from the strong and weak convergence the Dir-minimizing property of $u$.
\end{proof}
\end{subsection}
\end{section}
\begin{section}{Proof of the excess decay lemma}
In this section we prove the excess decay lemma, introduced in Section 5, from which we deduced the main theorems in this paper.
\begin{lemma}[$N$th excess decay with small angle and small second fundamental forms] Let $T$ and $C$ be chosen as in Assumption \ref{a:currentandcone}. For every $\theta_N \in \RR^+$ there exists $\varepsilon_N \in \RR+$,  $\eta_N \in (0,1/2)$ such that if 
\begin{equation*}
\mathbb{E}(T,C,\mathbf{B}_1) \leq \varepsilon_N \alpha(C)^2 \: \: \textup{and} \: \: \mathbf{A}_{\Gamma}+\mathbf{A}_{\Sigma}^2 \leq \varepsilon_{N} \mathbb{E}(T,C,\mathbf{B}_1) 
\end{equation*}
(when $N=1$ the first inequality becomes $\mathbb{E}(T,C,\bB_1) \leq \varepsilon_1$)
there exists a radius $r_0 \in [\eta_{N},1/2]$ and an admissible open book $C'$ such 
that
\begin{equation*}
\mathcal{G}(C',C)^2 \leq \gamma(Q,m,n,\overline{n})\mathbb{E}(T,C,\bB_1)
\end{equation*}
and
\begin{equation*}
    \mathbb{E}(T,C',\bB_{r_0}) \leq \theta_N \min\left\{\alpha(C')^2,\mathbb{E}(T,C,\bB_1) \right\}.
\end{equation*}
 \end{lemma}

 We will prove the above lemma by contradiction. If the above lemma does not hold we will have to study the properties of certain blowup sequences. We rescale the ball $\bB_1$ to $\bB_8$ and we consider the following definition of a blowup sequence:

\begin{assumption}[Blowup Sequence]\label{a:blowupsequence}
A sequence of currents $T_k$ is said to be a blowup sequence if $T_k,\Gamma_k,\Sigma_k,C_k$ satisfy Assumption \ref{A:general} with $T_{0}(\Gamma_k)=\RR^{m-1}\times \left\{0 \right\}$, and $C_k=\sum_{i=1}^N Q_i\a{H_{i,k}}$ being open books with the same number of sheets and labeling. Moreover, the following limits hold:
\begin{equation*}
\frac{\mathbb{E}(T_k,C_k,\bB_8)}{\alpha(C_k)^2} \rightarrow 0\; \textup{and} \; \frac{\mathbf{A}_{\Gamma}+\mathbf{A}_{\Sigma}^2}{\mathbb{E}(T_k,C_k,\bB_8)} \rightarrow 0.
\end{equation*}
\end{assumption}
We wish to take a Dir-minimizing blowup out of a blowup sequence of currents. Up to a subsequence $H_{i,k} \rightarrow H_i$. Define $C:=\sum_{i=1}^N Q_i\a{H_i}$ then $C_k \rightarrow C$. We also have $T_k \rightarrow C$, $\Sigma_k \rightarrow \RR^{m+\overline{n}}$, $\Gamma_k \rightarrow \RR^{m-1} \times \left\{0\right\}$. 

If $k$ is large enough, $T_k$ satisfy the hypothesis \ref{A:Lipschitzapproximation}. The excess bound is obviously satisfied so it only remains to check the mass bound. This is a consequence of the monotonicity formula and the fact that $T_k  \rightarrow C$ with $C$ an open book of multiplicity $Q$.

We take the Lipschitz approximations $u_i^k$ defined on the outer regions, which are subdomains of $\left(\bB_1 \setminus B_{1/k}(V)\right)\cap H_{i,k} $. We do need them to be defined on $\bB_1 \setminus B_{1/k}(V) \cap H_i$. This is easy to achieve by applying a rotation we define for $i,k$ a rotation $\sigma_{i,k}$ on $\RR^{m+n}$ that takes $H_{i,k}$ to $H_i$, while fixing $\RR^{m-1}\times \left\{0\right\}$.  We choose $\sigma_{i,k}$ to be the admissible rotation closest to the identity. Clearly $\sigma_{i,k}\rightarrow_{k} \textup{Id}$. We define
\begin{equation*}
\bar{u}_i^k:=\frac{\mathbf{p}_{H_i^{\perp}} \circ \sigma_{i,k} \circ u_i^k\circ\sigma_{i,k}^{-1}}{\sqrt{\mathbb{E}(T_k,C_k,\bB_8)}}.
\end{equation*}

Up to a subsequence $\overline{u}_i^k$ converges weakly to a function $\overline{u}_i$. We will be able to estimate the functions $\overline{u}_i$ at the scale $r=\frac{1}{12\sqrt{m-1}}$ since we have the non-concentration estimate at scale $\frac{1}{6\sqrt{m-1}}$. The functions $\overline{u}_i$ will be the Dir-minimizing blowup of the blowup sequence. 

\begin{lemma}[Dir-minimizing blowup]
Let $r=\frac{1}{12\sqrt{m-1}}$. The following properties hold for the sequence $\overline{u}_i^k$ and their weak limit $\overline{u}_i$:
\begin{itemize}
\item The functions $\overline{u}_i^k$ have equibounded Dirichlet energy.
\item $\overline{u}_i^k$ converge strongly in $W^{1,2}(\bB_r^+,\mathcal{A}_{Q_i}(\RR^n))$ (up to a subsequence) to a function $\overline{u}_i$ in $W^{1,2}(\bB_r^+,\mathcal{A}_{Q_i}(\RR^n))$.
\item  The functions $\overline{u}_i$ are Dir-minimizing on $\bB_r^+$ with zero boundary value along $V$.
\end{itemize}
\end{lemma}
\begin{proof}
We will show in order the following facts about the functions $\overline{u}_i^k$:
\begin{enumerate}
\item $\overline{u}_i^k$ converge locally strongly in $W^{1,2}(\bB_r^+,A_{Q_i}(\RR^n))$ (up to a subsequence) to a function $\overline{u}_i \in W^{1,2}(\bB_r^{+},\mathcal{A}_Q(\RR^n))$.
\item  The function $\overline{u}_i$  is Dir-minimizer on subdomains of $\bB_r^+$ with smooth boundary.

\item $|\overline{u}_i| \in W^{1,2}(\bB_r^+,\RR)$ are sub-harmonic.

\item The functions $\overline{u}_i^k$ have equi-bounded Dirichlet Energy, i.e., 
\begin{equation*}
\int |\overline{u}_i^k|^2 \lesssim 1.
\end{equation*}
\item The functions $\overline{u}_i^k$ have zero boundary value at $V$ in the classical sense.

\item 
$\overline{u}_i^k$ converge strongly in $W^{1,2}(\bB_r^+,A_{Q_i}(\RR^n))$ (up to a subsequence) to $\overline{u}_i$.

\item $\overline{u}_i$ are Dir-minimizing in $\bB_{r}^+$. \end{enumerate}

(1) By Lemma \ref{p:blowup}, the functions $\overline{u}_i^k$ converge locally strongly in $W^{1,2}(\bB_r^+,\mathcal{A}_Q(\RR^n))$ to a function which is locally Dir-minimizing in certain regions.

(2) Given a subdomain of $\Omega \subset \bB_r^+$ with smooth boundary, the functions $\overline{u}_i^k$ converge strongly in $W^{1,2}(\bB_r^+,\mathcal{A}_{Q_i}(\RR^n))$ to $\overline{u}_i$ and $\overline{u}_i$ is Dir-minimizer in $\Omega$. This is a consequence of the proof of Theorem 5.2 on the first harmonic approximation \cite{DS3}. The theorem uses balls as domains but it can be easily modified to be any domain with smooth boundary. We remark that at this stage we have not yet claimed the functions $\overline{u}_i$ to be non trivial.

We also know by \cite{DS3}
 that 
\begin{equation}\label{e:BadSetDirEstimate}
\int_{L_i \setminus \overline{K_i}(L)} |D\overline{u}_i^k|^2 \rightarrow_{k \rightarrow \infty} 0.
\end{equation}

(3) Since Dir-Minimizers have sub-harmonic norm, the local Dir-Minimizer property implies that $|\overline{u}_i|$ are sub-harmonic functions.

(4)We have for every $L \subseteq R^{out}$ that
\begin{equation*}
2^{2L(L)}\left\|\overline{u}_i^k\right\|^2_{L^\infty(L)}+2^{ml(L)} \int_{L}|D\overline{u}_i^k|^2 \lesssim \frac{\mathbf{E}(T_k,C_k,L)}{\mathbb{E}(T_k,C_k,\bB_8)}+2^{-2l(L)}\frac{\mathbf{A}_{\Sigma_{k}}^2}{\mathbb{E}(T_k,C_k,\bB_8)}.
\end{equation*}

By taking the limit as $k \rightarrow \infty$, we get that for any $\alpha>0$ there exists $C_{\alpha}$ such that $|\overline{u}_i| \leq C_{\alpha} 2^{-(1-\alpha/2)\ell(L)}.$ This implies that $\overline{u}_i$ have zero boundary value along $V$ (and in the classical sense).

Given $0<\beta<1$ we have the weighted estimate 
\begin{align*}
\int_{\bB_r^+} \dist(x,V)^{-\beta}|D\overline{u}_{i}|^2  \leq C\sum_{L} 2^{\ell(L)\beta} \limsup_{k \rightarrow \infty}\int_{L} |D\overline{u}_i^k|^2 \\\leq C(\alpha,Q,m,n,\overline{n}) \sum_{l=0}^{\infty} \frac{1}{2^{-l(m-1)}}2^{-l(m-\alpha+\beta)}=\sum_{l=0}^{\infty}2^{-l(1-\alpha+\beta)}\lesssim_{\alpha,\beta}1.
\end{align*}
The second inequality follows from \ref{t:nonconcentration} coupled with properties of the cubical decomposition. By taking $1-\beta<\alpha<1$ we get the desired convergence.

(5) The strong convergence is a classical consequence of equiboundedness in energy.

(6) We proved in Theorem \ref{t:localdirminimizers} that, in this context, we are able to obtain the Dir-minimizing property of the function if the function is locally Dir-minimizing in smooth subdomains.
\end{proof}

\begin{proof}[Proof of Excess Decay \ref{lem:saexcessdecay}] We rescale the ball $\bB_1$ to $\bB_8$. If the excess decay lemma does not hold there is sequence of currents as in the Assumption \ref{a:currentandcone} that fails the conclusions of the lemma. We can assume up to rotation that they satisfy Assumption \ref{a:blowupsequence}.
 The sequence of currents can be chosen such that 
\begin{equation*}
    \mathbb{E}(T_k,C'_k, \bB_r) \geq \theta_N  \left\{\alpha(C')^2,\mathbb{E}(T,C,\bB_1) \right\}.
\end{equation*}
for every $r \in [1/k,4]$ and every cone $C'_k$ such that 
\begin{equation*}
\mathcal{G}(C'_k,C_k)^2 \leq \gamma(Q,m,n,\overline{n}) \mathbb{E}(T_k,C'_k,\bB_8).
\end{equation*}

We will show that either there exists a sequence of radii $r_j \rightarrow 0$  and a sequence of cones $C'_{k,j}$, which are obtained by adjusting the open books $C_k$ satisfying the constraint
\begin{equation*}
\mathcal{G}(C'_{k,j},C_k)^2 \leq \gamma(Q,m,n,\overline{n}) \mathbb{E}(T_k,C'_k,\bB_8)
\end{equation*}
such that 
\begin{equation*}
\lim_{j \rightarrow 0}\limsup_{k \rightarrow 0}\frac{\mathbb{E}(T_k,C'_{k,j},\bB_{r_j})}{\mathbb{E}(T_k,C_k,\bB_8)}=0
\end{equation*}
and 
\begin{equation*}
\lim_{j \rightarrow 0} \limsup_{k \rightarrow 0}\frac{\mathbb{E}(T_k,C'_{k,j},\bB_{r_j})}{\alpha(C'_{k,j})^2}=0.
\end{equation*}
which will be a contradiction.

\textbf{The construction of the cones $C'_{k,j}$}
We take any sequence $r_j \rightarrow 0$ and choose $L_{i,j}$ as in Lemma \ref{l:dirminmizersdecaylinearguy} so that
\begin{equation}\label{eq:lineardecay}
\lim_{j \rightarrow 0} \frac{1}{\min_{1 \leq i \leq N}\alpha(L_{i,j})^2}\sum_{i=1}^N\frac{1}{{(4r_j)}^{m+2}}\int_{\bB_{4r_j}}\mathcal{G}(\overline{u}_i,L_{i,j})^2=0
\end{equation}
where if the second case of Lemma \ref{l:dirminmizersdecaylinearguy} holds we define $L_{i,j}:=0$ in order to unify the notation. We can apply Lemma \ref{l:dirminmizersdecaylinearguy} so that the same sequence subsequence of $r_j$ works simultaneously for every $i$.

We define then the adjusted cone as
\begin{equation*}
C'_{k,j}:=\sum_{i} \a{\sigma_{i,k}^{-1}(H_k\oplus \mathbb{E}(T_k,C_k,\bB_8)^{1/2}L_{i,j})}.\end{equation*}

We have that 
\begin{equation}\label{eq:angleadjustedcone}
\mathcal{G}(C'_{k,j},C_k)^2 \leq \gamma(Q,m,n,\overline{n}) \mathbb{E}(T_k,C_k,\bB_8). 
\end{equation}
The above equation follows from
\begin{equation*}
    \int_{\bB_1}|DL_{i,j}|^2 \leq \int_{\bB_1}|D\overline{u}_i|^2 \leq C(Q,m,n,\overline{n}).
\end{equation*}
There are two possible cases, the case where the cones $C'_{k,j}$ have all $N$ collapsed sheets and the case where $C'_{k,j}$ has more than $N$ sheets (which happens when for some $i$ the first case of the conclusions of Lemma \ref{l:dirminmizersdecaylinearguy} hold).

In the second case we have that for $k$ large enough
\begin{equation} \label{eq:anglenewcone}
\alpha(C'_{k,j}) \simeq \min_{1 \leq i \leq N}\alpha(L_{i,j}) \mathbb{E}(T_k,C_k,\bB_8)^{1/2}.
\end{equation}

\textbf{Bounding $\mathbf{E}(T_k,C'_{k,j},\bB_r)$ and $\mathbb{E}(T_k,C_k,\bB_r)$ with respect to $\mathbb{E}(T_k,C_k,\bB_8)$}:

We remind that
$\mathbf{E}(T_k,C_k,\bB_4) \lesssim \mathbb{E}(T_k,C_k,\bB_8)$. This allows us to estimate everything in terms of $\mathbb{E}(T_k,C_k,\bB_8)$.

We estimate for $r \leq \frac{1}{6\sqrt{m-1}}$:
\begin{align*}
\mathbf{E}(T_k,C'_{k,j},\bB_r) \leq& C\mathbf{E}(T_k,C_k,\bB_r \cap B_{\sigma}(V))+ C \mathcal{G}(C'_{k,j},C_k)^2\sigma^2 +\mathbf{E}(T_k,C'_{k,j},\bB_r \setminus B_{\sigma}(V)) \\\leq & C\sigma^{3-\alpha}\left(\mathbf{E}(T_k,C_k,\bB_4)+\mathbf{A}_{\Gamma}+\mathbf{A}_{\Sigma}^2\right) + C \sigma^2 \mathbb{E}(T_k,C_k,\bB_8) \\&+\mathbf{E}(T_k,C'_{k,j},\bB_r \setminus B_{\sigma}(V)).
\end{align*}
The first inequality comes from using the triangle inequality on the $L^2$ excess in the region $\bB_r \cap B_{\sigma}(V).$ In the second inequality the first term we used the non concentration estimate Theorem \ref{t:nonconcentration}, the second term we used equation \eqref{eq:angleadjustedcone}. 

We wish to bound the third term. It is possible to do so by a modification of the argument on Theorem \ref{t:nonconcentration}. In the same fashion we use that the current in each Whitney region is agrees up to a superlinear error the graphical approximation. We get the following bound for $k$ large enough
\begin{align*}
\mathbf{E}(T_k,C'_{k,j},\bB_{4r} \setminus B_{\sigma}(V)) \leq  & \frac{1}{r^{m+2}}\sum_{i }\int_{\bB_{4r} \setminus B_{\sigma}(V)} \mathcal{G}(u_i^k,\mathbb{E}(T_k,C_k,\bB_8)^{1/2}L_{i,j})^2
\\&+C\left( \mathbf{E}(T_k,C_k,\bB_4)+r^2\mathbf{A}_{\Sigma}^2 \right)\left(\mathbf{E}(T_k,C_k,\bB_4)\sigma^{-m-2}+r^2\mathbf{A}_{\Sigma}^2\right)^{\gamma} 
\end{align*}
We take first lim sup over $k$ and then $\sigma \rightarrow 0$ to get
\begin{equation*}
\limsup_{k \rightarrow \infty} \frac{\mathbf{E}(T_k,C'_{k,j},\bB_r)}{\mathbb{E}(T_k,C_k,\bB_8)} \leq C \sum_{i=1}^N\frac{1}{r^{m+2}}\int_{\bB_r}\mathcal{G}(\overline{u}_i,L_{i,j})^2.
\end{equation*}
The last term will disappear when taking $k \rightarrow \infty$ since it is superlinear with respect to the excess.

Moreover, by our choice of $L_{i,j}$ we have that 
\begin{equation}\label{eq:decayL2relativetoangle}
\lim_{j \rightarrow 0}\limsup_{k \rightarrow \infty} \frac{\mathbf{E}(T_k,C'_{k,j},\bB_{4r_j})}{\alpha(C'_{k,j})^2}=0
\end{equation}
and 
\begin{equation*}
\lim_{j \rightarrow 0}\limsup_{k \rightarrow \infty} \frac{\mathbf{E}(T_k,C'_{k,j},\bB_{4r_j})}{\mathbb{E}(T_k,C_k,\bB_8)}=0.
\end{equation*}
In the case that $C'_{k,j}$ has more sheets than $C_k$, \eqref{eq:decayL2relativetoangle} follows from \eqref{eq:anglenewcone} and \eqref{eq:lineardecay}. 

In the case that $C'_{k,j}$ has the same number of sheets as $C_k$: we obtain
\begin{equation*}
\limsup_{k \rightarrow \infty}\frac{\mathbb{E}(T_k,C_k,\bB_8)}{\alpha(C_k)^2} \rightarrow 0,
\end{equation*}
\begin{equation*}
|\alpha(C_k)- \alpha(C'_{k,j})| \leq C(Q,m)\mathcal{G}(C'_{k,j},C_k).
\end{equation*}
We can conclude from the above and \eqref{eq:angleadjustedcone} that for every $j$
\begin{equation*}
\lim_{k \rightarrow 0}\frac{\alpha(C_k)}{\alpha(C'_{k,j})}=1.
\end{equation*}
This immediately implies \eqref{eq:decayL2relativetoangle}.

\textbf{Separation estimates in the first case at scales comparable to $r_j$:}

As a consequence of \eqref{eq:decayL2relativetoangle}, for $j$ large enough and $k$ large enough depending on $j$
we must have that the $L^2-L^{\infty}$ height bound applies at scale $4r_j$ and thus for some $\delta$ which can be taken arbitrarily small
\begin{equation*}
\textup{spt}(T_k) \cap \left( \bB_{3_{r_j}} \setminus \bB_{10^{-10} r_j} \right)\subset \left\{x \in \bB_{3r_j} \setminus \bB_{10^{-10}r_j}: \dist(x,C'_{k,j}) \leq \delta \alpha(C'_{k,j}).\right\}
\end{equation*}
This must also hold at the level of the Dir-minimizers. 
Given $\delta>0$ there if $j$ is large enough
\begin{equation*}
\frac{1}{r_j^{m+2}}\sum_{i=1}^N \int_{\bB_{4r_j}} \mathcal{G}(\overline{u}_i,L_{i,j})^2 \leq \delta^2 \min_i \alpha(L_{i,j})^2.
\end{equation*}

This implies that
\begin{equation*}
\sup_{x \in \bB_{3r_j} \setminus \bB_{10^{-10}r_j}} \mathcal{G}(\overline{u}_i,L_{i,j}) \leq C(Q,m,n,\overline{n}) \delta \min_i \alpha(L_{i,j})^2. 
\end{equation*}

The graph $\overline{u}_i$ must separate at scales comparable to $r_j$ if $j$ is taken to be small enough. This implies that for $j$ large enough the graph of $u_i^k$ must separate into disjoint neighborhoods of $C'_{k,j}$ at scales comparable to $r_j$.

\textbf{Bounding $\mathbb{E}(T_k,C'_{k,j},\bB_{r_j})$ with respect to $\mathbf{E}(T_k,C'_{k,j},\bB_{4r_j})$}

We wish to use Lemma \ref{l:reverseexcessbound} to prove the reverse excess bound. In order to do that we must show that all large enough Whitney cubes with respect to $C'_{k,j}$ inside $\bB_{4r_j}$ are outer cubes whose multiplicities are correct. We give a brief sketch of what is left to conclude that from the separation we already obtained the outer cube property.

Let $L$ be a cube of size comparable to $r_j$ inside $\bB_{4r_j}$. For a bump function $\varphi_{L}$ we must have that 
\begin{equation*}
d(\varphi_{L} |T_k|, \varphi_{L}\sum_{i,l}\mathbf{G}_{u_i^k}) \leq C\left( \mathbf{E}(T_k,C_k,\bB_4)+r_j^2\mathbf{A}_{\Sigma}^2 \right)\left(\mathbf{E}(T_k,C_k,\bB_4)r_j^{-m-2}+r_j^2\mathbf{A}_{\Sigma}^2\right)^{\gamma}.
\end{equation*}
Thus for $k$ large enough, the separation estimates allow us to write $u_i^k$ as a graph over the cone $C'_{k,j}$. We must have that, up to reparameterization $u_i^k$ is a strong Lipschitz approximation of $T_k$ on top of the cone $C'_{k,j}$ with the correct multiplicity arrangement. This implies the outer cube property and allows us to conclude.

This implies that 
\begin{equation*}
\lim_{j \rightarrow \infty}\limsup_{k \rightarrow \infty} \frac{\mathbb{E}(T_k,C'_{k,j},\bB_{r_j})}{\mathbf{E}(T_k,C'_{k,j},\bB_{4r_j})} \leq C(Q,m,n,\overline{n})
\end{equation*}
which concludes with the desired estimate on $C'_{k,j}$ and thus gives us the desired estimates:
\begin{equation*}
\lim_{j \rightarrow 0}\limsup_{k \rightarrow 0}\frac{\mathbb{E}(T_k,C'_{k,j},\bB_{r_j})}{\mathbb{E}(T_k,C_k,\bB_8)}=0
\end{equation*}
and 
\begin{equation*}
\lim_{j \rightarrow 0} \limsup_{k \rightarrow 0}\frac{\mathbb{E}(T_k,C'_{k,j},\bB_{r_j})}{\alpha(C'_{k,j})^2}=0.
\end{equation*}
\end{proof}
\end{section}
\begin{section}{A further study of the linear problem}
In this section, we provide refined statements for theorems applicable to the linear problem. We will omit the proof, as they follow the same strategy as the proofs already performed on the nonlinear problem. In order to reduce the nonlinear problem to the decay of the linear problem we only needed to study the Dir-minimizing functions $u \in W^{1,2}(\bB_1^+ , \mathcal{A}_Q(\RR^n))$ with zero boundary value along $\bB_1 \cap \left(\RR^{m-1}\times \left\{0 \right\}\right)$.

The methods used before allow us to treat a more general version of the linear problem which we remind.
\begin{assumption}
Given a Dir-minimizing function $u \in W^{1,2}(\Omega,\mathcal{A}_Q(\RR^n))$. The boundary linear problem consists of functions $u$ such that $u \res \partial \Omega \cap \bB_1=Q\a{0}$.
\end{assumption}
This is equivalent to understanding the case in which a $u \res \partial \Omega \cap \bB_1=Q\a{\varphi}$ with suitable smoothness assumptions on $\varphi$ since removing the average for $u$ makes it have zero boundary value.

The frequency function must be defined in a different way since we have a curved boundary.
If $d(x,q)$ is a good distance function like in \cite{DDHM} and 
\begin{equation*}
\varphi (t) :=
\left\{
\begin{array}{ll}
1, &\mbox{for $0\leq t \leq \frac{1}{2}$,}\\
2 (1-t), &\mbox{for $\frac{1}{2}\leq t \leq 1$,}\\
0, &\mbox{for $t\geq 1$.}
\end{array}\right.
\end{equation*}

We define 
\begin{equation*}
D(q,r):=\int \varphi \left(\frac{d(q,x)}{r}\right)|Du(x)|^2 dx
\end{equation*}
and 
\begin{equation*}
H(q,r):=-\int  \varphi \left(\frac{d(q,x)}{r}\right)|\nabla d(x,q)|^2 \frac{|u(x)|^2}{d(x,q)}dx.
\end{equation*}
We define 
\begin{equation*}
I(q,r):=\frac{rD(q,r)}{H(q,r)}.
\end{equation*}
\begin{theorem}[Almost monotonicity of the frequency function] Assume that $\mathbf{A}_{\partial \Omega}$ is small enough. Then, for every $q \in \partial \Omega \cap \bB_{1/2}$ $0<s<r<1/2$
\begin{equation}
I(q,r) \geq e^{-C\mathbf{A}_{\partial \Omega}{\left(r/s\right)}} I(q,s).
\end{equation}
In particular
\begin{equation*}
I(q):=\lim_{r \rightarrow 0} I(q,r)
\end{equation*}
exists.
\end{theorem}
In loose terms we get the following
\begin{theorem}
Assume that $\Omega$ is a $C^2$ domain and $u$ as in Assumption \ref{a:linearproblem}. 
\begin{itemize}
    \item For all $q \in \partial \Omega$ the frequency can be defined in a neighbourhood of $q$ and is almost monotone. This means, \begin{equation*}
        e^{C\mathbf{A}_{\partial \Omega}r}I(r)
    \end{equation*} is monotone in $r$ for  an appropriate constant $C$.
    \item The frequency at the limit is independent of the choice of a distance function and $\forall q \in \partial \Omega$, $I(q) \geq 1$. 
    \item Either $I(q)=1$ or $I(q) \geq 1+\alpha(Q,m,n)$ for a positive constant $\alpha(Q,m,n)$. If $I(q)=1$ then the blowups at $q$ are linear, of the form
   \begin{equation}
    f(x)=\sum_{i=1}^N Q_i\a{v_i(x\cdot \boldsymbol{\eta}_{\partial \Omega,q})}
    \end{equation}
    with $\boldsymbol{\eta}_{\partial \Omega,q}$ the normal to $\partial \Omega$ at $q$,
    and unique. The uniqueness is quantitative and guarantees the existence of a well-defined multivalued normal derivative at the boundary, which is Hölder continuous.
\end{itemize}
\end{theorem}
The theorem above is a consequence of an excess decay type theorem for the linear problem.
More precisely the excess decay theorem is the following
\begin{theorem}
Given a constant $M$, assume that $\int_{\bB_1} |Du|^2 \leq M$. Further assume that $0 \in \partial \Omega$. There exists $\varepsilon>0$ and $\kappa>0$ such that if $L: \RR^{m} \rightarrow \mathcal{A}_{Q}(\RR^n)$ is linear with $L \res T_{0}(\partial \Omega)=Q\a{0}$ and
    \begin{equation*}
    \max\left\{\int_{\bB_1} \mathcal{G}(u(p),L(p))^2 dp,\kappa^{-1}\mathbf{A}_{\partial \Omega}\right\} <\varepsilon
    \end{equation*}
    then there exists a linear map $\mathbf{L}$ with $L \res T_0 (\partial \Omega)=Q\a{0}$ such that
    \begin{equation*}
     \max\left\{\frac{1}{r^{m+2}}\int_{\bB_r} \mathcal{G}(u(p),\mathbf{L}(p))^2 dp,\kappa^{-1}\mathbf{A}_{\partial \Omega}r\right\} \leq r^{\alpha} \max\left\{\int_{\bB_1} \mathcal{G}(u(p),L(p))^2 dp,\kappa^{-1}\mathbf{A}_{\partial \Omega}\right\}.
     \end{equation*}
\end{theorem}
We are not going to give a detailed proof of these facts, since they are the direct analogue of the existing results and the proofs are in fact even simpler. 

We do a brief discussion of the structure of the proof in the particular case of the linear problem on the half-space and comment on some of the simplifications that occur in this setting.
\begin{enumerate}
\item
As we observed already, the lower bound of the frequency by $1$ gives monotonicity of the normalized Dirichlet energy
\begin{equation*}
\frac{1}{r^m}\int_{\bB_r(x)^+}|Du|^2.
\end{equation*}
This plays the role of the monotonicity formula for the area. If we considered a more general domain instead, the same monotonicity formula holds up to an error.
\item In this setting we do also need to have a family of excess decay theorems depending on the number of "sheets" of $L$, where with the number of sheets of $L$ we mean $N$ such that
\begin{equation*}
    L=\sum_{i=1}^N Q_i \a{x_nv_i}
\end{equation*} where $v_i \in \RR^n$ are distinct and $Q_i$ are positive integers with $\sum Q_i=Q.$
\item 
There is of course an $L^2-L^{\infty}$ height bound in this setting as in \cite{de2023fineIII} (in fact the one we used in the nonlinear setting is a consequence of the one for the linear problem). Thus if $L=\sum_{i=1}^N Q_i \a{x_nv_i}$ there exists $\varepsilon_k$ such that if
\begin{equation*}
\int_{\bB_4}\mathcal{G}(u(p),L(p))^2 dp + \mathbf{A}_{\partial \Omega}<\varepsilon_k \alpha(L)^2
\end{equation*}
then
\begin{equation*}
\max_{p \in \bB_{2}^+ \setminus B_{1/32}(V)} \mathcal{G}(u(p),L(p)) << \alpha(L).
\end{equation*}
This implies that the function $u$ decomposes in $\bB_{2}^+ \setminus B_{1/32}(V)$ as a sum $u=\sum_{i=1}^k \left(u_i+x_nv_i \right)$ with $u_i \in W^{1,2}(\bB_{2} \setminus B_{1/32}(V),\mathcal{A}_{Q_i}(\RR^n))$. 

The $u_i$ are Dir-minimizing and "parameterized" on top of $\a{x_nv_i}$. We perform the same Whitney decomposition into regions satisfying the same estimates established in Section 6.

\item The Simon's non-concentration estimates still hold in this setting and allows us to get energy convergence in a contradiction limit to the excess decay lemma. We consider a contradiction sequence to the excess decay theorem and obtain, Dir-minimizing limits $\overline{u}_i \in W^{1,2}(\bB_1^+,\mathcal{A}_{Q_i})$ with zero boundary value. 
\item We can conclude by, if necessary, adjusting $L$ at suitable scales (as in Lemma  \ref{l:dirminmizersdecaylinearguy}) so $Q_i\a{x_nv_i}$ is replaced by the linear function
$Q_i\a{x_nv_i}+\mathbf{E}(u_k,L)^{1/2}\overline{L}_{i,j}$ where we denote $\mathbf{E}(u,L):=\int_{\bB_1} \mathcal{G}(u(p),L(p))^2dp.$ 
\end{enumerate}
The strategy described above was carried out for the interior in \cite{krummel2017fine} to prove $\HH^{m-2}$ almost everywhere uniqueness of the blowup. The advantage we have in this setting is that in the same way that for the nonlinear problem we obtain a no holes condition since $\Theta(T,q) \geq Q/2$ for every $q \in \Gamma$, we have that $I(q) \geq 1$ for every $q \in \partial \Omega$.

\end{section}
\bibliographystyle{alpha}

\addtocontents{toc}{\protect\enlargethispage{\baselineskip}}
\bibliography{UniquenessArxiv}

\begin{thebibliography}{DDHM23}

\bibitem[All69]{AllPhD}
William~K. Allard.
\newblock {On boundary regularity for {P}lateau's problem}.
\newblock {\em Bull. Amer. Math. Soc.}, 75:522--523, 1969.

\bibitem[All75]{AllB}
W.~K. Allard.
\newblock {On the first variation of a varifold: boundary behavior}.
\newblock {\em Ann. of Math. (2)}, 101:418--446, 1975.

\bibitem[Alm00]{Alm}
Jr. F.~J. Almgren.
\newblock {\em {Almgren's big regularity paper}}, volume~1 of {\em {World
  Scientific Monograph Series in Mathematics}}.
\newblock World Scientific Publishing Co. Inc., River Edge, NJ, 2000.

\bibitem[Bou16]{Theodora}
T.~Bourni.
\newblock Allard-type boundary regularity for $c^{1,\alpha}$ boundaries, 2016.

\bibitem[Bro86]{GMT_prob}
J.~E. Brothers.
\newblock Some open problems in geometric measure theory and its applications
  suggested by participants of the 1984 {AMS} summer institute.
\newblock In J.~E. Brothers, editor, {\em Geometric measure theory and the
  calculus of variations ({A}rcata, {C}alif., 1984)}, volume~44 of {\em Proc.
  Sympos. Pure Math.}, pages 441--464. Amer. Math. Soc., Providence, RI, 1986.

\bibitem[DDHM23]{DDHM}
C.~{De Lellis}, G.~{De Philippis}, J.~{Hirsch}, and A.~{Massaccesi}.
\newblock {On the boundary behavior of mass-minimizing integral currents}.
\newblock {\em Memoirs of the American Mathematical Society}, 291, 166pp, 2023.

\bibitem[DL15]{delellis2015size}
C.~De~Lellis.
\newblock The size of the singular set of area-minimizing currents, 2015.

\bibitem[DL22]{de2022regularity}
C.~De~Lellis.
\newblock The regularity theory for the area functional (in geometric measure
  theory).
\newblock In {\em International Congress of Mathematicians}, 2022.

\bibitem[DLMS23]{de2023fineIII}
C.~De~Lellis, P.~Minter, and A.~Skorobogatova.
\newblock The fine structure of the singular set of area-minimizing integral
  currents iii: Frequency 1 flat singular points and $\mathcal{H}^{m-2}$-ae
  uniqueness of tangent cones.
\newblock {\em arXiv preprint arXiv:2304.11553}, 2023.

\bibitem[DLMS24]{DMSModp}
C.~De~Lellis, P.~Minter, and A.~Skorobogatova.
\newblock Fine structure of singularities in area-minimizing currents mod(q).
\newblock {\em arXiv preprint arXiv:2403.15889}, 2024.

\bibitem[DLNS23]{delellis2021uniqueness}
C.~De~Lellis, S.~Nardulli, and S.~Steinbr{\"u}chel.
\newblock Uniqueness of boundary tangent cones for 2-dimensional
  area-minimizing currents.
\newblock {\em Nonlinear Analysis}, 230:113235, 2023.

\bibitem[DLNS24]{delellis2021allardtype}
C.~De~Lellis, S.~Nardulli, and S.~Steinbr{\"u}chel.
\newblock An allard-type boundary regularity theorem for $2 d $ minimizing
  currents at smooth curves with arbitrary multiplicity.
\newblock {\em Publications math{\'e}matiques de l'IH{\'E}S}, pages 1--118,
  2024.

\bibitem[DS11]{DS1}
C.~{De Lellis} and E.~Spadaro.
\newblock {{$Q$}-valued functions revisited}.
\newblock {\em Mem. Amer. Math. Soc.}, 211(991):vi+79, 2011.

\bibitem[DS14]{DS3}
C.~{De Lellis} and E.~Spadaro.
\newblock {Regularity of area minimizing currents {I}: gradient {$L^p$}
  estimates}.
\newblock {\em Geom. Funct. Anal.}, 24(6):1831--1884, 2014.

\bibitem[DS16a]{DS4}
C.~{De Lellis} and E.~Spadaro.
\newblock {Regularity of area minimizing currents {II}: center manifold}.
\newblock {\em Ann. of Math. (2)}, 183(2):499--575, 2016.

\bibitem[DS16b]{DS5}
C.~{De Lellis} and E.~Spadaro.
\newblock {Regularity of area minimizing currents {III}: blow-up}.
\newblock {\em Ann. of Math. (2)}, 183(2):577--617, 2016.

\bibitem[Fed69]{Fed}
H.~Federer.
\newblock {\em {Geometric measure theory}}.
\newblock {Die Grundlehren der mathematischen Wissenschaften, Band 153}.
  Springer-Verlag New York Inc., New York, 1969.

\bibitem[Fle24]{ian2024allard}
I.~Fleschler.
\newblock Allard-type regularity theory for area minimizing currents at
  boundaries with arbitrary multiplicity.
\newblock {\em arXiv preprint arXiv:2409.00820}, 2024.

\bibitem[Fle25]{ian2024example}
Ian Fleschler.
\newblock An essential one sided boundary singularity for a $3$-dimensional
  area minimizing current in $\mathbb{R}^5$.
\newblock {\em arXiv preprint arXiv:2502.15047}, 2025.

\bibitem[FR24]{ianreinaldo2024regularity}
I.~Fleschler and R.~Resende.
\newblock On the regularity of area minimizing currents at boundaries with
  arbitrary multiplicity.
\newblock {\em preprint on ArXiv}, 2024.

\bibitem[Hir16]{Jonas3}
J.~Hirsch.
\newblock {Boundary regularity of {D}irichlet minimizing {$Q$}-valued
  functions}.
\newblock {\em Ann. Sc. Norm. Super. Pisa Cl. Sci. (5)}, 16(4):1353--1407,
  2016.

\bibitem[HM19]{hirsch2019uniqueness}
J.~Hirsch and M.~Marini.
\newblock Uniqueness of tangent cones to boundary points of two-dimensional
  almost-minimizing currents.
\newblock {\em arXiv preprint arXiv:1909.13383}, 2019.

\bibitem[KW17]{krummel2017fine}
B.~Krummel and N.~Wickramasekera.
\newblock Fine properties of branch point singularities: Dirichlet energy
  minimizing multi-valued functions.
\newblock {\em arXiv preprint arXiv:1711.06222}, 2017.

\bibitem[KW23a]{krummel2023analysisI}
B.~Krummel and N.~Wickramasekera.
\newblock Analysis of singularities of area minimizing currents: a uniforn
  height bound, estimates away from branch points of rapid decay, and
  uniqueness of tangent cones.
\newblock {\em arXiv preprint arXiv:2304.10272}, 2023.

\bibitem[KW23b]{krummel2023analysisII}
B.~Krummel and N.~Wickramasekera.
\newblock Analysis of singularities of area minimizing currents: planar
  frequency, branch points of rapid decay, and weak locally uniform
  approximation.
\newblock {\em arXiv preprint arXiv:2304.10653}, 2023.

\bibitem[Min24]{Minter2024structure}
P.~Minter.
\newblock The structure of stable codimension one integral varifolds near
  classical cones of density q+ 1/2.
\newblock {\em Calculus of Variations and Partial Differential Equations},
  63(1):5, 2024.

\bibitem[Sim93]{simon1993cylindrical}
L.~Simon.
\newblock Cylindrical tangent cones and the singular set of minimal
  submanifolds.
\newblock {\em Journal of Differential Geometry}, 38(3):585--652, 1993.

\bibitem[Sim14]{simon2014introduction}
L.~Simon.
\newblock Introduction to geometric measure theory.
\newblock {\em Tsinghua Lectures}, 2014.

\bibitem[Whi83]{White1983tangent}
B.~White.
\newblock Tangent cones to two-dimensional area-minimizing integral currents
  are unique.
\newblock {\em Duke Math. J.}, 50(1):143--160, 1983.

\bibitem[Whi97]{White97}
B.~White.
\newblock {Stratification of minimal surfaces, mean curvature flows, and
  harmonic maps}.
\newblock {\em J. Reine Angew. Math.}, 488:1--35, 1997.

\bibitem[Wic14]{WickramasekeraSt}
N.~Wickramasekera.
\newblock A general regularity theory for stable codimension 1 integral
  varifolds.
\newblock {\em Annals of mathematics}, pages 843--1007, 2014.

\end{thebibliography}
\end{document}